\documentclass{amsart}

\usepackage{amsmath}
\usepackage{amssymb}
\usepackage{hyperref}
\usepackage{enumitem}

\newtheorem{theorem}{Theorem}[section]
\newtheorem{lemma}[theorem]{Lemma}
\newtheorem{corollary}[theorem]{Corollary}
\newtheorem{proposition}[theorem]{Proposition}

\theoremstyle{definition}
\newtheorem{definition}[theorem]{Definition}
\newtheorem{convention}[theorem]{Convention}

\newtheorem{remark}[theorem]{Remark}

\newcommand\R{\mathbb{R}}
\newcommand\Z{\mathbb{Z}}
\newcommand\T{\mathbb{T}}
\newcommand\C{\mathbb{C}}
\newcommand\N{\mathbb{N}}

\newcommand{\qtq}[1]{\quad\text{#1}\quad}

\newcommand\eps{\varepsilon}
\newcommand{\vk}{\varkappa}

\newcommand{\bigO}{\mathcal O}
\newcommand{\h}{\textup{\textsf{H}}}
\newcommand{\hbo}{H_{\text{\textup{BO}}}}
\newcommand{\hk}{H_\kappa}
\newcommand{\BA}{B^s_{\negmedspace A}} %other choices: \! (a.k.a. \negthinspace) and \negthickspace
\newcommand{\tint}{{{\textstyle\int} q}}
\newcommand{\btint}{\bigl[\tint\bigr]}

\newcommand{\sess}{\sigma_{\mkern-1mu\text{\upshape ess}}}

\let\Re=\undefined\DeclareMathOperator{\Re}{Re}
\let\Im=\undefined\DeclareMathOperator{\Im}{Im}
\DeclareMathOperator{\Id}{Id}

\DeclareMathOperator{\tr}{tr}
\DeclareMathOperator{\sgn}{sgn}
\newcommand{\CofP}{\textsl{CofP}}
\newcommand{\CofE}{\textsl{CofE}}
\newcommand{\CofB}{\textsl{Cof$\beta$}}
\newcommand{\VofP}{\textsl{VofP}}

\newcommand{\mcPbk}{\mc P^\beta_\kappa}
\newcommand{\olN}{{\,\ol{\! N}}}
\newcommand{\olM}{{\,\ol{\! M}}}

\newcommand{\del}{\delta}
\newcommand{\mc}[1]{\mathcal{#1}}
\newcommand{\ol}[1]{\overline{#1}}
\newcommand{\wh}[1]{\widehat{#1}}
\newcommand{\wt}[1]{\widetilde{#1}}
\newcommand{\norm}[1]{\left\lVert#1\right\rVert}
\newcommand{\snorm}[1]{\lVert#1\rVert}
\newcommand{\bnorm}[1]{\big\lVert#1\big\rVert}

\numberwithin{equation}{section}

\allowdisplaybreaks

\begin{document}

\title[Sharp well-posedness for the Benjamin--Ono equation]{Sharp well-posedness for\\the Benjamin--Ono equation}

\author[R.~Killip]{Rowan Killip}
\address{Department of Mathematics, University of California, Los Angeles, CA 90095, USA}
\email{killip@math.ucla.edu}

\author[T.~Laurens]{Thierry Laurens}
\address{Department of Mathematics, University of California, Los Angeles, CA 90095, USA}
\email{laurenst@math.ucla.edu}

\author[M.~Vi\c{s}an]{Monica Vi\c{s}an}
\address{Department of Mathematics, University of California, Los Angeles, CA 90095, USA}
\email{visan@math.ucla.edu}

\begin{abstract}
The Benjamin--Ono equation is shown to be well-posed, both on the line and on the circle, in the Sobolev spaces $H^s$ for $s>-\tfrac12$.  The proof rests on a new gauge transformation and benefits from our introduction of a modified Lax pair representation of the full hierarchy.  As we will show, these developments yield important additional dividends beyond well-posedness, including (i) the unification of the diverse approaches to polynomial conservation laws; (ii) a generalization of G\'erard's explicit formula to the full hierarchy; and (iii) new virial-type identities covering all equations in the hierarchy.
\end{abstract}

\maketitle

\tableofcontents

%%%%%%%%%%%%%%%%%%%%%%%%%%%%%%%%%%%%%%%%%%%%%%%%%%%%%%%%%%%%%%%%%%%%%%%%%%
\section{Introduction} \label{s:intro}
%%%%%%%%%%%%%%%%%%%%%%%%%%%%%%%%%%%%%%%%%%%%%%%%%%%%%%%%%%%%%%%%%%%%%%%%%%

This paper is devoted to the study of real-valued solutions to the Benjamin--Ono equation
\begin{equation}\tag{BO}\label{BO} 
    \tfrac{d}{dt} q =  \h q'' - 2qq' ,
\end{equation}
which describe the motion of internal waves in stratified fluids of great total depth.  The symbol $\h$ appearing here denotes the Hilbert transform; see \eqref{HT}.

This model arose contemporaneously in works by Benjamin~\cite{Benjamin1967} and by Davis--Acrivos \cite{Davis1967}.  The latter authors also performed extensive experiments in a tank using fresh water floating on an equal volume of salt water.  They observed excellent agreement.

Inspired by these works, Ono undertook a series of investigations of \eqref{BO}, beginning with \cite{Ono1975}, which in turn generated considerable interest in this model.  Among his contributions was the suggestion that the ease with which Davis and Acrivos were able to generate solitary wave solutions in the tank may be taken as a sign that these were, in fact, soliton solutions of the type then only recently discovered in the context of the Korteweg--de Vries equation.

With our sign conventions, these solitary waves take the form 
\begin{equation}\label{soliton Qc}
Q_c(t,x) = \frac{2c}{c^2 (x-ct)^2 + 1} \qtq{with} c>0.
\end{equation}
They are positive and travel to the right. By comparison, solutions of the linearized equation $\frac{d}{dt} q = \h q''$ travel to the left.

An instantly striking feature of the functions $Q_c$ is their mere algebraic decay.  This is ultimately traceable to the presence of the Hilbert transform in \eqref{BO}, which in turn expresses the highly nonlocal nature of the wave dynamics.  This nonlocality originates from the depth of the fluid; in shallow water, both surface and internal waves have been successfully modeled by local equations such as KdV.

We will be studying the initial-value problem for \eqref{BO} posed both on the real line $\R$ and on circle $\T=\R/\Z$, the latter being equivalent to the study of an initially periodic excitation.  A key determiner of which classes of initial data may be expected to lead to well-behaved solutions is the scaling symmetry.  For the \eqref{BO} equation, this takes the form
\begin{align}\label{scaling}
q(t,x) \mapsto q_\lambda(t,x) = \lambda q(\lambda^2 t, \lambda x) \qquad\text{for $\lambda>0$}
\end{align}
and identifies $s=-\frac12$ as the scaling-critical regularity for $H^s$ spaces.

Conservation laws also play a major role in identifying natural classes of initial data and in demonstrating that such classes are dynamically invariant.  Basic physical considerations already present us with three such conserved quantities: the momentum and the energy are given by
\begin{equation}\label{P&H}
P(q) = \int \tfrac{1}{2}q^2\, dx \qtq{and}
\hbo(q) = \int   {\tfrac{1}{2} q\h q' } - \tfrac{1}{3} q^3 \, dx,
\end{equation}
while $\int q$ denotes the surplus of water relative to equilibrium $q\equiv 0$.

The energy functional $\hbo$ serves as the Hamiltonian for \eqref{BO} with respect to the Poisson structure
\begin{equation}\label{Poisson}
\{ F,G \} = \int \tfrac{\del F}{\del q}(x) \cdot \big( \tfrac{\del G}{\del q} \big)'(x)\, dx ,
\end{equation}
while the momentum functional $P$ generates translations.

By comparison, $\int q$ is a Casimir.  This functional will play a limited role in what follows.  In the circle case, it can be altered by redefining the notion of equilibrium depth.  Physically, this amounts to exploiting the Galilei symmetry of \eqref{BO}: if $\widetilde q(t,x)$ is a solution, then so to is
\begin{equation}\label{E:Galilei}
q(t,x) = \widetilde  q(t,x-2ct) + c.
\end{equation}

In the line setting, one cannot use the Galilei transform to force $\int q=0$.  More significantly from our point of view, is the fact that one needs to impose rather strong decay assumptions in order to make sense of this quantity.   Moreover, regularity hypotheses must also be imposed to ensure that any such $L^1$ assumption is not immediately destroyed by wave dispersion.

Our first main result is the well-posedness of the \eqref{BO} flow under minimal assumptions on the initial data.  As we will see, this has been a much studied problem and our resolution depends not only on the recently introduced method of commuting flows, but also on the development of broader algebraic and analytic structures underlying the \eqref{BO} equation.  In subsection~\ref{ss:applic}, we will discuss several other dividends of these developments,  not directly related to well-posedness.

\begin{theorem}\label{t:main}
Fix $s > -\frac{1}{2}$.  The equation~\eqref{BO} is globally well-posed for initial data in $H^s(\R)$ or $H^s(\T)$.
%: The data-to-solution map $\Phi$ extends uniquely from $H^\infty$ to a jointly continuous map $\Phi : \R\times H^s \to H^s$.
\end{theorem}

As we will discuss more fully below, the long-standing record on the line was well-posedness for $s\geq 0$.  This was also the threshold for the circle case until the very recent breakthrough \cite{Gerard2020}, which proved well-posedness for all $s>-\frac12$.   The paper \cite{Gerard2020} also shows ill-posedness in $H^{-1/2}(\T)$ via instantaneous norm inflation.  A simple argument showing the breakdown of well-posedness for $s<-\tfrac12$ was known much earlier \cite{AnguloPava2010,MR1837253}.  In the line case, ill-posedness for $s<-\frac12$ can be deduced from the fact that the solutions \eqref{soliton Qc} converge in $H^s(\R)$ to a delta function at $t=0$ as $c\to\infty$, but do not converge at any other time.

\subsection{Prior work on well-posedness}\label{ss:prior}
Here we give a quick overview of the history of well-posedness for \eqref{BO}; for a comprehensive account, we recommend the recent book \cite{MR4400881}.  The first phase in these developments was the construction of weak solutions; see, for example, \cite{Ginibre1989a,Ginibre1989,Ginibre1991,Saut1979}.

Early proofs of well-posedness employed energy/uniqueness arguments; see, for example, \cite{Abdelouhab1989,Iorio1986,Ponce1991,Saut1979}.  Included in \cite{Abdelouhab1989} is a proof that \eqref{BO} is well-posed in $H^\infty$ in both geometries.  This period culminated in the proof that \eqref{BO} is well-posed in $H^s$ for $s>\frac32$ on $\T$ and for $s\geq \frac32$ on $\R$.  The endpoint in the line setting was achieved in \cite{Ponce1991} by incorporating local smoothing into the traditional Gronwall argument.

A striking feature of \eqref{BO} is that there was no subsequent Strichartz revolution, nor did the development of $X^{s,b}$ analysis immediately transform the study of \eqref{BO}.   There is a fundamental reason for this: \eqref{BO} is not analytically well-posed in any $H^s(\R)$ space!  This was first demonstrated in \cite{Molinet2001}, which proved that the data-to-solution map is not $C^2$.  Later in \cite{Koch2005} it was shown that for $s\geq 0$, this map is not even uniformly continuous in any neighborhood of the origin.

By their very nature, proofs by contraction mapping yield a data-to-solution map that is real-analytic.  The results discussed in the previous paragraph show that \eqref{BO} cannot be solved by this method, no matter what auxiliary norms are introduced, nor what ingenious estimates one proves.

By incorporating Strichartz control into energy methods, \cite{Koch2003} advanced well-posedness on the line to $s> \frac54$.  Further refinements of this style of argument in \cite{Kenig2003} led to well-posedness for $s>\frac98$.

The well-posedness theory for \eqref{BO} was much transformed by the paper \cite{Tao2004} which treated data in $H^1(\R)$. The transformative new idea here was the introduction of a gauge (a change of unknown) that substantially ameliorated the troublesome high-low frequency interaction responsible for the poor behavior of the data-to-solution map just discussed. The motivation for this gauge transformation is described in \cite[\S4.4]{MR2233925}, including parallels with the Cole--Hopf transformation.  Attention is also drawn to an analogue for the derivative nonlinear Schr\"odinger equation (cf. \cite{MR700302}).

By exploiting Tao's gauge transformation, well-posedness in $H^1(\T)$ was subsequently shown in \cite{Molinet2009}.  Well-posedness in $H^1$ is automatically global due to the conservation of
\begin{equation}\label{H2 defn}
H_2 := \int \tfrac12 \big[q'\bigr]^2  - \tfrac34 q^2 \h q' +  \tfrac14 q^4 \,dx .
\end{equation}

Tao's gauge transformation lead to a flurry of progress on the well-posedness problem, including \cite{Burq2006} which treated  $s>\frac{1}{4}$ on $\R$ and \cite{Molinet2007} which treated $s\geq \frac{1}{2}$ on~$\T$.  Evidently, both yield well-posedness for finite energy initial data.

As noted earlier, the long-standing record for \eqref{BO} on the line was well-posedness in $L^2(\R)$.  This was proved in \cite{Ionescu2007} via a synthesis of Tao's gauge transformation and $X^{s,b}$ techniques.  Well-posedness in $L^2(\T)$ was proved in \cite{Molinet2008} via consonant methods.

Well-posedness in \cite{Ionescu2007} means that the data-to-solution map admits a unique continuous extension from smooth initial data to a mapping from $H^s$ to $C_t H^s$.   This is also the meaning of Theorem~\ref{t:main}.

The landmark papers \cite{Ionescu2007,Molinet2008} stubbornly resisted improvement for a long period.  The topic of well-posedness in $L^2$ has been revisited several times via a variety of methods without yielding any improvement on the $H^s$ scale; see, \cite{Ifrim2019,Molinet2012,Talbut2021}.

Gibbs-distributed initial data on the circle (with momentum cutoff) lies right at cusp of the $L^2$ theory.  The existence of solutions and  preservation of this law was shown in \cite{MR3346690}.   Although the subsequent work \cite{Gerard2020} proves that Gibbs initial data leads to global solutions, it is unclear to us how readily this approach leads to invariance of the Gibbs law.  By comparison, the manner in which we prove Theorem~\ref{t:main} is well-suited to this problem.  The proof of \cite[Th.~3.4]{MR4145790} demonstrates how the method of commuting flows blends seamlessly with invariance of measure arguments in finite volume.

On the circle, the question of well-posedness in $H^s$ spaces was recently completely resolved in \cite{Gerard2020}, namely, the equation is well-posed for $s>-\frac12$ and ill-posed otherwise.  This is achieved through the construction of a Birkhoff normal form transformation developed in a series papers; see, for example, \cite{MR4275336,MR4155287}.  This approach is reminiscent of the earlier breakthrough \cite{MR2267286} for the KdV equation; however, the Lax operator \eqref{L defn} associated to \eqref{BO} is of an unconventional type, especially when compared to the much-studied Sturm--Liouville operators associated with KdV.

The direct analogue of such an approach to Theorem~\ref{t:main} on the line would be via inverse scattering, which is currently utterly untenable.  The only complete theory of both forward and inverse scattering is that of \cite{MR1073870}.  This requires weighted $L^1$ hypotheses that are incompatible with the soliton solutions \eqref{soliton Qc}, as well as a small data hypothesis.  The state of the art for the forward scattering problem is presented in \cite{Wu2017}, which requires $\langle x\rangle^\alpha q\in L^2$ for $\alpha >\frac12$.  Much remains to be done to advance the inverse scattering theory up to this threshold.

Our pessimism regarding an inverse scattering approach to Theorem~\ref{t:main} is also informed by the state of the art regarding the inverse scattering problem for the Schr\"odinger equation, which has been intensively studied for generations.  This is what is relevant to the KdV equation.   At this moment, strong spatial decay assumptions are required, which then beget regularity hypotheses (to preserve such decay at later times).  For a discussion of the significant hurdles associated with this approach already in the KdV setting, see, for example, \cite{Killip2019}.   Later in the introduction we will draw attention to some interesting questions in the spectral theory of the Lax operator $\mc L$ for \eqref{BO} that arise naturally from this perspective.

In this paper, we will approach the well-posedness problem via the method of commuting flows introduced in \cite{Killip2019} and developed in several subsequent papers \cite{MR4304314,harropgriffiths2022global,harropgriffiths2020sharp,killip2021wellposedness,MR4356987,MR4541923,MR4520307}.  This strategy was previously employed in \cite{Talbut2021}; however, the culmination of Talbut's work was well-posedness in $L^2$, both on the line and on the circle.  It will take us some time to explain the obstacles that lay in Talbut's path and how we are able to overcome them.

\subsection{The Lax structure}\label{ss:Lax}
A Lax-pair representation of \eqref{BO} appeared first in \cite{Nakamura1979} and then more directly in \cite{Bock1979}.  Our presentation here is also influenced by \cite{MR3484397}, where it is shown that any negative eigenvalues of $\mc L$ are necessarily simple.

Both operators of the Lax pair act on the Hardy space $L^2_+$ comprised of those functions in $L^2$ whose Fourier transform is supported on $[0,\infty)$.  Such functions may also be viewed as the boundary values of certain holomorphic functions in the upper half-plane or disk, depending on the geometry.  We avoid the more popular $H^p$ notation for the Hardy spaces because it collides with our notations for Sobolev spaces, Hamiltonians, and for the Hilbert transform $\h$.  We will write $C_\pm$ for the Cauchy--Szeg\H{o} projections; see \eqref{C defn}. 

In Proposition~\ref{P:Lax} we will show that the formal expression
\begin{equation}\label{L defn}
\mc L f = -i f' - C_+\bigl( q f\bigr)
\end{equation}
defines a semi-bounded selfadjoint operator $\mc L$ on $L^2_+$ for every $q\in H^s$ with $s>-\tfrac12$.  Its companion in the Lax pair is variously given as
\begin{equation}\label{P defn}
\mc P := - i\partial^2 - 2 \partial C_+q + 2 q'_+ \qtq{or} 
	\mc P -  i\mc{L}^2 = i C_+(\h q') - iC_+qC_+q .
\end{equation}
Following \cite{MR3484397}, we will insist on the former; the latter is the original one from \cite{Bock1979,Nakamura1979}.  These operators are transparently anti-selfadjoint when $q\in H^\infty$ and we shall not need to make sense of them for more irregular functions $q$.

Earlier, we promised to draw attention to some basic questions in the spectral theory of $\mc L$ that we regard as both intrinsically interesting and crucial milestones toward understanding inverse scattering for slowly decreasing initial data on the line.  Specifically, we ask what is the decay threshold for $q$,  expressed via power-law and/or $L^p$ integrability exponent, at which each of the following spectral transitions takes place:
\begin{itemize}
\item The appearance of embedded eigenvalues;
\item The appearance of embedded singular-continuous spectrum;
\item The disappearance of absolutely continuous spectrum. 
\end{itemize}
Note that for any $q\in L^p(\R)$, $p<\infty$, Weyl's Theorem guarantees that the essential spectrum of $\mc L$ fills $[0,\infty)$.  Our questions seek to clarify the spectral type.  The only progress on these problems of which we are aware is the paper \cite{Sun2021}, which shows absence of embedded eigenvalues when $\langle x\rangle q \in L^2$.   For a discussion of these problems in the setting of one-dimensional Schr\"odinger operators, see \cite{MR2307748,MR2310217}.

\subsection{Conservation laws}\label{ss:conserve}
We have already seen several conserved quantities for \eqref{BO} in \eqref{P&H} and \eqref{H2 defn}.  Although Theorem~\ref{t:main} requires conservation laws at lower regularity, we will first discuss the general family of `polynomial' conservation laws because it will highlight several important characters, as well as introduce some of our broader goals in this paper.

At present, there are multiple competing approaches to understanding these polynomial conservation laws; see, for example, \cite{Miller2012} for an accessible and succinct review.  As an offshoot of the developments needed for Theorem~\ref{t:main}, we will offer a new unity between these approaches by connecting them back to the central objects of our analysis.

The first demonstrations \cite{Bock1979,Nakamura1979} that \eqref{BO} admits infinitely many conservation laws followed the approach of \cite{MR252826}, by introducing one-parameter families of Miura-type transformations.  The connection between these two papers was later explained in \cite{MR784923}.  We will revisit the Bock--Kruskal approach in subsection~\ref{ss:Bock}; in Theorem~\ref{T:BK}, we link the Bock--Kruskal transformation to our own gauge.

A completely different approach was introduced in \cite{Fokas1981}, which presented a vector field $\tau$ which recursively generates conserved densities via forming commutators.  We will discuss this further in subsection~\ref{ss:higher} before presenting our own generalization in Section~\ref{S:L&P}; see Theorem~\ref{T:Virial}.

Another perspective on the conservation laws grew out of the development of an inverse scattering approach to \eqref{BO}, as detailed in \cite{MR690743,Fokas1983,Kaup1998a,Kaup1998}.  Already in \cite{MR690743}, it is remarked that the quantity
\begin{equation}\label{Nbar q}
\int q(x) \olN (x;z,q) \,dx
\end{equation}
is conserved under the \eqref{BO} flow.  Here $\ol N$ represents a certain formal solution of an \emph{inhomogeneous} eigenfunction equation:
\begin{equation}\label{Nbar}
-i \partial_x \olN  - C_+( q \olN) = z \olN - z \quad\text{with $\olN(x)\to 1$ as $x\to+\infty$}
\end{equation}
and spectral parameter $z\in[0,\infty)$, which is the essential spectrum of $\mc L$.  The word formal indicates that this is not an element of the underlying Hilbert space.  The nonlocal nature of the operator $\mc L$ makes the question of the existence of such solutions a delicate matter; see \cite{MR1073870,Wu2017}.

The inhomogeneity of \eqref{Nbar} is quite unexpected from an inverse scattering point of view --- one would expect honest eigenfunctions to be the central objects.  In fact, this approach lead to the study of two families of formal eigenfunctions, traditionally denoted $N$ and $\olM$, as well as two families of solutions to \eqref{Nbar}, namely, $\olN$ and $M$.  (We caution the reader that the bar appearing here does not indicate complex conjugation.)

Even in the familiar territory of Sturm--Liouville operators, we learn a lot by moving the spectral parameter off the spectrum.  Taking this step, \cite{Kaup1998a} considers the Fredholm equation, which in our preferred notation reads
\begin{equation}\label{W eq}
W = 1 + (\mc L_0 - z)^{-1} C_+(q W), \qtq{where} z\in\C\setminus[0,\infty)
\end{equation}
and $\mc L_0$ denotes $-i\partial_x$ acting on $L^2_+(\R)$, by analogy with \eqref{L defn} with $q\equiv 0$.  This paper also observes that $W$ is analytic in $z$ and that the functions $M$ and $\olN$ mentioned earlier may be realized as the boundary values (from above and below) of $W$.

Our central object in this paper will be $m(x;\kappa,q)$, defined via 
\begin{equation}\label{m eq0}
-i m' - C_+[q(m+1)] +\kappa m =0 \qtq{or equivalently,}  m = (\mc L + \kappa)^{-1} C_+ q . 
\end{equation}
The sign change in the spectral parameter is motivated by the fact that we shall only need to consider $-z=\kappa >0$; moreover, $\kappa$ will be sufficiently large so that $\mc L + \kappa$ is indeed invertible.

In the line setting, $m$ differs little from $W$; indeed, $W=1+m$.  However, one of the virtues of $m$ is that it allows us to transition seamlessly between the line and circle geometries. 

The direct analogue of the conserved quantity mentioned in \eqref{Nbar q} is
\begin{equation}\label{beta defn}
\beta(\kappa;q) := \int q(x) m(x;\kappa,q)\, dx = \langle q_+, (\mc L +\kappa)^{-1} q_+ \rangle_{L^2_+} .
\end{equation}
The only difference is the removal of the term $\int q$, whose inclusion would curtail applicability of this to $q\in L^1$.  
In calling this quantity $\beta$, we are following Talbut~\cite{Talbut2021}, where it arises after differentiating the perturbation determinant with respect to the spectral parameter; see subsection~\ref{ss:PD}.  This use of $\beta$ is very different from the object with this name in~\cite{Kaup1998}!    

Kaup--Matsuno~\cite{Kaup1998} approached the question of polynomial conservation laws by expanding \eqref{Nbar q} in increasing powers of $z$, noting that \eqref{Nbar} gave a means of recursively generating the coefficients.  In the line geometry, one finds
\begin{equation}\label{beta as genfun}
\beta(\kappa;q) = \kappa^{-1} P(q) - \kappa^{-2} \hbo(q) + \kappa^{-3} H_2(q) + \bigO(\kappa^{-4}).
\end{equation}
On the circle, by comparison, one has
\begin{equation*}
\beta(\kappa;q) = \kappa^{-1} \Bigl(P(q)+\tfrac12\tint\Bigr) - \kappa^{-2}\Bigl( \hbo(q) - \btint P(q) -\tfrac16 \btint^3 \Bigr)
	+ \bigO(\kappa^{-3}).
\end{equation*}
%nevertheless, we do see that $\beta(\kappa)$ serves as a generating function for the polynomial conserved quantities.

A variation on this approach discussed, for example, in \cite{MR4275336,MR4148823,Sun2021} is to expand the resolvent in \eqref{beta defn} to obtain 
\begin{equation}\label{beta as geo}
\beta(\kappa;q) \sim \sum_{\ell\geq 0} (-1)^\ell \kappa^{-\ell-1} \langle q_+, \mc L^\ell q_+ \rangle ,
\end{equation}
which exhibits a very direct relationship between the Lax operator and the conservation laws of a type not seen, for example, for KdV.  In the circle setting, one may exploit the fact that $q_+=\mc L 1$ to present this formula in a different way; see \eqref{beta4}.  

While the polynomial conservation laws only make sense for very smooth initial data, we will show that their generating function $\beta(\kappa;q)$ makes sense in either geometry for $q\in H^s$ with $s>-\tfrac12$; see Proposition~\ref{P:beta}.  As we will demonstrate, this can be used to obtain $H^s$-bounds on smooth solutions, yielding a new proof of the following:

\begin{theorem}[Conservation laws, \cite{Talbut2021a}]\label{T:Talbut}
Let $q$ be a (global) $H^\infty$ solution to \eqref{BO}, either on the line or on the circle.  Then for all $-\frac12<s<0$ and $t\in \R$ we have 
\begin{align*}
%\bigl(1+ \|q(0)\|_{H^s}^{\frac2{1+2s}}\bigr)^s\sup_{t\in \R} \|q(t)\|_{H^s}\lesssim 
\bigl(1+ \|q(0)\|_{H^s}\bigr)^{-2|s|} \|q(0)\|_{H^s}\lesssim
	\|q(t)\|_{H^s}\lesssim \bigl(1+ \|q(0)\|_{H^s}\bigr)^{\frac{2|s|}{1-2|s|}}\|q(0)\|_{H^s}.
\end{align*}
\end{theorem}

This is not a verbatim recapitulation of Talbut's result: he imposes a mean-zero assumption in the circle case and formulates an inferior lower bound on $q(t)$.  Nevertheless, this result can be deduced from his arguments with only minor changes.

The argument in \cite{Talbut2021a} is based on the analysis of a renormalized perturbation determinant in a manner inspired by \cite{MR3820439}.  This object will be described in subsection~\ref{ss:PD}, where we will also discuss its relationship to $\beta(\kappa;q)$.  We will give a direct proof of Theorem~\ref{T:Talbut}, based solely on $\beta(\kappa;q)$; see Corollary~\ref{C:Qstar}.  In fact, Corollary~\ref{C:Qstar}, and Lemma~\ref{L:beta boot} on which it is based, are stronger than Theorem~\ref{T:Talbut} in two ways: they allow more general flows from the \eqref{BO} hierarchy and they demonstrate not only that solutions are bounded, but also that equicontinuous sets of initial data lead to equicontinuous ensembles of orbits.

\subsection{The method of commuting flows}\label{ss:commute}
\emph{A priori} equicontinuity results of the type with which we ended the previous subsection have been an integral part of the method of commuting flows since its inception. They have many roles.  For example, suppose we have a bounded sequence in $H^s$ that is convergent in $H^{-100}$; then this sequence converges in $H^s$ if and only if it is $H^s$-equicontinuous.  In this way, equicontinuity allows us to recover any loss of derivatives that may appear when proving that the flow depends continuously on the initial data.

The main question we need to address is this: How are we to estimate the divergence of two solutions with slightly different initial data?  One approach that has a long tradition is to interpose a regularized flow. Historically, this would typically be done via parabolic regularization, which introduces dissipation.  We will employ a Hamiltonian flow.  This will be generated by $H_\kappa$, which may be regarded as an approximation to $\hbo$.  In this way, we may rewrite the difference of the two solutions to \eqref{BO} with initial data $q^0$ and $\tilde q^{\:\!0}$ as
\begin{align}
e^{tJ\nabla \hbo}(q^0) - e^{tJ\nabla \hbo}(\tilde q^{\:\!0})  &= e^{tJ\nabla \hbo}(q^0) - e^{tJ\nabla H_\kappa}(q^0) \notag \\
	&\qquad + e^{tJ\nabla H_\kappa }(q^0) -  e^{tJ\nabla H_\kappa}(\tilde q^{\:\!0})  \label{3pieces} \\ 
	&\qquad\quad + e^{tJ\nabla H_\kappa}(\tilde q^{\:\!0})  - e^{tJ\nabla \hbo}(\tilde q^{\:\!0}). \notag
\end{align}
Here, $J$ stands for the operator $\partial_x$ of the Poisson bracket \eqref{Poisson}.

Any reasonable choice of regularized flow makes the middle term in RHS\eqref{3pieces} easy to estimate; this shifts the burden to estimating the first and last terms.  For these terms, the initial data is the same; however, the flows themselves are different.

%Moreover, we have not addressed one of the central problems: the data-to-solution map for \eqref{BO} is very irregular; this undermines typical approaches to estimating this difference.

The central principle of the method of commuting flows is to choose $H_\kappa$ to Poisson commute with $\hbo$ so that the corresponding flows commute. This commutativity allows us to write
\begin{align}\label{diff of flows}
e^{tJ\nabla \hbo}(q^0) - e^{tJ\nabla H_\kappa}(q^0) = \bigl[ e^{tJ\nabla(\hbo - H_\kappa) } - \Id\bigr]\circ e^{tJ\nabla H_\kappa}(q^0) .
\end{align}
In this way, we are led to the following problem: show that the flow generated by $\hbo - H_\kappa$ is close to the identity, while accepting that the initial data for this more complicated flow is not simply $q^0$.  Indeed, $q^0$ is `scrambled' by the $H_\kappa$ flow, for which we have little uniform control as $\kappa\to\infty$.

Prior work on other models informs where to seek inspiration for the choice of the regularized Hamiltonian $H_\kappa$, namely, from the expansion \eqref{beta as genfun} and its torus analogue.  This reasoning leads us to select
\begin{equation}\label{Hk intro}
\hk (q):= \begin{cases}  \kappa P(q) - \kappa^2 \beta(\kappa;q) & \text{on $\R$,} \\
	\bigl[ \kappa + \tint \bigr] P(q) - \kappa^2 \beta(\kappa;q)  + \tfrac{\kappa}2 \btint^2 + \tfrac16 \btint^3 & \text{on $\T$.}
	\end{cases}
\end{equation}

Although there are many facets to the full story, we would like to focus attention on \eqref{diff of flows}, how it limited Talbut's analysis to the case of $L^2$ initial data, and how we were able to overcome these obstructions.

By writing the nonlinearity as a complete derivative, we see that the vector field defining the \eqref{BO} flow is actually continuous on $L^2$, albeit $H^{-2}$-valued.  Likewise, the $H_\kappa$ Hamiltonian defines a continuous vector field on $L^2$.  In this way, we may analyze the difference flow directly as the difference of these two vector fields.  As noted above, the inevitable loss of two derivatives may be recovered by exploiting equicontinuity.  This is what Talbut does in \cite{Talbut2021}.  However, as soon as $s<0$, we may no longer make sense of $q^2$, for $q\in H^s$, even as a distribution. 

The idea of incorporating a gauge transformation into the method of commuting flows appears already in \cite{Killip2019}, although it is not always a prerequisite for obtaining sharp results; see \cite{harropgriffiths2022global}.  The big hurdle is finding the right transformation.

It is natural to try Tao's gauge \cite{Tao2004}. However, the high-low interactions that are so troublesome for his style of analysis and which this gauge removes, are of no consequence for our methodology; indeed, outermost derivatives are handled with equicontinuity.  Ultimately, we do not find this transformation helpful for  our analysis.

In previous incarnations of the method of commuting flows, it was the diagonal Green's function that played a central role.  It is elementary to verify that even when $q\equiv0$, the Green's function diverges on the diagonal; thus, renormalization is required.  In the case of \eqref{BO}, however, we found this approach to be fruitless.

Our next attempt was to employ the gauge transformation introduced by Bock and Kruskal \cite{Bock1979} in their study of conservation laws for \eqref{BO} posed on the line.  This gauge is defined implicitly via
\begin{equation}\label{BK u}
2q=  \tfrac{1}{w+\kappa} \h (w') + \h \bigl[\tfrac{w'}{w+\kappa}\bigr] +\tfrac{2\kappa w}{w+\kappa}.
\end{equation}
In subsection~\ref{ss:Bock}, we will demonstrate the existence and uniqueness of such a $w$; indeed, we will show this is possible even for $q\in H^s$ with $s>-\frac12$, and that the transformed unknown $w$ lies in $H^{s+1}$.

As noted in \cite{Bock1979}, it is not difficult to verify that \eqref{BO} may be written as
\begin{equation}\label{BK u dot} 
\tfrac{d}{dt} w = \h w'' - 2qw',
\end{equation}
which does not appear to constitute progress --- how can we hope to multiply $q\in H^s$ and $w'\in H^s$?  However, combining this with \eqref{BK u}, a little work reveals that 
\begin{equation*}%\label{BK u dot''} 
\tfrac{d}{dt} w = \h w'' + 2i C_+\bigl(w'\bigr) \cdot C_+\bigl(\tfrac{w'}{\kappa+w}\bigr) - 2i C_-\bigl(w'\bigr) \cdot C_-\bigl(\tfrac{w'}{\kappa+w}\bigr)
	 + \tfrac{2\kappa w w'}{\kappa+w} .
\end{equation*}
This was our first breakthrough on the problem! The fact that this is progress rests on a simple but fundamental observation: the product of two functions in $H^s_+$ is a well-defined distribution; see Lemma~\ref{L:+*+}.  Of course, this is not true without the frequency restriction.

Next, we must find a description of the dynamics of $w$ under the regularized Hamiltonian \eqref{Hk intro}.  Immediately, we strike new hurdles.  In past analyses employing a gauge transformation, we were lead to the regularized dynamics of the gauge variable through the biHamiltonian relation.  However, \cite{Fokas1981} shows that there is no such biHamiltonian formulation of \eqref{BO}!  On top of this, we could not find any documented relationship between $\beta(\kappa)$ and $w$, which might help derive such dynamics.  This is the important role of  Theorem~\ref{T:BK} in our story: it connects $w$ to $m$ and thence to $\beta$.

As we investigated $w$ through its connection to $m$, it soon became apparent that our treatment could be much simplified by abandoning $w$ and adopting $m$ as our new gauge.  It is striking to us that despite the long history of $m$ in the theory of \eqref{BO}, its value as a gauge transformation has been overlooked until now.

The abandonment of $w$ and adoption of $m$ as our gauge transformation accelerated us toward a proof of Theorem~\ref{t:main}, albeit not the proof presented here.  The simplicity of the arguments in this paper benefits substantially from a further innovation, namely, the Lax pair presented in Proposition~\ref{P:HK}.  We do not alter the traditional Lax operator $\mc L$, only its antisymmetric partner $\mc P$, which we call the Peter operator (Lax's first name).

Although a Lax representation of the flow generated by $\beta(\kappa)$ has appeared previously in Proposition~2.17 of \cite{Sun2021}, this would not lead one to \eqref{PkR} or \eqref{PkT}; the first term in each equation is new.  At first glance, this may seem inconsequential; however, the inclusion of these first terms makes a huge difference.  It is only these modified Peter operators that satisfy the special properties \eqref{Hk flow Pk} and \eqref{n dot BO v1}, which much simplify the proof of Theorem~\ref{t:main} in Section~\ref{S:GWP}.  Additional special properties of our Peter operators are discussed in Section~\ref{S:L&P}.

%In addition, our Peter operators have the special property of annihilating the constant function $1$.  (For a clarification of what this means in the line setting, see Lemma~\ref{L:P10}.)  

\subsection{Applications of the new Lax pair}\label{ss:applic}
Section~\ref{S:L&P} is devoted to reaping certain other rewards from our new Lax pair, not directly related to well-posedness.  Here the reader will find Theorem~\ref{T:Gerard}, which provides an extension of G\'erard's recent explicit formula \cite{Gerard2022} for \eqref{BO} to the full hierarchy, as well as Theorem~\ref{T:Virial} which describes the action of a one-parameter family of higher symmetries. 

The notion of a higher symmetry is described in subsection~\ref{ss:higher}.  It is a symmetry that lies outside the commuting flows of the hierarchy because it does \emph{not} preserve the values of the commuting Hamiltonians.  Scaling and Galilei boosts are simple examples.  We also discuss a much more profound example from \cite{Fokas1981}, for which we provide a mechanical explanation: the center of energy travels at a constant speed under every flow of the hierarchy.  One is then led to ask if there are centers associated to the other conserved quantities that also travel at constant speed.  Theorem~\ref{T:Virial} answers this in the affirmative, thereby presenting new recursion relations within the hierarchy.  As a consonant  example of the utility of our Lax pair, we present a generalization of the variance identity of \cite{Ifrim2019} to the full \eqref{BO} hierarchy.

In extending G\'erard's formula to the full hierarchy, we actually find an explicit formula for the $\tau$-function associated with \eqref{BO}.  By a $\tau$-function, we mean an expression for the solution under a general Hamiltonian.  Traditionally,
\begin{align}\label{2:17}
q(\vec t;q_0) = \Bigl[\exp\bigl\{ {\textstyle\sum} t_i J\nabla H_i \bigr\}q_0 \Bigr] ( x=0 )
\end{align} 
would be written as a logarithmic derivative of the $\tau$-function; however, such a $\tau$-function evidently contains as much information as $q(\vec t;q_0)$.  Here $H_i$ enumerate the commuting Hamiltonians of the hierarchy, while $\vec t$ denotes a vector of times (with only finitely many non-zero terms).  Note that this function is scalar-valued.  This is no loss of generality because momentum is one of the Hamiltonians, traditionally assigned index $i=0$; consequently, one may recover the value of the solution at any spatial point by using the variable $t_0$.

The relation \eqref{beta as geo} has inspired us to propose parameterizing the $\tau$-function in a different way, namely, by continuous functions $\phi$.  Just as
\begin{equation}\label{Hphi defn}
H_\phi(q) := \langle q_+, \phi(\mc L)q_+\rangle
\end{equation}
defines a conserved quantity for the hierarchy, so we may define
\begin{equation}\label{qphi defn}
q(\phi;q_0) = \bigl( e^{ J\nabla H_\phi} q_0\bigr)( x=0 ).
\end{equation}
When $\phi$ is a polynomial, this reproduces \eqref{2:17}.

In Section~\ref{S:L&P} we will prove the following formula for a dense class of functions $\phi$:
\begin{equation}\label{tau phi}
C_+\bigl( e^{ J\nabla H_\phi} q_0\bigr)( x+iy ) = \tfrac{1}{2\pi i} I_+ \Big( \big( X - t \psi(\mc L_{q_0}) - x - iy \big)^{-1} q^0_+\Big) \end{equation}
for $x\in\R$, $y>0$.  Recall that functions in the Hardy space are analytic in the upper half-plane; moreover, as $q$ is real-valued, it may be recovered from its positive-frequency part.  Here $X$ denotes the operator of multiplication by $x$ and $I_+$ denotes a kind of conditional integral; both are described in detail in Section~\ref{s:pre}.  The function $\psi$ applied to the Lax operator associated to the initial data $q_0$ is defined via 
\begin{equation}\label{psi from phi}
\psi(E) = \phi(E) +  E\phi'(E).
\end{equation}
This new algebraic relation has an important role: it reveals exactly how the explicit formula \eqref{tau phi} varies in response to changes in the Hamiltonian.

\subsection*{Acknowledgements} R.K. was supported by NSF grants DMS-1856755 and DMS-2154022;  M.V. was supported by NSF grant DMS-2054194.  The work of T.L. was also supported by these grants.

\section{Notation and preliminaries} \label{sec;preliminaries}

Our conventions for the Fourier transform are
\begin{align*}
\hat f(\xi) = \tfrac{1}{\sqrt{2\pi}} \int_\R e^{-i\xi x} f(x)\,dx  \qtq{so} f(x) = \tfrac{1}{\sqrt{2\pi}} \int_\R e^{i\xi x} \hat f(\xi)\,d\xi
\end{align*}
for functions on the line, while on the circle,
\begin{align*}
\hat f(\xi) = \int_0^1 e^{- i\xi x} f(x)\,dx \qtq{so} f(x) = \sum_{\xi\in 2\pi\Z} \hat f(\xi) e^{i\xi x}.
\end{align*}
These Fourier transforms are unitary on $L^2$ and yield the Plancherel identities
\begin{align*}
    \|f\|_{L^2(\mathbb R)}=\|\hat f\|_{L^2(\mathbb R)} \qtq{and}   \|f\|_{L^2(\mathbb T)}=\sum_{\xi\in 2\pi \mathbb Z}|\hat f(\xi)|^2.
\end{align*}

With these conventions, we define the Hilbert transform via
\begin{equation}\label{HT}
\wh{\h f}(\xi) = -i\sgn(\xi)\wh{f}(\xi)
\end{equation}
with the understanding that $\sgn(0)=0$, which is only important on the circle.
%as well as the following convolution identity on $\R$:
%\begin{align*}
%    \widehat {fg}= \tfrac{1}{\sqrt{2\pi}} \hat f \ast \hat g.
%\end{align*}

We will also employ the Cauchy--Szeg\H{o} projections defined via
\begin{equation}\label{C defn}
\wh{C_\pm f}(\xi) = 1_{[0,\infty)}(\pm\xi) \wh{f}(\xi)
\end{equation}
and often write $q_\pm = C_\pm q$.  Although $ i \h = C_+ - C_-$ in both geometries, we have
\begin{align}\label{Cid}
C_++C_-=1 \quad \text{only on the line; on the circle,} \quad C_+ f + C_- f = f + \textstyle \int \!f.
\end{align}

To avoid an unnecessary proliferation of parentheses, we adopt the following rules for the operators $C_\pm$: Their precedence is lower than multiplication indicated by juxtaposition (e.g., $fg$), but higher than multiplication indicated with a dot, addition, and subtraction.  Thus,
by our conventions,
\begin{equation*}
C_+ f \cdot C_+(\ol m +g)h + q = \bigl[C_+ f\bigr]  \bigl[C_+\bigl((\ol m+g)h\bigr)\bigr] + q .
\end{equation*}

For \(\sigma\in \R\) and \(\kappa\geq 1\) we define the Sobolev spaces $H^\sigma_\kappa(\R)$ and $H^\sigma_\kappa(\T)$ as the completion of $\mathcal S(\R)$ and $C^\infty(\T)$, respectively, with respect to the norms
\begin{equation*}
\norm{f}_{H^\sigma_\kappa(\R)}^2 = \int  (|\xi|+\kappa)^{2\sigma} |\wh{f}(\xi)|^2\,d\xi \qtq{and}
\| f\|_{H^{\sigma}_\kappa(\T)}^2 = \sum_{\xi\in 2\pi\Z} (|\xi| +\kappa)^{2\sigma} |\hat f(\xi)|^2 .
\end{equation*}
When $\kappa=1$, we simply write $H^\sigma(\R)$ and $H^\sigma(\T)$.  We write $H^\sigma_+$ for the subspace of $H^\sigma$ comprised of functions holomorphic in the upper half-plane.

Throughout the paper, we will employ the $L^2$ pairing: $\langle g, f \rangle = \int \ol g(x) f(x)\,dx$.  This informs our identification of $H^{\sigma}_\kappa$ and $H^{-\sigma}_\kappa$ as dual spaces.

For the remainder of the paper, we constrain
\begin{align}\label{seps}
s\in (-\tfrac12,0) \qtq{and define} \eps:= \tfrac12( \tfrac 12-|s|)\in (0,\tfrac14).
\end{align}
All implicit constants are permitted to depend on $s$.

As  $s+1>\frac{1}{2}$, the space $H^{s+1}_\kappa$ is an algebra in either geometry.  Indeed, we have 
\begin{equation}\label{high reg alg}
\norm{fg}_{H^{s+1}_\kappa} \lesssim \norm{f}_{H^{s+1}} \norm{g}_{H^{s+1}_\kappa} \quad\text{uniformly for $\kappa\geq 1$.}
\end{equation}
However, we will also need to handle products at considerably lower regularity; this is the topic of the next two lemmas.

\begin{lemma}\label{L:Hs mult}
The product of any $f\in H^s$ and $g\in H^{s+1}$ belongs to $H^s$; indeed,
\begin{align}\label{E:Hs mult}
\| gf \|_{H^s_\kappa} \lesssim \bigl[ \|g\|_{L^\infty} + \| g\|_{H^{1/2}} \bigr] \| f \|_{H^s_\kappa}
	\lesssim \kappa^{-2\eps}\| g\|_{H^{s+1}_\kappa} \| f \|_{H^s_\kappa},
\end{align}
uniformly for $\kappa\geq 1$.  Here $s, \eps$ are as in \eqref{seps}.
\end{lemma}

\begin{proof}
The second inequality in \eqref{E:Hs mult} is elementary.  We focus on the first.

By duality, it suffices to verify that
\begin{align}\label{sigma mult}
\| g h \|_{H^\sigma_\kappa} \lesssim \bigl[ \|g\|_{L^\infty} + \| g\|_{H^{1/2}}  \bigr] \| h  \|_{H^\sigma_\kappa}
\end{align}
holds with $\sigma=|s|$.  In fact, \eqref{sigma mult} holds for any $\sigma\in[0,\frac12)$.  This is a special case of Theorem~II.3.2 in \cite{MR0215084}.  For completeness, we give an elementary proof of our own.

Our argument is based on the Besov--Slobodeckij characterization:
\begin{align}\label{3:18}
\| h \|_{H^\sigma_\kappa}^2 \sim  \kappa^{2\sigma}\|h\|_{L^2}^2 + \iint \frac{|h(x)-h(y)|^2}{|x-y|^{2\sigma+1}} \,dx\,dy \quad\text{for any $\sigma\in(0,1)$.}
\end{align}

It is not difficult to see that
\begin{align}\label{3:19}
|(gh)( x) - (gh)( y)|^2 \lesssim \|g\|_{L^\infty}^2 |h( x) -h( y)|^2 + |h( x)| |h( y)| |g(x) -g(y)|^2 .
\end{align}
The first summand presents no difficulty.  For the second summand we employ H\"older's inequality and then the homogeneous Sobolev embedding $\dot H^\sigma\hookrightarrow L^{2/(1-2\sigma)}$:
\begin{align*}
\iint \frac{|h( x)| |h( y)| |g(x) -g(y)|^2}{|x-y|^{2\sigma+1}} \,dx\,dy
	&\lesssim  \|h( x)h( y)\|_{L^{\frac2{1-2\sigma}}_{x,y}}
		\biggl\| \frac{|g(x) -g(y)|^2}{|x-y|^{2\sigma+1}} \biggr\|_{L^{\frac2{1+2\sigma}}_{x,y}} \\
&\lesssim \|h\|_{H^\sigma_\kappa}^2 \| g \|_{L^\infty}^{1-2\sigma} \|g\|_{H^{1/2}}^{1+2\sigma}. \qedhere
\end{align*}
\end{proof}

In general, pointwise multipliers on negative regularity spaces must have considerable positive regularity; indeed, this is evident from the duality reduction performed in this proof.  There is one important exception, namely, when both functions lie in the same Hardy space.  This observation, whose proof is quite elementary, plays a crucial role in our analysis.

\begin{lemma}\label{L:+*+}
Fix $r<0$.  Then for $f,g \in H^r_+$ we have
\begin{equation}\label{low reg alg}
\norm{fg}_{H^{2r-1}} \lesssim \norm{f}_{H^r} \norm{g}_{H^r} .
\end{equation}
\end{lemma}

\begin{proof}
We start by rewriting LHS\eqref{low reg alg} in Fourier variables:
\begin{align*}
\norm{fg}_{H^{2r-1}}^2
&= \frac{1}{2\pi} \int_0^\infty \frac{1}{(\xi+1)^{4|r|+2}} \bigg| \int_0^\xi \wh{f}(\xi-\eta) \wh{g}(\eta)\,d\eta\bigg|^2\,d\xi.
\end{align*}
	
Using that for $\eta\in [0,\xi]$ we have
\begin{equation*}
\tfrac{1}{(\xi+1)^2}\leq \tfrac{1}{\xi-\eta+1} \cdot \tfrac{1}{\eta+1},
	\end{equation*}
distributing the factors of $(\xi+1)^{4r}$ evenly between $f$ and $g$, and using Cauchy--Schwarz, we may bound
\begin{align*}
\int_0^\infty \frac{1}{(\xi+1)^{4|r|+2}}& \bigg| \int_0^{\xi} \wh{f}(\xi-\eta) \wh{g}(\eta)\,d\eta\bigg|^2\,d\xi \\
&\leq \int_0^\infty \frac{1}{(\xi+1)^2} \bigg(  \int_0^\xi \frac{ |\wh{f}(\xi-\eta)| }{ (\xi-\eta+1)^{|r|} }\, \frac{ |\wh{g}(\eta)| }{ (\eta+1)^{|r|} } \,d\eta\bigg)^2\,d\xi \\
&\leq \int_0^\infty \frac{1}{(\xi+1)^2} \norm{f}_{H^r}^2 \norm{g}_{H^r}^2\,d\xi 
\lesssim \norm{f}_{H^r}^2 \norm{g}_{H^r}^2 .\qedhere
\end{align*}
\end{proof}

\begin{definition}[Equicontinuity]\label{d:equicontinuity}
Fix $\sigma\in \R$. A bounded set \(Q\subset H^\sigma\) is said to be \emph{equicontinuous} if
\[
\limsup_{\delta\to0} \ \sup_{q\in Q} \ \sup_{|y|<\delta}\|q(\cdot +y) - q(\cdot)\|_{H^\sigma} = 0.
\]
\end{definition}

By Plancherel, equicontinuity in the spatial variable is equivalent to tightness in the Fourier variable. Specifically, a bounded set $Q\subset H^\sigma$ is equicontinuous if and only if
\begin{equation}\label{equicty 1}
\lim_{ \kappa\to\infty}\sup_{q\in Q}\int_{|\xi|\geq \kappa} |\wh{q}(\xi)|^2 (|\xi|+1)^{2\sigma}\, d\xi =0\quad\text{on $\R$}
\end{equation}
or 
\begin{equation}\label{equicty 1'}
\lim_{ \kappa\to\infty}\sup_{q\in Q}\sum_{|\xi|\geq \kappa} |\wh{q}(\xi)|^2 (|\xi|+1)^{2\sigma}=0\quad\text{on $\T$}.
\end{equation}

It is important for our arguments that we are able to transfer the equicontinuity property from classes of initial data to the corresponding orbits.  This we achieve by combining the following characterization of equicontinuity with the two-sided estimate \eqref{beta est} and the conservation of $\beta(\kappa;q)$. 

\begin{lemma}[Characterization of equicontinuity]\label{L:equi}
Let $Q$ be a bounded subset of $H^s$.  Then the following are equivalent:
\begin{enumerate}[label=(\roman*)]
\item The subset $Q$ is equicontinuous in $H^s$.
\item $\norm{q}_{H^s_\kappa} \to 0$ as $\kappa\to\infty$ uniformly for $q\in Q$.
\end{enumerate}
\end{lemma}

\begin{proof}
We only consider the real-line case below; the argument on the circle is similar, with integrals being replaced by sums.

First, we show that (i) implies (ii).  Fix $\del>0$.  For $\vk\geq 1$ to be chosen later, we may bound
\begin{equation*}
\int_{\R} \frac{|\wh{q}(\xi)|^2}{(|\xi|+\kappa)^{2|s|}}\, d\xi
\lesssim \frac{\vk^{2|s|}}{\kappa^{2|s|}} \int_{\R} \frac{|\wh{q}(\xi)|^2}{(|\xi|+1)^{2|s|}}\, d\xi + \int_{|\xi|\geq \vk} \frac{|\wh{q}(\xi)|^2}{(|\xi|+1)^{2|s|}}\, d\xi .
\end{equation*}
As $Q$ is equicontinuous, we may pick $\vk=\vk(\delta)$ sufficiently large so that the second integral on the right-hand side is at most $\del$.  Then, as $Q$ is bounded in $H^s$, we may choose $\kappa$ sufficiently large so that the first term on the right-hand side is at most $\del$.  Together, this shows that the left-hand side is at most $2\del$ for all $\kappa$ sufficiently large, uniformly for $q\in Q$. As $\delta>0$ was arbitrary, this proves (ii).
	
Conversely, the inequality
\begin{equation*}
\int_{|\xi|\geq \kappa} \frac{|\wh{q}(\xi)|^2}{(|\xi|+1)^{2|s|}}\, d\xi
\lesssim \int_{\R} \frac{|\wh{q}(\xi)|^2}{(|\xi|+\kappa)^{2|s|}}\, d\xi 
\end{equation*}
shows that (ii) implies (i).
\end{proof}

%%%%%%%%%%%%%%%%%%%%%%%%%%%%%%%%%%%%%%%%%%%%%%%%%%%%%%%%%%%%%%%%%%%%%%%%%%
\section{The Lax operator}\label{s:pre} %%%%%%%%%%%%%%%%%%%%%%%%%%%%%%%%%%%%%
%%%%%%%%%%%%%%%%%%%%%%%%%%%%%%%%%%%%%%%%%%%%%%%%%%%%%%%%%%%%%%%%%%%%%%%%%%

In this section, we investigate the Lax operator and its mapping properties.  We begin by establishing inequalities that will allow us to prove convergence of the various resolvent expansions that arise in our analysis.

\begin{lemma}\label{L:key est}
For $s, \eps$ as in \eqref{seps}, we have
\begin{equation}
\begin{aligned}
\norm{ C_+qR_0(\kappa)C_+f }_{H^s} &\lesssim \kappa^{-2\eps} \norm{q}_{H^s} \norm{f_+}_{H^s_\kappa} ,\\
\norm{ C_+qR_0(\kappa)C_+f }_{H^s_\kappa} &\lesssim \kappa^{-2\eps} \norm{q}_{H^s_\kappa} \norm{f_+}_{H^s_\kappa} ,
\end{aligned}
\label{key est}
\end{equation}
where the implicit constants are uniform in $\kappa\geq 1$.  Moreover, 
\begin{align}\label{ke2}
\norm{ C_+\, q\,C_+\, f }_{H^{-1/2}_\kappa} 
&\lesssim \kappa^{-2\eps} \| q\|_{H^s_\kappa} \norm{f_+}_{H^{1/2}_\kappa}.
\end{align}
\end{lemma}

\begin{proof}
We present the details on the line; the argument on the circle is a close analogue, with integrals replaced by sums.  There is little difference between the proofs of the two estimates \eqref{key est}.  We will illustrate the argument with the former because it contains both normal and $\kappa$-modified Sobolev norms. 

In Fourier variables, we have
\begin{align*}
\norm{ C_+qR_0(\kappa)C_+f }_{H^s} ^2
&= \frac{1}{2\pi} \int_0^\infty \frac{1}{(\xi+1)^{2|s|}} \bigg| \int_0^\infty \wh{q}(\xi-\eta) \frac{\wh{f}(\eta)}{\eta+\kappa}\,d\eta \bigg|^2 d\xi  .
\end{align*}

To estimate the contribution of the region where $\eta \geq 2\xi\geq 0$, we use that 
\begin{equation*}
\frac{1}{\eta+\kappa} \lesssim \frac{1}{(\eta+\kappa)^{|s|}}\, \frac{1}{(|\xi-\eta|+1)^{|s|}}\, \frac{1}{(\xi+\kappa)^{1-2|s|}}
\end{equation*}
uniformly for $\eta \geq 2\xi\geq 0$ and $\kappa\geq 1$.  Together with Cauchy--Schwarz, this yields
\begin{align*}
\int_0^\infty & (\xi+1)^{2s} \bigg| \int_{2\xi}^\infty \wh{q}(\xi-\eta) \frac{\wh{f}(\eta)}{\eta+\kappa}\,d\eta \bigg|^2 d\xi  \\
&\lesssim \int_0^\infty \frac{(\xi+1)^{2s}}{(\xi+\kappa)^{2 - 4|s|}} \bigg( \int_{2\xi}^\infty \frac{|\wh{q}(\xi-\eta)|}{(|\xi-\eta|+1)^{|s|}}
	\frac{|\wh{f}(\eta)|}{(\eta+\kappa)^{|s|}}\,d\eta \bigg)^2 d\xi \\
&\leq \int_0^\infty \frac{d\xi}{(\xi+1)^{1-4\eps}(\xi+\kappa)^{8\eps}} \norm{q}_{H^s}^2 \norm{f_+}_{H^{s}_\kappa}^2 \\
&\lesssim \kappa^{-4\eps}\norm{q}_{H^s}^2 \norm{f_+}_{H^{s}_\kappa}^2 .
\end{align*}
In the last step we integrated separately over $\xi\in[0,\kappa]$ and $\xi\in[\kappa,\infty)$.
	
To estimate the contribution of the remaining region, $0\leq \eta \leq 2\xi$, we use that
\begin{equation*}
\frac{1}{(\xi+1)^{2|s|}} \lesssim \frac{1}{(|\xi-\eta|+1)^{2|s|}} \qtq{uniformly for}
	0\leq \eta \leq 2\xi.
\end{equation*}
Together with the Minkowski and Cauchy--Schwarz inequalities, this yields
\begin{align*}
\int_0^\infty &\frac{1}{(\xi+1)^{2|s|}} \bigg| \int_0^{2\xi} \wh{q}(\xi-\eta) \frac{\wh{f}(\eta)}{\eta+\kappa}\,d\eta \bigg|^2 d\xi  \\
&\lesssim \int_0^\infty \bigg( \int_0^\infty \frac{|\wh{q}(\xi-\eta)|}{(|\xi-\eta|+1)^{|s|}}\, 
	\frac{|\wh{f}(\eta)|}{(\eta+\kappa)^{|s|}}\, \frac{d\eta}{(\eta+\kappa)^{1-|s|}} \bigg)^2 d\xi  \\
&\leq \norm{q}_{H^s}^2 \biggl( \int_0^\infty \frac{|\wh{f}(\eta)|}{(\eta+\kappa)^{|s|}}\, \frac{d\eta}{(\eta+\kappa)^{1-|s|}} \biggr)^2 \\
&\leq \norm{q}_{H^s}^2 \norm{f_+}_{H^{s}_\kappa}^2 \int_0^\infty \frac{d\eta}{(\eta+\kappa)^{1+4\eps}} \\
&\lesssim \kappa^{-4\eps}\norm{q}_{H^s}^2 \norm{f_+}_{H^{s}_\kappa}^2 .
\end{align*}
As in the previous region, $\eps>0$ is needed for convergence of the integral.  Together with our treatment of the first region, this proves~\eqref{key est}.  

It remains to prove \eqref{ke2}.  We proceed as previously.  Observing that
$$
\kappa^\eps (\xi+\kappa)^\eps \lesssim  (|\xi-\eta|+\kappa)^{-|s|} \sqrt{\eta+\kappa}  \quad\text{uniformly for $2\eta > \xi > 0$ and $\kappa\geq 1$}
$$
and using Cauchy--Schwarz, we deduce that
\begin{align}
\int_0^\infty \bigg| \int_{\xi/2}^\infty \wh{q}(\xi-\eta)  &\wh{f}(\eta)\,d\eta \bigg|^2 \frac{d\xi}{\xi+\kappa} \notag \\
&\lesssim \int_0^\infty \frac{\kappa^{-2\eps} }{(\xi+\kappa)^{1+2\eps}}
	\bigg| \int_0^\infty \frac{|\wh{q}(\xi-\eta)| \,|\wh{f}(\eta)|}{(|\xi-\eta|+\kappa)^{|s|}} \sqrt{\eta+\kappa} \,d\eta \bigg|^2 d\xi \notag\\
&\lesssim \kappa^{-4\eps} \| q\|_{H^s_\kappa}^2 \norm{f_+}_{H^{1/2}_\kappa}^2. \label{aaa}
\end{align}
Complementing this, we have  
$$
(\xi+ \kappa)^{-1} \lesssim  (|\xi-\eta|+\kappa)^{-2|s|} (\eta+\kappa)^{2|s|-1}  \quad\text{uniformly for $0< 2\eta < \xi$ and $\kappa\geq 1$}.
$$
Consequently, by Minkowski and Cauchy--Schwarz, 
\begin{align}
\int_0^\infty \bigg| \int_0^{\xi/2} \wh{q}(\xi-\eta) & \wh{f}(\eta)\,d\eta \bigg|^2 \frac{d\xi}{\xi+\kappa} \notag \\
&\lesssim \int_0^\infty \bigg| \int_0^\infty \frac{|\wh{q}(\xi-\eta)|}{(|\xi-\eta|+\kappa)^{|s|}}
	\frac{\sqrt{\eta+\kappa} \; |\wh{f}(\eta)|}{(\eta+\kappa)^{1-|s|}} \,d\eta \bigg|^2 d\xi  \notag \\
&\lesssim \kappa^{-4\eps} \| q\|_{H^s_\kappa}^2 \norm{f_+}_{H^{1/2}_\kappa}^2. \label{bbb}
\end{align}
Combining \eqref{aaa} and \eqref{bbb} proves \eqref{ke2}.
\end{proof}

We now come to the principal purpose of this section, namely, understanding $\mc L$ as a selfadjoint operator and obtaining quantitative information on its mapping properties, as well as those of its resolvent. 

\begin{proposition}[Lax operator]\label{P:Lax}
Let $s, \eps$ be as in \eqref{seps}. Given $q\in H^s$, there is a unique selfadjoint, semi-bounded operator $\mc{L}$ associated to the quadratic form
\begin{equation*}
f \mapsto \langle f, \mc L_0 f\rangle - \int q(x) |f(x)|^2 \,dx
\end{equation*}
having form domain $H^{1/2}_+$.  This operator satisfies
\begin{equation}\label{L est}
\norm{ \mc{L}f }_{H^s} \lesssim \bigl[ 1 + \norm{q}_{H^s} \bigr]\norm{f}_{H^{s+1}} .
\end{equation}
Moreover, there is a constant $C_s\geq 1$ so that whenever
\begin{equation}\label{big k0 0}
\kappa \geq C_s \bigl( 1 + \norm{q}_{H^s_\kappa} \bigr)^{\frac1{2\eps}},
\end{equation}
the resolvent $R(\kappa;q)$ of $\mc L$ exists, maps $H^{-1/2}_+$ into $H^{1/2}_+$, and satisfies
\begin{equation}\label{R est}
\norm{ R(\kappa) f }_{H^{s+1}_\kappa} \lesssim \norm{f}_{H^{s}_\kappa}
\ \text{and}\ \ % 
\bigl\| [R(\kappa)-R_0(\kappa)] f \bigr\|_{H^{s+1}_\kappa} \lesssim \kappa^{-2\eps} \norm{q}_{H^s_\kappa}\norm{f}_{H^s_\kappa} .
\end{equation}

The essential spectrum $\sess(\mc L)$ agrees with that of $\mc L_0$ and for any $f\in H^{s}_+$,
\begin{align}\label{mero QF}
z \mapsto \langle f, (\mc L+z)^{-1} f\rangle 
\end{align}
defines a meromorphic function on the region where $-z\in\C\setminus\sess(\mc L)$.
\end{proposition}

\begin{proof}
For $f\in H^{1/2}_+$, the estimate~\eqref{ke2} shows
\begin{align}\label{q vs L0}
\bigl| \langle f, qf\rangle \bigr|
	\lesssim \kappa^{-2\eps} \norm{q}_{H^s_\kappa} \|f\|_{H^{1/2}_\kappa}^2
	= \kappa^{-2\eps} \norm{q}_{H^s}  \langle f, (\mc L_0 + \kappa) f \rangle.
\end{align}
By choosing $\kappa$ large, we see that the potential $q$ is an infinitesimally form-bounded perturbation of the operator $\mc L_0$.  Therefore the existence and uniqueness of $\mc{L}$ follows from~\cite[Th.~X.17]{MR0493420}.  The operator so defined automatically maps the form domain $H^{1/2}_+$ into its dual space $H^{-1/2}_+$.  (It will not be important for us to discuss the operator domain of $\mc L$.)  The estimate \eqref{L est}  follows directly from \eqref{E:Hs mult}.

By virtue of Lemma~\ref{L:key est}, there is a choice of $C_s\geq 1$ so that \eqref{big k0 0} ensures
\begin{align}\label{big k0 0'}
 \| C_+qR_0(\kappa) C_+\|_{H^s_{\kappa} \to H^s_{\kappa}} <\tfrac12 \qtq{and} \| C_+qR_0(\kappa) \|_{H^{-1/2}_+\to H^{-1/2}_+} <\tfrac12 .
\end{align}
This in turn guarantees the convergence of the resolvent series 
\begin{equation}\label{R series}
R(\kappa;q) = (\mc{L}+\kappa)^{-1} =  R_0(\kappa)\sum_{\ell\geq 0}\bigl[C_+qR_0(\kappa)\bigr]^\ell, 
\end{equation}
both as an operator from $H^s_\kappa$ to $H^{s+1}_\kappa$ and as an operator from $H^{-\frac12}_+$ to $H^{\frac12}_+$.  This also proves both claims in \eqref{R est}.

To show that $\sess(\mc L)=\sess(\mc L_0)$, we need only demonstrate that $R(\kappa)-R_0(\kappa)$ is a compact operator for some $\kappa>0$; see \cite[Th. XIII.14]{MR0493421}.  For this purpose, we write
\begin{equation}\label{root R naughty}
R(\kappa)-R_0(\kappa) = R_0(\kappa) C_+ q \sqrt{R_0(\kappa)} \cdot \sqrt{R_0(\kappa)} \bigl[ 1 + C_+ q R(\kappa) \bigr].
\end{equation}
It is easy to verify that the first factor in this expansion is compact by computing its Hilbert--Schmidt norm.  On the line, for example,
$$
\bigl\| R_0(\kappa) C_+ q \sqrt{R_0(\kappa)}\bigr\|_{\textrm{HS}}^2
	= \frac{1}{2\pi} \int_0^\infty\!\!\int_0^\infty  \frac{|\widehat q(\xi-\eta)|^2\,d\eta\,d\xi}{(\eta+\kappa)(\xi+\kappa)^2}
	\lesssim \kappa^{-1} \| q \|_{H^{-1/2}_\kappa}^2 . 
$$ 
Boundedness on $L^2$ of the second factor on RHS\eqref{root R naughty}, for $\kappa$ sufficiently large, follows from \eqref{R est} and \eqref{E:Hs mult}.

The spectral theorem already guarantees that the mapping defined in \eqref{mero QF} is meromorphic off the essential spectrum provided that the vector $f$ belongs to the quadratic form domain of the resolvent, which is to say, the dual of the quadratic form domain.  In this way, we see that the argument could be expanded beyond $f\in H^s_+$ to $f\in H^{-1/2}_+$. 
\end{proof}

Clearly, \eqref{big k0 0} is implied by the simpler condition
\begin{equation}\label{big k0 0000}
q\in \BA:=\{ \text{real-valued }q\in H^s  : \norm{q}_{H^s} \leq A \} \qtq{and} \kappa \geq C_s \bigl( 1 + A \bigr)^{\frac1{2\eps}}.
\end{equation}
However, we will need to continue with the more complicated formulation in order to close a bootstrap argument in the proof of Lemma~\ref{L:beta boot}.

The conditions \eqref{big k0 0} and \eqref{big k0 0000} guarantee the constructive invertibility of $\mc L+\kappa$ via the series \eqref{R series}.   In this regard, they cannot be substantially improved; this can be easily seen by considering the family of solitons \eqref{soliton Qc}.  Indeed, when $q=Q_c$, the operator $\mc L$ has an eigenvalue at $-c/2$ with eigenvector $(cx+i)^{-1}$.  By comparison, the $H^s_c$ norm of $Q_c$ is comparable to $c^{2\eps}$.

Our next lemma will be needed for the proof of Lemma~\ref{L:best n dot}.

\begin{lemma}\label{L:id0}
For $f,g\in H^{s+1}_+$ we have
\begin{equation}\label{L id}
C_+\big( f\,\ol{\mc{L}g} - \ol{g}\,\mc{L}f \big) = iC_+\big(f\ol{g}\big)' + f[1-C_-](q_+\ol{g}) .
\end{equation}
\end{lemma}

\begin{proof}
We compute
\begin{align*}
C_+\big( f\,\ol{\mc{L}g} - \ol{g}\,\mc{L}f \big) 
&= C_+\big\{ if\ol{g}' - fC_-(q\ol{g}) + if'\ol{g} + \ol{g}C_+(qf) \big\} \\
&= iC_+\big(f\ol{g}\big)'  + C_+\bigl\{ \ol{g}qf - fC_-(q\ol{g})  \big\}\\
&= iC_+\big(f\ol{g}\big)'  + C_+\bigl\{ f [1-C_-] (q\ol{g}) \big\}\\
&=  iC_+\big(f\ol{g}\big)' + f [1-C_-] (q\ol{g}).
\end{align*}
Finally, noting that the presence of $[1-C_-]$ allows us to replace $q$ by $q_+$ in the last term, we obtain \eqref{L id}.
\end{proof}

The remainder of this section concerns the interaction between the Lax operator $\mc L$ and the operator of multiplication by $x$.  To do this, we must first describe how multiplication by $x$ can be interpreted as an operator on the Hardy space $L^2_+(\R)$.  It cannot be realized as a selfadjoint operator!  

In order to make sense of multiplication by $x$ on $L^2_+(\R)$, it is easiest to employ Fourier transformation and the theory of semigroups.  We wish to make sense of $i\partial_\xi$ as an operator on a half-line.  The naturally associated semigroups $e^{t\partial}$ and  $e^{-t\partial}$ represent translation to the left (with truncation to $[0,\infty)$) and translation to the right (padded with zero), respectively.  Each gives rise to a strongly continuous semigroup and we may then define multiplication by $x$ as the associated generator.

We adopt the left shift as the basis for our notion of multiplication by $x$ since this leads to an operator with larger domain.  We record here some basic results of the general theory presented, for example, in \cite[\S X.8]{MR0493420}:

\begin{lemma}\label{L:X}
Let $X$ denote the (unbounded) operator on $L^2_+(\R)$ with
$$
D(X) = \bigl\{ f\in H^s_+ (\R): \widehat f\in H^1\bigl([0,\infty)\bigr) \bigr\} \qtq{and} \widehat{Xf}(\xi) =i\tfrac{d{\widehat f}}{d\xi}(\xi)
	\qtq{for} f\in D(X).
$$
Then $iX$ is maximally accretive and is the generator of the semigroup
$$
e^{-itX} f = \tfrac{1}{\sqrt{2\pi}} \int_0^\infty e^{i\xi x} \widehat f(\xi + t)\,d\xi = C_+\bigl( e^{-itx} f \bigr) 
$$
defined on $L^2_+(\R)$.   The spectrum of $X$ consists of the closed lower half-plane.  For $\Im z>0$, the resolvent is given by
$$
(X-z)^{-1} f = \tfrac{f(x)-f(z)}{x-z}
$$
where $f(z)$ is defined via analytic continuation to the upper half-plane.
\end{lemma}

Each $z$ with $\Im z<0$ is actually an eigenvalue of $X$ with eigenvector $1/(x-z)$.

The adjoint $X^*$ of $X$ is the generator of right translations. Its domain is smaller, being comprised of those $f\in L^2_+$ such that $\wh f\in H^1_0([0,\infty))$.  For such $f$, we have $X^*f=Xf$.

Functions in the domain of $X^*$ are absolutely integrable and integrate to zero.  Typical functions in $D(X)$ are not absolutely integrable: their Fourier transform has a jump discontinuity at the origin.  Nevertheless, they are `conditionally integrable' with a value representing half the height of the jump.  For example, using the Poisson integral formula, we have
\begin{align}\label{I+ defn0}
\lim_{y\to\infty} \pi y f(iy) =  \lim_{y\to\infty}  \int \frac{y^2}{x^2+y^2} f(x) \,dx = \lim_{\xi\downarrow 0} \tfrac{\sqrt{2\pi}}{2}\widehat f(\xi) 
\end{align}
for all $f\in D(X)$.  Following earlier models, such as \cite{Gerard2022,Sun2021}, we define a linear functional representing twice this value: For $f\in D(X)$,
\begin{align}\label{I+ defn}
I_+(f) := \lim_{y\to\infty} 2\pi y f(iy) =  \lim_{y\to\infty} \bigl\langle \chi_y, f \bigr\rangle
	= \lim_{\xi\downarrow 0} \sqrt{{2\pi}} \widehat f(\xi) \ \text{with}\  \chi_y(x) = \tfrac{iy}{x+iy}.
\end{align}
One may regard the middle expression in \eqref{I+ defn} as originating from splitting the Poisson kernel into its Hardy-space components, or as simply the Cauchy integral formula.

Another form of the Cauchy integral formula, which follows from the above, is
\begin{align}\label{CIT}
f(z)  = \tfrac{1}{2\pi i} I_+\bigl( (X-z)^{-1} f \bigr) = \lim_{y\to\infty} \tfrac{1}{2\pi i} \bigl\langle \chi_y, (X-z)^{-1} f \bigr\rangle
\end{align}
valid for all $f\in L^2_+$ and $\Im z>0$.

\begin{lemma}
If $q\in H^{\infty}(\R)$ and $f\in D(X)$, then $C_+(qf)\in D(X)$,
\begin{equation}\label{comm 1}
[X,C_+q] f = \tfrac{i}{2\pi} q_+ I_+(f),  \qtq{and} [X,\mc L ] f = i - \tfrac{i}{2\pi} q_+ I_+(f)  .
\end{equation}
\end{lemma}

This expresses the well-known facts that the commutator of $X$ with a Toeplitz operator, such as $f\mapsto C_+(qf)$, is a rank-one operator, while that of $\partial$ and $X$ is the identity.   These observations follow from straightforward computations in Fourier variables; see, for example, \cite[Lem.~3.1]{Sun2021} for details.

%%%%%%%%%%%%%%%%%%%%%%%%%%%%%%%%%%%%%%%%%%%%%%%%%%%%%%%%%%%%%%%%%%%%%%%%%%
\section{A new gauge}\label{S:gauge} %%%%%%%%%%%%%%%%%%%%%%%%%%%%%%%%%%%%%
%%%%%%%%%%%%%%%%%%%%%%%%%%%%%%%%%%%%%%%%%%%%%%%%%%%%%%%%%%%%%%%%%%%%%%%%%%

In this section, we analyze the function $m=m(\kappa,q)$, which was introduced as the solution to the modified eigenvalue equation
\begin{equation}\label{m eq}
m' = -i\kappa m + i C_+[q(m+1)],
\end{equation}
or what is equivalent, $(\mc L + \kappa) m = q_+$.

As we will see in this section, this object plays many roles in the theory of \eqref{BO}.  The title of the section, however, reflects our new and crucial application of $m$ as a gauge transformation, replacing $q$ as the dynamical variable.

First we must show that such a function exists and derive its basic properties.  This certainly requires restrictions on $\kappa$; most naturally, we  should avoid the spectrum of $\mc L$.  For our purposes, it will suffice to consider $\kappa$ large and positive.  For the moment, we will continue to use the approach of Proposition~\ref{P:Lax} by requiring
\begin{equation}\label{big k0 1}
\kappa \geq C_s \bigl( 1 + \|q\|_{H^s_\kappa} \bigr)^{\frac1{2\eps}},
\end{equation}
for a suitable large constant $C_s$ and $\eps$ as in \eqref{seps}.  Once we have developed sufficient preliminaries, we will adopt the more permanent solution expounded in Convention~\ref{C:Akmn} below.

\begin{proposition}[Existence and Uniqueness]\label{P:m EU}
There is a constant $C_s\geq 1$ so that the following hold: For any $q\in H^s$ and $\kappa$ satisfying \eqref{big k0 1}, there is a unique $m\in H^{s+1}_+$ solving~\eqref{m eq}.  It is given by
\begin{equation}\label{m series}
m(x;\kappa,q) := R(\kappa,q) q_+ = R_0(\kappa) \sum_{\ell\geq 1} [C_+q R_0(\kappa)]^{\ell-1} q_+ 
\end{equation}
and satisfies
\begin{align}\label{m est}
\norm{m}_{H^{s+1}_\kappa} \lesssim \norm{q}_{H^s_\kappa}, \quad \|m\|_{L^\infty}<1,   \qtq{and} \norm{m}_{H^s} \lesssim \kappa^{-1} \norm{q}_{H^s} .
\end{align}
Moreover, if $q(x)$ belongs to $H^\infty$ then so too does $m(x)$.
\end{proposition}

\begin{proof}
Proposition~\ref{P:Lax} guarantees the existence of $C_s\geq 1$ so that $\mc L+\kappa$ is invertible whenever \eqref{big k0 1} holds; indeed, this is demonstrated by proving the convergence of the series \eqref{R series}.  This verifies the existence and uniqueness of $m$, as well as formula \eqref{m series}.  In fact, by Proposition ~\ref{P:Lax} we see that $m$ is unique not only in $H^{s+1}$ but also in the larger space $H^{1/2}$. 

The first estimate in \eqref{m est} follows directly from \eqref{R est}.  Using this we also see that
$$
	\| (\mc L_0 +\kappa) m \|_{H^{s}_\kappa} = \| m \|_{H^{s+1}_\kappa} \lesssim \norm{q}_{H^s_\kappa}.
$$
Writing $m=R_0(\kappa)\bigl[ q_+ + C_+ q R_0(\kappa) (\mc L_0 +\kappa) m\bigr]$ and using \eqref{key est}, we deduce that
$$
	\| m\|_{H^s} \lesssim \kappa^{-1}\bigl[\|q\|_{H^s} + \kappa^{-2\eps} \|q\|_{H^s} \| q \|_{H^{s}_\kappa}\bigr].
$$
The last estimate in \eqref{m est} now follows from our assumption on $\kappa$.

Using Cauchy--Schwarz in the frequency variable and \eqref{big k0 1}, we find
\begin{align*}
\|m\|_{L^\infty}\lesssim \kappa^{-2\eps} \norm{m}_{H^{s+1}_\kappa} \lesssim C_s^{-2\eps}.
\end{align*}
The middle bound in \eqref{m est} follows by choosing $C_s$ large enough.
 		
Finally, we turn to the statement that $q\in H^\infty$ implies $m\in H^\infty$.  By uniqueness, the $m$ associated to a translated potential is simply given by the translation of $m$:
\begin{equation}\label{m id}
m(x+h;\kappa,q) = m(x;\kappa,q(\cdot+h)) \quad\text{for all }h\in\R .
\end{equation}
For any integer $\sigma\geq1$, we use \eqref{m id} and \eqref{m series} to see that
\begin{equation*}
m^{(\sigma)} = \sum_{\ell\geq 1} \sum_{\substack{ \sigma_1,\dots,\sigma_\ell\geq 0 \\ \sigma_1+\dots+\sigma_\ell = \sigma }}  \binom{\sigma}{\sigma_1 \dots \sigma_\ell}R_0C_+q^{(\sigma_1)} R_0C_+q^{(\sigma_2)} \cdots R_0 q^{(\sigma_\ell)}_+ 
\end{equation*}
and so deduce that
\begin{equation*}
\| m^{(\sigma)} \|_{H^{s+1}_\kappa} \leq \sum_{\ell\geq 1} \ell^\sigma
	\sup_{\substack{ \sigma_1,\dots,\sigma_\ell\geq 0 \\ \sigma_1+\dots+\sigma_\ell = \sigma }}
	\| q^{(\sigma_\ell)}\|_{H^s_\kappa} \prod_{i=1}^{\ell-1}\| C_+q^{(\sigma_i)} R_0 C_+\|_{H^s_\kappa\to H^s_\kappa} .
\end{equation*}
For any $1\leq i\leq\ell-1$ with $\sigma_i=0$, we apply \eqref{big k0 0'}. This leaves at most $\sigma$ many of the coefficients $\sigma_1,\ldots,\sigma_{\ell-1}$ that may be non-zero.  We estimate these remaining factors with \eqref{key est}, combine them with $q^{(\sigma_\ell)}$, and use that
$$
\prod_{j=1}^J \| q^{(\tilde \sigma_j)} \|_{H^s_\kappa} \leq \| q \|_{H^s_\kappa}^{J-1}\| q \|_{H^{s+\sigma}_\kappa}
	\qtq{whenever} \tilde\sigma_1+\cdots+\tilde\sigma_J = \sigma .
$$
In this way, we obtain 
\begin{equation}\label{quant H infty}
\snorm{ m^{(\sigma)} }_{H^{s+1}_\kappa}
\lesssim \sum_{\ell=1}^\infty \ell^\sigma 2^{\sigma-\ell} \Bigl( 1 + \| q\|_{H^{s}_\kappa}\Bigr)^\sigma \| q \|_{H^{s+\sigma}_\kappa}  < \infty
\end{equation}
for any $q\in H^{\infty}$ and any $\kappa$ satisfying \eqref{big k0 1}.
\end{proof}

\begin{proposition}[Diffeomorphism property]\label{P:m Diff}
There is a constant $C_s\geq 1$ so that for any $A>0$ and $\kappa$ satisfying
\begin{equation}\label{big k0 1111}
\kappa \geq C_s \bigl( 1 + A \bigr)^{\frac1{2\eps}},
\end{equation}
the mapping $q\mapsto m$ is a diffeomorphism from $\BA$ into $H^{s+1}$.
\end{proposition}

\begin{proof}
Initially, we choose $C_s$ as required by Propositions~\ref{P:Lax} and \ref{P:m EU}.  For $g\in H^s$, the resolvent identity implies
\begin{equation}\label{dm}
dm|_q(g)= \frac{d}{d\theta} m(x;\kappa,q+\theta g) \bigg|_{\theta=0}= R(\kappa,q) \bigl[ (m+1)C_+g\bigr],
\end{equation}
which for $q \equiv 0$ reduces to
\begin{equation}\label{dm0}
dm|_0(g)= R_0(\kappa)C_+ g .
\end{equation}
Taking a supremum over $g\in H^s_\kappa$ and using \eqref{R est}, \eqref{E:Hs mult}, and \eqref{m est}, we deduce that
\begin{equation}
\bnorm{ dm|_q - dm|_0 }_{H^s_\kappa \to H^{s+1}_\kappa} \lesssim \kappa^{-2\eps} \|q\|_{H^s_\kappa} \lesssim C_s^{-2\eps},
\label{diffeo 1}
\end{equation}
uniformly for $q\in \BA$ and $\kappa$ satisfying \eqref{big k0 1111}.
	
On the other hand, for $f\in H^{s+1}_+$ we have
\begin{equation*}
\bnorm{ (dm|_0)^{-1}(f) }_{H^s_\kappa}^2 \leq 2\norm{f}_{H^{s+1}_\kappa}^2 ,
\end{equation*}
and so
\begin{equation}\label{diffeo 2}
\bnorm{ (dm|_0)^{-1} }^{-1}_{H^{s+1}_\kappa \to H^s_\kappa}\geq \tfrac{1}{\sqrt{2}}.
\end{equation}
	
Combining~\eqref{diffeo 1} and~\eqref{diffeo 2}, we see that enlarging $C_s$ if necessary,
\begin{equation*}
\bnorm{ dm|_q - dm|_0 }_{H^s_\kappa \to H^{s+1}_\kappa} \leq \tfrac12 \bnorm{ (dm|_0)^{-1} }^{-1}_{H^{s+1}_\kappa \to H^s_\kappa} .
\end{equation*}
Using this as input for the standard contraction-mapping proof of the inverse function theorem, we conclude that we may pick $C_s$ sufficiently large so that
\begin{equation*}
q \mapsto m \quad\text{is a diffeomorphism from }\{ q: \norm{q}_{H^s_\kappa} \leq A \} \text{ into }H^{s+1}_\kappa
\end{equation*}
for all $\kappa$ satisfying \eqref{big k0 1111}. As the domain $\{ q: \norm{q}_{H^s_\kappa} \leq A \}$ includes the smaller domain $\BA$, this completes the proof.
\end{proof}

\begin{proposition}\label{P:beta}
There is a constant $C_s\geq 1$ so that for $q\in H^s$ and $\kappa$ satisfying \eqref{big k0 1}, the quantity
\begin{equation}\label{beta}
\beta(\kappa;q): = \int q(x) m(x;\kappa,q)\,dx = \int q(x) \ol{m}(x;\kappa,q)\,dx = \bigl\langle q_+, (\mc L+\kappa)^{-1} q_+ \bigr\rangle
\end{equation}
is finite and real-valued.  For such $\kappa$, this is a real-analytic function of $q$ with
\begin{equation}\label{dbeta}
\tfrac{\del\beta}{\del q} =  m + \ol m +  |m|^2 
\end{equation}
and satisfies
\begin{equation}\label{beta est}
C_s^{-1} \norm{q}_{H^s_\kappa}^2 \leq \int_{\kappa}^\infty \vk^{2s} \beta(\vk;q)\,d\vk \leq C_s\norm{q}_{H^s_\kappa}^2.
\end{equation}
Lastly, for each $q\in H^s$, the mapping $z\mapsto\beta(z;q)$ extends to a meromorphic function on $\{z\in\C:\Re z >0\}$.
\end{proposition}

\begin{proof}
Proposition~\ref{P:m EU} shows that for a suitable choice of $C_s$, we are guaranteed that $m=(\mc L+\kappa)^{-1} q_+$ exists and lies in $H^{s+1}$.  This in turn means that $m$ defines a bounded linear functional on $H^s$ under the natural pairing:
\begin{align*}
\int q(x) m(x;\kappa,q)\,dx = \langle q, m\rangle = \langle q_+, m\rangle = \bigl\langle q_+, (\mc L+\kappa)^{-1} q_+ \bigr\rangle.
\end{align*}
As $\mc L$ is a selfadjoint operator, this quantity is real.  This proves all the identities stated in \eqref{beta}.  The possibility of extending this to a meromorphic function in the right half-plane follows from Proposition~\ref{P:Lax} and the final representation in \eqref{beta}.

The fact that $\beta$ is a real-analytic function of $q$ follows from the convergence of the series \eqref{m series}.
Using the functional derivative~\eqref{dm} of $m$, we see that
\begin{align*}
d\beta|_q(f) &= \int  f m  + q \cdot R(\kappa,q)[(m+1)f_+] \,dx  \\
&= \int f m + \ol{R(\kappa,q)q_+} \cdot (m+1)f \,dx\\
&= \int [m +\ol{m}(m+1) ] f \,dx,
\end{align*}
which yields \eqref{dbeta}.

It remains to prove \eqref{beta est}.  As we will see, this may require us to increase $C_s$.  
Let us first examine a quadratic approximation of the central object.  By Plancherel and Fubini,
\begin{align}\label{equicty 5}
\int_{\kappa}^\infty \!\vk^{2s} \langle q_+, R_0(\vk)q_+\rangle \,d\vk
&= \int_{0}^\infty \! \int_{\kappa}^\infty \! \vk^{2s} \frac{|\wh{q}(\xi)|^2}{\xi+\vk}\,d\vk\,d\xi
\simeq_s \norm{q_+}_{H^{s}_\kappa}^2.
\end{align}

This leaves us to control the remainder. Using the duality of $H^{s+1}_\vk$ and $H^{-(s+1)}_\vk$ and \eqref{R est}, we have
\begin{align*}
\bigl\langle q_+, \bigl[R(\vk)-R_0(\vk)\bigr]q_+\bigr\rangle
&\lesssim \vk^{-2\eps} \norm{q_+}_{H^{-(s+1)}_\vk} \|q\|_{H^s_\vk}^2 \lesssim \vk^{-1-2s-2\eps} \, \|q\|_{H^s_\vk}^3
\end{align*}
for any $q\in H^s$ and $\vk\geq\kappa$.  In this way,  we deduce that
\begin{align*}
\int_{\kappa}^\infty \!\vk^{2s} \bigl\langle q_+, \bigl[R(\vk)-R_0(\vk)\bigr]q_+\bigr\rangle \,d\vk 
&\lesssim \|q\|_{H^s_\kappa} \int_{\kappa}^\infty \!  \vk^{-1-2\eps} \|q\|_{H^s_\vk}^2 \, d\vk  \\ 
&\lesssim_s \kappa^{-2\eps} \norm{q_+}_{H^{s}_\kappa}^3.
\end{align*}
Combining this with \eqref{equicty 5} and taking $C_s$ sufficiently large, we conclude that \eqref{beta est} holds.
\end{proof}

Propositions~\ref{P:Lax}, \ref{P:m EU}, \ref{P:m Diff}, and~\ref{P:beta} show important quantitative properties of $m$ and $\beta$ under the restriction that $\kappa$ is large enough, depending on the size of $q$.  Ultimately, we wish to consider trajectories in $H^s$ rather than individual $q\in H^s$ and so we must account for the possibility that the $H^s$ norm of solutions may grow.  

For the flows of interest to us, $\beta$ is conserved and our next lemma shows how this fact can be leveraged to control the growth and equicontinuity of trajectories.  Indeed, this will lead to an alternate proof of Theorem~\ref{T:Talbut} based on $\beta(\kappa;q)$, rather than the perturbation determinant; see Corollary~\ref{C:Qstar}.  

One may wonder what conservation of $\beta$ means if the $\kappa$-interval on which it is defined depends on $q$ itself.  It was to address this irritation that we demonstrated that $\beta$ can be interpreted as a meromorphic function on the right half-plane.  Evidently, if $\beta(\kappa;q_0)$ and $\beta(\kappa;q_1)$ agree on some ray $\kappa\geq \kappa_1$ then they agree throughout the right half-plane (as meromorphic functions).

\begin{lemma}\label{L:beta boot}
Given $A>0$ and $Q\subset \BA$, let 
\begin{align*}
Q_{**}= \Bigl\{ q(b) \Big|\, q:[a,b]\to H^s \text{ is continuous, } q(a)\in Q, \text{ and } \beta(z;q(t))\equiv \beta(z;q(a))\Bigr\},
\end{align*}
where $\beta(z;q(t))\equiv \beta(z;q(a))$ indicates equality as meromorphic functions on the right-half plane for all $t\in[a,b]$.
Then $Q_{**}$ is bounded; indeed, for $C_s$ as in Proposition~\ref{P:beta},
\begin{align}\label{beta APB}
\sup_{q\in Q_{**}}\norm{q}_{H^s}  \lesssim C_s^{1+|s|} \bigl( 1 + 2 C_sA \bigr)^{\frac{2|s|}{1-2|s|}} A .
\end{align}
Moreover, if $Q$ is $H^s$-equicontinuous, then so too is $Q_{**}$.
\end{lemma}

\begin{proof}
Given $q(a)\in Q$, consider
\begin{equation}\label{kappa boot0}
\kappa \geq C_s \bigl( 1 + 2 C_s\|q(a)\|_{H^s_\kappa} \bigr)^{\frac1{2\eps}}.
\end{equation}
For such $\kappa$ and any time interval $[a,T]$ on which
\begin{equation}
\|q(t)\|_{H^s_\kappa} \leq 2 C_s \|q(a)\|_{H^s_\kappa}, 
\end{equation}
we may apply the equivalence \eqref{beta est} to deduce that 
\begin{equation}\label{beta est'}
\norm{q(t)}_{H^s_\kappa} \leq  C_s \norm{q(a)}_{H^s_\kappa} .
\end{equation}
A standard bootstrap argument then shows that \eqref{beta est'} holds on the entire time interval $[a,b]$.

As $Q\subset \BA$, the hypothesis \eqref{kappa boot0} is satisfied for every $q(a)\in Q$ with
\begin{equation}
\kappa = C_s \bigl( 1 + 2 C_sA \bigr)^{\frac1{2\eps}}.
\end{equation}
Using this choice, we obtain
\begin{equation}
\kappa^s\sup_{q\in Q_{**}}\norm{q}_{H^s} \leq \sup_{q\in Q_{**}}\norm{q}_{H^s_\kappa} \leq  C_s \sup_{q\in Q}\norm{q}_{H^s_\kappa}
	\leq  C_s \sup_{q\in Q}\norm{q}_{H^s}
\end{equation}
and thence \eqref{beta APB}.

The equicontinuity of $Q_{**}$ follows from that of $Q$ by Lemma~\ref{L:equi} and \eqref{beta est'}.
\end{proof}

We have now proven all the results we need that require us to adjust the constant $C_s$ and so are ready to adopt our unified notion of $\kappa$ being sufficiently large.  Moreover, Lemma~\ref{L:beta boot} allows us do this in a way that ensures $\kappa$ remains sufficiently large for all trajectories of interest to us.  We also take the opportunity to introduce the abbreviated notation \eqref{mn abbrev}.
 
\begin{convention}\label{C:Akmn}
Given $A>0$, we choose $\kappa_0=\kappa_0(A)$ large enough so that the hypotheses of Propositions~\ref{P:Lax}, \ref{P:m EU}, \ref{P:m Diff}, and~\ref{P:beta} are all met whenever $\kappa\geq\kappa_0$ and $q\in (\BA)_{**}$. Moreover, for such $q\in (\BA)_{**}$, we write
\begin{equation}\label{mn abbrev}
m := m(x;\kappa,q) \qtq{and} n := m(x;\vk,q) 
\end{equation}
and demand that $\kappa,\vk\geq \kappa_0(A)$.
\end{convention}

\begin{lemma}[Equicontinuity properties of $m$]
Given $A>0$ and an equicontinuous set $Q\subset \BA$, we have 
\begin{equation} \label{equicty 6}
\lim_{ \kappa\to\infty}\ \sup_{q\in Q}\norm{m}_{H^{s+1}_\kappa} = 0 \qtq{and}
\lim_{ \kappa\to\infty}\ \sup_{q\in Q}\norm{ \mc{L} R(\kappa,q) n}_{H^{s+1}} =0
\end{equation}
for all $\kappa,\vk\geq \kappa_0(A)$ as dictated by Convention~\ref{C:Akmn}.
\end{lemma}

\begin{proof}
The first claim in~\eqref{equicty 6} follows immediately from the estimate~\eqref{m est} and the characterization (ii) of equicontinuity from Lemma~\ref{L:equi}.
	
For the second claim in~\eqref{equicty 6}, we write
\begin{equation*}
R(\kappa,q) n = R(\kappa,q) R(\vk,q) q_+ = R(\vk,q) R(\kappa,q) q_+ = R(\vk,q) m.
\end{equation*}
Commuting $\mc{L}$ and $R(\vk,q)$ and using the estimates~\eqref{L est} and~\eqref{R est} for these operators, we find
\begin{align*}
\norm{ \mc{L} R(\kappa,q) n}_{H^{s+1}}
= \norm{ R(\vk,q) \mc{L} m}_{H^{s+1}}
&\lesssim \norm{ \mc{L} m }_{H^{s}}\lesssim (1+\|q\|_{H^s})\norm{ m}_{H^{s+1}} .
\end{align*}
The right-hand side above tends to zero as $\kappa\to \infty$ by the first claim in~\eqref{equicty 6}.
\end{proof}

\begin{proposition} [Dynamics]
For an $H^\infty$ solution $q(t)$ to~\eqref{BO},
\begin{align}
\tfrac{d}{dt} q_+ &= \mc{P} q_+ = - i q_+'' - 2C_+(qq_+)' + 2q_+ q'_+ \label{q_+ dot BO} \\
\tfrac{d}{dt} m &= - im'' - 2C_+([q-q_+] m)' - 2q_+m' \label{m dot BO} \\
\tfrac{d}{dt} \beta(\kappa) &= 0 .\label{beta dot BO} 
\end{align}
Here $\mc P$ is given by \eqref{P defn} and Convention~\ref{C:Akmn} applies.
\end{proposition}

\begin{proof}
As $q^2 = 2qq_+ - (q_+)^2 +(q-q_+)^2$, so
\begin{equation*}
C_+( 2qq') = C_+\bigl( 2qq_+ - (q_+)^2 +(q-q_+)^2 \bigr)' = 2 C_+(qq_+)' - 2 q_+' q_+ 
\end{equation*}
and consequently,
\begin{align}\label{surprise!}
\mc{P}q_+ = - iq''_+ - 2C_+(qq_+)' + 2q_+'q_+=  C_+\bigl(\h q'' - 2qq'\bigr) = \tfrac{d}{dt} q_+\,.
\end{align}
This proves \eqref{q_+ dot BO}.

By virtue of the Lax pair representation, \eqref{m series}, and \eqref{q_+ dot BO},
\begin{align*}
\tfrac{d}{dt} m &= [\mc{P}, R(\kappa) ] q_+ + R(\kappa) \mc P q_+ = \mc P m  .
\end{align*}
From here, \eqref{m dot BO} follows easily:
\begin{equation*}
\tfrac{d}{dt} m= \mc{P}m = - im'' - 2C_+(qm)' + 2q'_+ m = - im'' - 2C_+([q-q_+] m)' - 2q_+m' .
\end{equation*}

From the final representation in \eqref{beta} and \eqref{q_+ dot BO}, we deduce that
\begin{align*}
\tfrac{d}{dt} \beta(\kappa) &= \bigl\langle \mc P q_+, R(\kappa) q_+\bigr\rangle  + \bigl\langle  q_+, [\mc P, R(\kappa) \:\!] q_+\bigr\rangle
	+ \bigl\langle  q_+, R(\kappa) \mc P q_+\bigr\rangle \\
	&=\bigl\langle \mc P q_+, R(\kappa) q_+\bigr\rangle +\bigl\langle  q_+, \mc P R(\kappa) q_+\bigr\rangle.
\end{align*}
This vanishes because $\mc P$ is an antisymmetric operator on the Hardy space $L^2_+$.  Thus  \eqref{beta dot BO} holds.
\end{proof}

We pause to note that the right-hand side of \eqref{m dot BO} extends continuously (in $H^{-2}$, for example) from $q\in H^\infty$ to $q\in H^s$.  For the first term, this follows from Proposition~\ref{P:m Diff}. For the second, we also apply Lemma~\ref{L:Hs mult}.  In the third term, $q_+$ and $m'$ do not have enough Sobolev regularity to make sense of the product.  Here it is essential that both are holomorphic, which allows us to use Lemma~\ref{L:+*+}.

Employing the Stone--Weierstrass (on a compactified interval $[-E_0,\infty]$) and spectral theorems, it is not difficult to deduce from \eqref{beta dot BO} and \eqref{beta} that for any measurable function $F:\R\to\R$ satisfying
\begin{equation}\label{sc law}
\bigl| F(E) \bigr| \lesssim (1+|E|)^{-1}, \qtq{the functional} q\mapsto
\bigl\langle q_+,F(\mc L) q_+ \bigr\rangle
\end{equation}
defines a conserved quantity for the \eqref{BO} flow.   This is interesting because it provides a clear way of separating out the contribution of any embedded point or singular continuous spectrum to the conserved quantities.  We know of no analogue of this fact in the much-studied KdV equation, for example.

Our next lemma presents other ways in which $m$ and $\beta$ are related, beyond the definition \eqref{beta}. 

\begin{lemma}\label{L:other beta}
Under Convention \ref{C:Akmn},
\begin{align}\label{beta2}
\int \ol{m} n \,dx = \int m \ol{n} \,dx =  \bigl\langle (\mc L+\vk)^{-1} q_+, (\mc L+\kappa)^{-1} q_+ \bigr\rangle = - \frac{\beta(\kappa)-\beta(\vk)}{\kappa-\vk}
\end{align}
for any $q\in \BA$ and distinct $\vk,\kappa\geq\kappa_0(A)$.  In the periodic case, we also have
\begin{align}\label{beta3}
\kappa \int \ol{m} \,dx = \kappa \int m \,dx =  \int qm \,dx + \int q \,dx = \beta(\kappa) + \int q \,dx 
\end{align}
and, writing $1$ for the constant function,
\begin{align}\label{beta4}
\langle 1, (\mc L + \kappa)^{-1} 1\rangle = \kappa^{-1} + \kappa^{-2} \beta(\kappa;q) + \kappa^{-2}\int q\,dx .
\end{align}
\end{lemma}

\begin{proof}
The identities \eqref{beta2} are evident from the definitions of $m$, $n$ and
$$
(\mc L+\vk)^{-1} (\mc L+\kappa)^{-1} = (\mc L+\kappa)^{-1} (\mc L+\vk)^{-1}
	= \frac{-1}{\kappa-\vk}\bigl[ (\mc L+\kappa)^{-1} - (\mc L+\vk)^{-1}\bigr].
$$

The identities \eqref{beta3} follow by integrating \eqref{m eq} over the circle and using that $\beta(\kappa)$ is real-valued.

As $\mc L_0 1=0$, the resolvent identity gives
\begin{align*}
(\mc L + \kappa)^{-1} 1  =  (\mc L_0 + \kappa)^{-1} 1 + (\mc L + \kappa)^{-1}C_+ q(\mc L_0 + \kappa)^{-1} 1
	= \kappa^{-1} (1 + m) .
\end{align*}
Thus \eqref{beta4} follows from \eqref{beta3}.
\end{proof}

\begin{remark}
In Lemma~\ref{L:prop x}, we will show that $m+\ol m \in L^1(\R)$ when $\langle x\rangle q\in L^2(\R)$.
Mimicking the arguments above yields the following synthesis of \eqref{beta2} and \eqref{beta3}:
\begin{equation}\label{betaX}
\kappa \int m+ \ol{m} +|m|^2 \,dx = 2\beta(\kappa) - \kappa\frac{\partial\beta}{\partial\kappa}+ 2\int q \,dx,
\end{equation}
valid both on the line and on the circle.
\end{remark}

Our next result is an important identity, which first appeared as \cite[Eq.~(58)]{Kaup1998a}.  In that paper, it was used as a stepping stone in the calculation of Poisson brackets between certain scattering-theoretic data, defined for smooth rapidly decreasing $q$.  Our first application of this identity will be to demonstrating Poisson commutativity of $\beta(\kappa)$ at differing spectral parameters.  In subsection~\ref{ss:Bock} we will also see that it provides an important key for unlocking the significance of the Bock--Kruskal transformation.

\begin{lemma}\label{L:id}
For $q\in H^\infty(\R)$ we have
\begin{equation}\label{m ODE 1}
\begin{aligned}
\h (m\ol{n}+m+\ol{n})' + i(m+1)\ol{n}' -im'(\ol{n}+1) - 2q (m+1)(\ol{n}+1)& \\
{}+ i(\kappa - \vk)\h (m\ol{n}+m+\ol{n}) + (\kappa + \vk)(m\ol{n}+m+\ol{n}) &= 0 ,
\end{aligned}
\end{equation}
subject to Convention \ref{C:Akmn}.  For $q\in H^\infty(\T)$, this expression need not vanish; however, it is a real-valued constant function:
\begin{equation}
\text{\upshape LHS\eqref{m ODE 1}} = \vk \int m \,dx + \kappa\int \ol n \,dx .
	\label{m ODE 1T}
\end{equation}
\end{lemma}

\begin{proof}
Employing equation~\eqref{m eq} to eliminate $m'$ and $\ol{n}'$, we obtain
\begin{align*}
\text{LHS}\eqref{m ODE 1}
={} &\kappa ( 1+i\h )\ol{n} + \vk ( 1 -i\h )m + (1+i\h) \big[ (\ol{n}+1)C_+(q(m+1)) \big] \\
&+  (1-i\h)  \big[ (m+1)C_-(q(n+1)) \big] - 2q(m+1)(\ol{n}+1) .
\end{align*}
Thence, using the operator identity $2=(1+i\h)+(1-i\h)$ on the last term yields
\begin{align*}
\text{LHS}\eqref{m ODE 1}
={} &\kappa ( 1+i\h )\ol{n} + \vk ( 1 -i\h )m + (1+i\h) \big[ (\ol{n}+1)[C_+-1] (q(m+1)) \big] \\
&+  (1-i\h)  \big[ (m+1)[C_- -1] (q(n+1)) \big] .
\end{align*}
Consideration of the Fourier supports shows that the last two terms vanish in either geometry.  The first two terms vanish on the line but reduce to RHS\eqref{m ODE 1T} in the circle case.  The fact that this constant is real (and generically nonzero) follows from \eqref{beta3}.
\end{proof}

\begin{lemma}\label{L:Poisson 2}
Under Convention \ref{C:Akmn},
\begin{equation}\label{poisson 2}
\{ \beta(\kappa) , \beta(\vk) \} = 0 \qtq{and} \{ P , \beta(\vk) \} = 0 
\end{equation}
as functions on  $\BA$ and $\BA\cap H^\infty$, respectively.
\end{lemma}

\begin{proof}
By \eqref{Poisson}, \eqref{dbeta}, and integration by parts, 
\begin{equation*}
\{ \beta(\kappa),\beta(\vk) \}
= \int \big( |m|^2 + m + \ol m \big) \big( |n|^2 + n + \ol n \big)'\,dx
= \tfrac12\!\int F(x)\,dx ,
\end{equation*}
where we adopt the notation
\begin{align}
F:= \bigl[ |m|^2 + m + \ol{m}\bigr] \bigl[|n|^2 + n + \ol{n}\bigr]'
	- \bigl[ |m|^2 + m + \ol{m}\bigr]' \bigl[|n|^2 + n + \ol{n}\bigr].
\end{align}
Proposition~\ref{P:m EU} shows that these expressions are all well-defined on $\BA$.

To continue, we rewrite $F$ as 
\begin{align*}
F=G+\ol{G}+K+\ol{K} \qtq{where} G &=  [ \ol{m}n + \ol{m} + n] \bigl[ (m+1)\ol{n}' - m'(\ol{n}+1)  \bigr]\\
\qtq{and} K &= (m -n)(\ol{m}+\ol{n})' . 
\end{align*}
We split $F$ in this way in order to take advantage of Lemma~\ref{L:id}, which shows that
\begin{equation}\label{m ODE 2222}
\begin{aligned} 
(m+1)\ol{n}' - m'(\ol{n}+1)  &= i \h (m\ol{n}+m+\ol{n})' -(\kappa - \vk)\h (m\ol{n}+m+\ol{n}) \\
&\quad{}  + i (\kappa + \vk - 2 q)(m\ol{n}+m+\ol{n})  - 2 i q - i c ,
\end{aligned}
\end{equation}
where the constant function $c$ denotes the value of LHS\eqref{m ODE 1} appropriate to each geometry. Recall that $c=0$ on $\R$ and is real on $\T$. Combining this identity with the antisymmetry of $\h$ and of $i\h \partial$, we find that
\begin{align}
\int\! (G + \overline G) \, dx = i\int (2q+c) [ m\ol{n} - \ol{m}n + m - \ol{m} - n + \ol{n}]   \,dx.
\end{align}
Using \eqref{beta}, \eqref{beta2}, and \eqref{beta3}, this further simplifies to
\begin{align}\label{G part}
\int\! (G + \overline G) \, dx = 2i \int q [ m\ol{n} - \ol{m}n]   \,dx.
\end{align}

On the other hand, integrating by parts and employing \eqref{m eq}, we obtain
\begin{align*}
\int (K + \ol{K})  \, dx &= 2 \int m \ol{n}' + \ol{m} n' \,dx \\
&= 2i \int (\vk - q) [m \ol{n} - \ol{m} n] \,dx - 2 i \int q[m-\ol{m}]\,dx.
\end{align*}
Using \eqref{beta} and \eqref{beta2}, this simplifies to
\begin{align}\label{K part}
\int (K + \ol{K})  \, dx &= - 2i \int q [m \ol{n} - \ol{m} n] \,dx .
\end{align}

Combining \eqref{G part} and \eqref{K part} gives $\int F =0$ and so proves the first identity in \eqref{poisson 2}.

To prove the commutativity of $\beta(\vk)$ and the momentum $P=\frac12 \int q^2\,dx$, we use the functional derivative~\eqref{dbeta} for $\beta$ to compute
\begin{align*}
&\{\beta(\vk),P(q)\}
= \int \big( |n|^2 + n + \ol n \big) q'\, dx
= \int - [(1+\ol n)n' + (1+n)\ol n']q\, dx.
\end{align*}
Next, we use the equation~\eqref{m eq} for $n'$ together with \eqref{beta} to deduce that
\begin{align*}
\{\beta(\vk),P(q)\} &= i \int \vk[ n(\ol n+1)-\ol n(n+1) ]q \,dx \\
&\qquad - i \int (\ol n+1)q \cdot C_+ (n+1)q - (n+1)q \cdot C_- (\ol n+1)q \,dx \\
& =0.\qedhere
\end{align*}
\end{proof}

\subsection{The Bock--Kruskal transformation}\label{ss:Bock}
In~\cite{Bock1979}, Bock and Kruskal introduced an analogue of the Miura transform applicable to the Benjamin--Ono equation and used this to show the existence of infinitely many conserved quantities, at least for smooth solutions decaying sufficiently rapidly at (spatial) infinity.  This transformation $q\mapsto w$ was defined implicitly via the formula
\begin{equation}\label{BKeqn}
2q = \tfrac{1}{w+\kappa} \h (w') + \h \big( \tfrac{w'}{w+\kappa} \big) + \tfrac{2\kappa w}{w+ \kappa} .
\end{equation}
The function $w$ is real-valued.  As in the original paper \cite{Bock1979}, we will confine our discussion to the $\R$ geometry.

In the introduction, we described the important inspirational role that the Bock--Kruskal transformation played in developing the methods ultimately employed in this paper.  Given this pivotal role, we feel compelled to share with the reader how it connects to the principal themes of this paper.  Concretely, we will demonstrate the unique solvability of \eqref{BKeqn} and identify this solution in terms of the central object $m(x;\kappa,q)$ of this section.

Evidently, some restriction on $w$ (beyond mere regularity) must be imposed to handle the denominators $\kappa+w$ appearing in \eqref{BKeqn}.  As any  $w\in H^{s+1}$ is automatically continuous and converges to zero at (spatial) infinity, the natural condition is this:
\begin{equation}\label{1+w>0}
\inf_x  \bigl(\kappa+w(x)\bigr) > 0.
\end{equation}

\begin{theorem}\label{T:BK}
Suppose $A>0$ and $\kappa_0(A)$ satisfies Convention~\ref{C:Akmn}.  Then, for any $q\in \BA$ and any $\kappa\geq \kappa_0$,
\begin{equation}\label{m w}
w = \kappa \tfrac{\del\beta}{\del q} = \kappa \big( |m|^2+ m + \ol m \big)
\end{equation}
is the unique $H^{s+1}(\R)$ solution to \eqref{BKeqn} satisfying \eqref{1+w>0}.
\end{theorem}

\begin{proof}
By virtue of \eqref{m est}, we must have $\|m\|_{L^\infty}<1$. Consequently, the function $\kappa\tfrac{\del\beta}{\del q}= \kappa|m+1|^2-\kappa$ satisfies \eqref{1+w>0}. Setting $\kappa=\vk$ in \eqref{m ODE 1} and dividing by $|m+1|^2$, we find that
\begin{align*}
2q&= \tfrac{\h(|m|^2+m+\ol m)'}{|m+1|^2} + i\Bigl[ \tfrac{\ol{m}'}{\ol m +1} - \tfrac{m'}{m+1} \Bigr] + 2\kappa \tfrac{|m|^2+m+ \ol m}{|m+1|^2}\\
&=\tfrac{\h(|m|^2+m+\ol m)'}{|m+1|^2} + i\h\Bigl[ \tfrac{\ol{m}'}{\ol m +1} + \tfrac{m'}{m+1} \Bigr]+ 2\kappa \tfrac{|m|^2+m+ \ol m}{|m+1|^2},
\end{align*}
which demonstrates that the function $\kappa\tfrac{\del\beta}{\del q}$ satisfies \eqref{BKeqn}. 

It remains to verify the uniqueness of $H^{s+1}$ solutions to \eqref{m w} satisfying \eqref{1+w>0}. We will focus on the unknown $u = \kappa^{-1} w$.  Suppose first that $w$ is a solution of the type described.  The restriction \eqref{1+w>0} guarantees that $\log(1+u)\in H^{s+1}$; see, for example, \eqref{3:18}.  Thus, we may factor
\begin{align}\label{WienerHopf}
1 + u(x) = [1+\mu(x)] [1+\ol\mu(x)] \qtq{with} \mu \in H^{s+1}_+(\R).
\end{align}

The next step is to insert $w=\kappa[1+\mu][1+\ol\mu]- \kappa$ in \eqref{BKeqn}. In doing so, we take advantage of the following:
\begin{align*}
(1+u) \h \big[ \tfrac{u'}{1+u} \big] = 
	|\mu+1|^2 \h \big( \tfrac{\ol\mu'}{1+\ol\mu} + \tfrac{\mu'}{1+\mu} \big) =
	i(1+\mu)\ol\mu' - i\mu' (\ol\mu+1) .
\end{align*}
This allows us to completely eliminate the denominators in \eqref{BKeqn}; indeed, combining this with $2C_\pm=[I\pm i\h]$, we find the equivalent formulation
\begin{equation}\label{BKeqn'}
2q [1+\mu] [1+\ol\mu]  = 2C_-\bigl[ i(1+\mu)\ol\mu'\bigr] - 2C_+\bigl[ i\mu' (\ol\mu+1) \bigr] + 2\kappa [\mu+\ol\mu+|\mu|^2] .
\end{equation}

Isolating the positive-frequency component of \eqref{BKeqn'}, we get
\begin{equation}\label{BKeqn''}
C_+\Bigl[ (1+\ol\mu)\bigl( - i\mu' - C_+(q\mu) + \kappa\mu - q_+ \bigr) \Bigr] =0 .  
\end{equation}
In fact, this is equivalent to \eqref{BKeqn'} because the negative-frequency component is simply the complex conjugate of this.

Let us write $f$ for the quantity inside the square brackets of \eqref{BKeqn''}.   By Lemma~\ref{L:Hs mult}, we know $f\in H^s(\R)$.  Thus we may interpret \eqref{BKeqn''} as saying that $f$ belongs to the Hardy--Sobolev space $H^s_-$, which in turn shows
\begin{equation}\label{BKeqn'''}
 - i\mu' - C_+(q\mu) + \kappa\mu - q_+ = \tfrac{f}{1+\ol\mu} \in H^s_-(\R) .  
\end{equation}
However every term in LHS\eqref{BKeqn'''} belongs to the \emph{other} Hardy--Sobolev space $H^s_+(\R)$.  Only the zero function belongs to both spaces and so we deduce that $\mu$ is a solution of \eqref{m eq}.  However, Proposition~\ref{P:m EU} guarantees that $m$ is the only solution of this equation. Thus $\mu=m$, which then yields $w=\kappa u = \kappa  \tfrac{\del\beta}{\del q}$.
\end{proof}

The Bock--Kruskal approach to conservation laws is that $w$ is a conserved density and consequently, its formal expansion in powers of $\kappa^{-1}$ provides an infinite family of conservation laws of polynomial type.  Combining \eqref{m w} with \eqref{betaX} allows us to connect this approach to the conservation of $\beta(\kappa)$. Concretely, for $\langle x\rangle q \in L^2(\R)$,
\begin{equation}\label{BK as m}
\int w(x;\kappa, q) \,dx = 2\beta(\kappa) - \kappa\frac{\partial\beta}{\partial\kappa}+ 2\int q \,dx .
\end{equation}

\subsection{The perturbation determinant}\label{ss:PD}
Our next result establishes the connection between our gauge $m$ and
the logarithm of the renormalized perturbation determinant
\begin{align}\label{alpha}
\alpha(\kappa;q):= \sum_{\ell\geq 2} \tfrac1\ell\tr\bigl\{
(R_0(\kappa)C_+ q)^\ell \bigr\},
\end{align}
which is the central object in Talbut's proof of Theorem~\ref{T:Talbut} in \cite{Talbut2021a}.  Such a connection in the line setting was presented by Talbut in his thesis \cite[\S3.3]{Talbut2021}.

On the line, convergence of the series \eqref{alpha} may be demonstrated as follows:  For $A>0$ and $\kappa_0=\kappa_0(A)$ chosen according to
Convention~\ref{C:Akmn}, we have
\begin{align*}
\|\sqrt{R_0(\kappa)}C_+ q\sqrt{R_0(\kappa)}\|_{\textrm{HS}}^2
&= \frac{1}{2\pi} \int_0^\infty\!\!\int_0^\infty  \frac{|\widehat
q(\xi-\eta)|^2\,d\eta\,d\xi}{(\eta+\kappa)(\xi+\kappa)}\\
&=\frac{1}{2\pi} \int_\R \frac{\log(1+
\frac{|\xi|}\kappa)}{|\xi|}|\widehat q(\xi)|^2\,d\xi \lesssim
\kappa^{-4\eps} \|q\|_{H^s_\kappa}^2<1,
\end{align*}
whenever $\kappa\geq \kappa_0$ and $q\in \BA$.  In particular, the
H\"older inequality in Schatten classes yields convergence of the
series defining $\alpha$.  Parallel arguments yield convergence in the
circle setting.

\begin{lemma} For $A>0$ and $\kappa_0=\kappa_0(A)$ satisfying
Convention~\ref{C:Akmn}, we have
\begin{align*}
\alpha(\kappa;q)= \tfrac1{2\pi}\int_0^\infty
\tfrac{\beta(\kappa+\xi;q)}{\kappa+\xi}\, d\xi \quad\text{on $\R$}
\qquad\text{and} \qquad \alpha(\kappa;q)= \sum_{\xi\in 2\pi
\Z_+}\tfrac{\beta(\kappa+\xi;q)}{\kappa+\xi} \quad\text{on $\T$},
\end{align*}
whenever $q\in \BA$ and $\kappa\geq \kappa_0$.  Here, $\Z_+=\{0, 1, 2, \ldots\}$.
\end{lemma}

\begin{proof}
We will present the details in the circle setting.  The computations in the line setting are a close parallel.

Using symmetry followed by a change of variables and Plancherel, we may write
\begin{align*}
\frac1\ell &\tr\bigl\{ (R_0(\kappa)C_+ q)^\ell \bigr\}\\
&= \sum_{\xi_1, \ldots, \xi_\ell\in 2\pi\Z_+} \frac1\ell  \, \frac{\widehat q(\xi_1-\xi_2)}{\kappa+\xi_1}
	\frac{\widehat q(\xi_2-\xi_3)}{\kappa+\xi_2} \cdots \frac{\widehat q(\xi_\ell-\xi_1)}{\kappa+\xi_\ell}\\
&=\sum_{\substack{\xi_1\leq\min\{\xi_2, \ldots, \xi_\ell\} \\[0.5ex] \xi_1, \ldots, \xi_\ell\in 2\pi\Z_+}}
	\frac{\widehat q(\xi_1-\xi_2)}{\kappa+\xi_1} \frac{\widehat q(\xi_2-\xi_3)}{\kappa+\xi_2} \cdots \frac{\widehat
q(\xi_\ell-\xi_1)}{\kappa+\xi_\ell}\\
&=\sum_{\xi\in 2\pi\Z_+} \frac1{\kappa+\xi}
	\sum_{\substack{ \eta_2, \ldots, \eta_\ell \in 2\pi\Z \\[0.5ex] \eta_j+ \cdots + \eta_\ell\geq 0,  \forall 2\leq j\leq \ell}}
		\widehat q \bigl(-(\eta_2+\cdots\eta_\ell)\bigr) \prod_{j=2}^\ell\frac{\widehat q(\eta_j)}{\kappa+\xi + \eta_j+ \cdots + \eta_\ell }\\
&=\sum_{\xi\in 2\pi\Z_+} \frac1{\kappa+\xi} \Bigl\langle q,\, \bigl(R_0(\kappa+\xi) C_+ q\bigr)^{\ell-2}R_0(\kappa+\xi) q_+\Bigr\rangle.
\end{align*}
Recalling \eqref{m series}, \eqref{beta}, and summing over $\ell\geq
2$, we obtain
$$
\alpha(\kappa;q) = \sum_{\xi\in 2\pi\Z_+} \frac1{\kappa+\xi}  \langle
q, m(\kappa+\xi, q)\rangle =\sum_{\xi\in 2\pi\Z_+} \frac1{\kappa+\xi}
\beta(\kappa+\xi; q),
$$
which completes the proof in the circle setting.
\end{proof}

\subsection{The action of higher symmetries}\label{ss:higher}
With infinitely many conserved quantities, the Benjamin--Ono equation possesses a wide array of Hamiltonian symmetries.  As these Hamiltonians are all mutually commuting, these symmetries preserve the values of all these conserved quantities.

By \emph{higher} symmetries, we mean those that do not preserve the conserved quantities.  Scaling and Galilei/Lorentz boosts are important examples, common to a rich class of Hamiltonian PDE.  In the Benjamin--Ono setting, these symmetries take the forms given in \eqref{scaling} and \eqref{E:Galilei}, respectively.

The scaling symmetry is Hamiltonian; indeed,  the center of momentum
\begin{align}\label{CofP}
\CofP &:= \int \tfrac12 x q(x)^2\,dx  \qtq{generates} \tfrac{d}{dt} q = (xq)' = xq' + q = \tfrac{d q_\lambda}{d\lambda} \bigr|_{\lambda=1}.
\end{align}
While one should actually divide by the total momentum to find the true centroid, this muddies the formulas without yielding better physical insight. 

The Galilei symmetry is not Hamiltonian; indeed, no Hamiltonian flow can change the value of the Casimir $\int q$.

Our first result describes the action of these higher symmetries on the totality of the conserved quantities, expressed in terms of their generating function $\beta$:

\begin{lemma}
Working on the line, with $q_\lambda$ defined by \eqref{scaling}, we have
\begin{equation}\label{scale beta}
\beta(\lambda \kappa; q_\lambda) = \beta(\kappa; q) \quad\text{for any $\lambda >0$.}
\end{equation}
On the circle, the Galilean symmetry acts as follows: for any $c\in\R$,
\begin{equation}\label{Galilei beta}
\beta(\kappa; q+c) + {\textstyle\int} (q+c)\,dx = \tfrac{\kappa^2}{(\kappa-c)^2} \Bigl[\beta(\kappa -c;q) + {\textstyle\int} q\,dx\Bigr]
	+ \tfrac{c \kappa}{\kappa-c}.
\end{equation}
\end{lemma}

\begin{proof}
We define an operator $\mc U_\lambda$ via
\begin{equation}
[ \mc U_\lambda f](x) = \sqrt\lambda\,f(\lambda x).
\end{equation}
This is unitary on $L^2_+(\R)$.  It differs from the scaling \eqref{scaling} by $q_\lambda = \sqrt\lambda\,\mc U_\lambda q$.

Direct computation shows that
$
(\mc L(q_\lambda) + \lambda\kappa) \;\!\mc{U}_\lambda =  \lambda \;\!\mc U_\lambda (\mc L(q) + \kappa )
$, which implies
\begin{align*}
\mc{U}_\lambda (\mc L(q) + \kappa )^{-1}  =  \lambda (\mc L(q_\lambda) + \lambda\kappa)^{-1} \;\!\mc U_\lambda .
\end{align*}
The identity \eqref{scale beta} now follows easily:
\begin{align*}
\beta(\lambda\kappa; q_\lambda) = \lambda \langle \mc U_\lambda q_+,  (\mc L(q_\lambda) + \lambda\kappa)^{-1} \;\!\mc U_\lambda q_+\rangle
	= \langle \mc U_\lambda q_+,  \;\!\mc U_\lambda (\mc L(q) + \kappa)^{-1}  q_+\rangle = \beta(\kappa;q).
\end{align*} 

The identity \eqref{Galilei beta} follows from \eqref{beta3}, \eqref{beta4}, the observation $\mc L(q+c) = \mc L(q) - c$,
and elementary manipulations.
\end{proof}

By differentiating the identity \eqref{scale beta} with respect to $\lambda$ and setting $\lambda=1$, we obtain the following virial-type identity:
\begin{align}\label{P virial}
\{ \beta(\kappa), \CofP \} &=  -\kappa \tfrac{\partial\beta}{\partial\kappa}.
\end{align}
Understanding the $\CofP$ as the generator of scaling and matching coefficients in the $\kappa\to\infty$ expansion, we see that \eqref{P virial} shows that the Hamiltonians for which $\beta(\kappa)$ is the generating function are individually homogeneous under scaling.

An alternate physical interpretation of \eqref{P virial} is that it reveals the time dependence of $\CofP$ under each of the Hamiltonians; specifically, it shows that the center of momentum travels at a constant speed equal to a numerical multiple of the Hamiltonian.

A third perspective on \eqref{P virial} is this: Given a conserved quantity, taking the Poisson bracket with $\CofP$ will yield a new conserved quantity.  Sadly, it is not really `new'; each term in the expansion of $\beta$ merely picks up a numerical prefactor illustrating its scaling degree.  The Galilei symmetry is more exciting.  The formula \eqref{Galilei beta} shows that by performing a Galilei boost on a single Hamiltonian yields a polynomial in $c$ whose coefficients are all the preceding Hamiltonians.  It allows one to descend through the hierarchy!

As mentioned in the introduction, Fokas and Fuchssteiner \cite{Fokas1981} found a vector field $\tau$ that allowed them to \emph{ascend} in the hierarchy.  As the culmination of this section, we will now explain how the preceding discussion led us to a new and physically appealing interpretation of their discovery.  Then in Section~\ref{S:L&P} we will present a far reaching generalization; see Theorem~\ref{T:Virial}. 

Let us declare that the center of energy is given by
\begin{align}\label{CofE}
\CofE := \int \tfrac12 x q(x) \cdot \h\partial q(x) - \tfrac13 x q(x)^3\,dx .
\end{align}
The term $xq^3$ is not controversial.  However, we have selected a very specific way of inserting the weight $x$ into the kinetic energy term and would have to admit other possibilities, but for the following dramatic observation: The Hamiltonian vector field associated to $\CofE$ is (subject to our sign conventions) precisely the $\tau$ vector field of \cite{Fokas1981}!  To see this, we use that $[\h\partial,x]= \h$ and so 
\begin{align}
\partial_x \bigl(\tfrac{\delta\ }{\delta q} \CofE\bigr) = \bigl[ x\h q' +\tfrac12 \h q\bigr]' - [ x q^2]' = x (\h q'' -2qq') - q^2 +\tfrac32 \h q'.
\end{align} 

In this way, the miraculous property of $\tau$ can be summarized as 
\begin{align}\label{E virial}
\{ \beta(\kappa), \CofE \} &= \kappa^2 \tfrac{\partial\beta}{\partial\kappa} + \kappa \beta(\kappa),
\end{align}
which shows that the Poisson bracket of $\CofE$ and one of Hamiltonians of the hierarchy yields the next \emph{higher} Hamiltonian.  Equivalently,  the center of energy travels at a constant speed, which is given by this higher Hamiltonian.

This presentation leads us naturally to ask: Is there a coherent way of defining the center for every one of the conserved quantities? Perhaps even a unifying $\CofB$?   And can this be done in such a way that these centers move at a constant speed? Naturally, this speed would be another conserved quantity.  We will answer all these questions successfully in Section~\ref{S:L&P}.  This will include a proof of \eqref{E virial}.

For such a direct identity involving $m$ and $q$, it is tempting to imagine that \eqref{E virial} should follow quickly from \eqref{m eq}, \eqref{beta}, and \eqref{beta2} together with some strategic integrations by parts.  We know of no simple argument of this type.  Nevertheless, our discovery of just the right Lax representation of the flows, presented in the next section, will yield the result very quickly indeed.

%%%%%%%%%%%%%%%%%%%%%%%%%%%%%%%%%%%%%%%%%%%%%%%%%%%%%%%%%%%%%%%%%%%%%%%%%%
\section{Well-posedness}\label{S:GWP} %%%%%%%%%%%%%%%%%%%%%%%%%%%%%%%%%%%%%
%%%%%%%%%%%%%%%%%%%%%%%%%%%%%%%%%%%%%%%%%%%%%%%%%%%%%%%%%%%%%%%%%%%%%%%%%%

Our analysis begins with the discussion of the evolution dictated by our regularized Hamilonians $\hk$ introduced in \eqref{Hk intro}.
These Hamiltonians are not globally defined: for a given size of initial data, $\kappa$ needs to be chosen sufficiently large.  With this in mind, Convention~\ref{C:Akmn} will be in force throughout this section.

In this section, we will verify that $\beta$ is conserved under the $\hk$ flow, as well as under \eqref{BO}.  In this way, our convention ensures that $\kappa$ will be large enough, not only for the initial data, but also for all trajectories of interest to us.

Before turning to the well-posedness of the $\hk$ flow, our first result is devoted to describing the associated vector field.

\begin{proposition}\label{P:HK}
The evolution induced by the Hamiltonian $\hk$ is
\begin{align} \label{Hk flow}
\tfrac{d}{dt} q = \begin{cases}
	- \kappa^2 \big( m + \ol m +  |m|^2 \big)'  + \kappa q'  & \text{ on $\R$,}\\
	- \kappa^2 \big( m + \ol m +  |m|^2 \big)'  + \bigl[ \kappa +\tint \bigr] q'  & \text{ on $\T$.}
\end{cases}
\end{align}
Moreover, we have the following Lax pair representation: $q$ solves \eqref{Hk flow} if and only~if 
\begin{equation}\label{Lax form}
\tfrac{d}{dt} \mc{L} = [\mc{P}_\kappa,\mc{L}]
\end{equation}
where $\mc L=\mc L(q(t))$ is the Lax operator described in Proposition~\ref{P:Lax} and
\begin{equation}\label{PkR}
\mc{P}_\kappa:= i\kappa^3(\mc L +\kappa)^{-1} - i\kappa^2 (m+1)C_+(\ol{m}+1) + \kappa\partial,
\end{equation}
on the line; on the circle, $\mc P_\kappa$ is defined by
\begin{equation}\label{PkT}
\mc{P}_\kappa:=i\kappa^2\bigl[\kappa+\beta(\kappa)+\tint\bigr](\mc L +\kappa)^{-1}
	- i\kappa^2 (m+1)C_+(\ol{m}+1) + \bigl[\kappa +\tint \bigr]\partial .
\end{equation}
These operators have the special property
\begin{align} \label{Hk flow Pk}
\tfrac{d}{dt} q_+ = \mc P_\kappa q_+.
\end{align}
\end{proposition}

Before turning to the proof of this result, we pause to note that irrespective of the geometry, the first term in the definition of $\mc P_\kappa$ is inconsequential to the Lax-pair property, because it commutes with $\mc L$.  However, its removal would destroy the special property \eqref{Hk flow Pk}, which greatly expedites the arguments of this section and played a crucial role in our discoveries reported in the next section.

Let us also note that while restricting the torus evolution to $\int q=0$ would unify the dynamical equations \eqref{Hk flow}, it would not do the same for the operators $\mc P_\kappa$; they would still differ by the summand $i\kappa^2 \beta(\kappa) R(\kappa)$.

\begin{proof}[Proof of Proposition~\ref{P:HK}]
To avoid repeating ourselves, we will only present the details in the periodic case, which are slightly more involved.

The equation~\eqref{Hk flow} follows from \eqref{Hk intro}, \eqref{dbeta},  and the Poisson structure \eqref{Poisson}:
\begin{equation}\label{6:12}
\tfrac{d}{dt} q =  - \kappa^2 \big( \tfrac{\del\beta}{\del q} \big)' + \bigl[\kappa+\tint\bigr] \big(\tfrac{\del P}{\del q} \big)' 
	= - \kappa^2 \big( m + \ol m +  |m|^2  \big)' + \bigl[\kappa+\tint\bigr] q'.
\end{equation}
 
Next we address the Lax pair formulation of the $\hk$ flow.  As noted above, it suffices to prove the Lax property with
\begin{equation}\label{Pk tilde}
\widetilde{\mc P}_\kappa := - i\kappa^2 (m+1)C_+(\ol{m}+1) + \bigl[\kappa +\tint \bigr]\partial  .
\end{equation}
If $q$ satisfies \eqref{Hk flow}, then
\begin{equation}\label{5:42}
\tfrac{d}{dt} \mc{L} = \kappa^2 C_+ m'(\ol m +1) + \kappa^2 C_+ (m+1) \ol m' - \bigl[\kappa +\tint \bigr] C_+ q'
\end{equation}
as operators on $L^2_+$.  We will show that RHS\eqref{Lax form}=RHS\eqref{5:42}, which proves that \eqref{Hk flow} implies \eqref{Lax form}. 

Conversely, as $(\tfrac{d}{dt} \mc{L})f =  - C_+( \tfrac{dq}{dt} f)$, the time derivative of $\mc L$ uniquely determines $\tfrac{dq}{dt}$. Thus, the equality RHS\eqref{Lax form}=RHS\eqref{5:42} also shows that \eqref{Lax form} implies \eqref{Hk flow}.

Proceeding directly from the definitions, we find
\begin{align}\label{5:43}
[\widetilde{\mc P}_\kappa,\mc{L}]
&= \kappa^2 \bigl\{ m' C_+(\ol{m}+1) + (m+1)C_+ \ol{m}'  \big\} - \bigl[\kappa +\tint \bigr] C_+q' \\
&\quad - i\kappa^2 C_+ q(m+1)C_+(\ol{m}+1) + i\kappa^2 (m+1)C_+(\ol{m}+1)C_+ q  \notag\\
&=\text{RHS}\eqref{5:42} -\kappa^2 C_+ m' [1 - C_+] (\ol{m}+1)  -\kappa^2 C_+ (m+1) [1 - C_+] \ol m' \notag\\
&\quad - i\kappa^2 C_+ q(m+1)C_+(\ol{m}+1) + i\kappa^2 (m+1)C_+(\ol{m}+1)C_+ q  \notag
\end{align}
as  operators on $L^2_+$.  Now for $f\in H^\infty_+$, \eqref{m eq} yields
\begin{align*}
C_+ m' & [1 - C_+] (\ol{m}+1) f + C_+ (m+1) [1 - C_+] \ol m' f \\
&= -i\kappa C_+  m [1 - C_+] (\ol{m}+1)f + i\kappa C_+  (m+1) [1 - C_+] \ol m f \\ 
&\quad + i C_+ [q(m+1)]_+ [1 - C_+] (\ol{m}+1) f - i C_+ (m+1) [1 - C_+]  [(\ol m+1)q]_- f\\
&= i C_+ q(m+1) [1 - C_+] (\ol{m}+1) f - i C_+ (m+1) [1 - C_+] (\ol m+1) q f \\
&= - i C_+ q(m+1) C_+ (\ol{m}+1) f + i C_+ (m+1) C_+ (\ol m+1) qf \\
&= - i C_+ q(m+1) C_+ (\ol{m}+1) f + i C_+ (m+1) C_+ (\ol m+1) C_+ qf  .
\end{align*}
Substituting this into \eqref{5:43} gives $[\widetilde {\mc P}_\kappa,\mc{L}]  =\text{RHS}\eqref{5:42}$, which shows that RHS\eqref{Lax form} and RHS\eqref{5:42} are equal, thereby completing the proof of the Lax pair formulation.

It remains to justify \eqref{Hk flow Pk}. One important distinction between the two geometries is \eqref{Cid}.  For example, working on $\T$ and using \eqref{beta} and \eqref{m eq}, we have
\begin{align*}%\label{4:44}
C_+(\ol m +1)q_+ &= C_+(\ol m +1)q = [1-C_-](\ol m +1)q + \bigl[ \beta(\kappa) + \tint \bigr] \\
&= (\ol m +1)q - i\ol m' - \kappa \ol m + \bigl[ \beta(\kappa) + \tint \bigr].
\end{align*}
Similar reasoning using \eqref{m eq} shows that
\begin{align*}%\label{4:45}
C_+(m+1)(\ol m +1)q &= C_+(\ol m +1)C_+ (m+1)q + C_+(\ol m +1)[1-C_+] (m+1)q \\
&= C_+(\ol m +1)(-im' +\kappa m),
\end{align*}
irrespective of the geometry.

Combining our last two calculations, we find that on $\T$,
\begin{align*}
i\kappa^2 C_+ & (m+1) C_+(\ol m +1)q_+ \\
&= i\kappa^2C_+  \Bigl[ (m+1)(\ol m +1)q  + (m+1)\bigl[ - i\ol m' - \kappa \ol m + \beta(\kappa) + \tint \bigr] \Bigr]\\
&= i\kappa^2C_+  \Bigl[ (\ol m +1)(-im' +\kappa m)  + (m+1)\bigl[ - i\ol m' - \kappa \ol m + \beta(\kappa) + \tint \bigr] \Bigr] \\
&= \kappa^2 C_+\bigl[ m+\ol m + |m|^2\bigr]' + i \kappa^3C_+(m-\ol m) + i\kappa^2\bigl[ \beta(\kappa) + \tint \bigr] (m+1).
\end{align*}
Again we meet a distinction.  On the line, $C_+\ol m=0$; however, on $\T$, \eqref{beta3} shows
\begin{align*}
i \kappa^3C_+ \ol m = i \kappa^3 {\textstyle\int} \ol m  = i\kappa^2  \bigl[ \beta(\kappa) + \tint \bigr].
\end{align*}
In this way, we deduce that on $\T$,
\begin{align*}
i\kappa^2 C_+ & (m+1) C_+(\ol m +1)q_+ = \kappa^2 C_+\bigl( m+\ol m+|m|^2\bigr)' +  i\kappa^2 \bigl[\kappa + \beta(\kappa) + \tint \bigr] m,
\end{align*}
from which \eqref{Hk flow Pk} follows easily.
\end{proof}

\begin{theorem}[Well-posedness of the $\hk$ flow]\label{t:Hk flow}
Given $A>0$, let $\kappa_0(A)$ be chosen according to Convention~\ref{C:Akmn}.  For $\kappa\geq \kappa_0$, the $\hk$ flow is globally well-posed for initial data in $\BA$.  Moreover, the quantity $\beta(\vk;q(t))$ is conserved by the $\hk$ flow:
\begin{equation}\label{beta dot Hk}
\tfrac{d}{dt} \beta(\vk;q(t))= 0 \qtq{for any} \vk\geq \kappa_0.
\end{equation}
Furthermore, if $q(0)\in \BA \cap H^\infty$ then $q(t)\in H^\infty$ for all $t\in\R$ and the $\hk$ flow commutes with the Benjamin--Ono flow on $H^\infty$.
\end{theorem}

\begin{proof}	
We present the proof in the line setting. On the circle, the linearized flow contains an additional translation at speed $\tint$.  This alters several formulas, but introduces no additional difficulty.

We begin by recasting \eqref{Hk flow} as the integral equation
\begin{equation*}
q(t) = e^{t\kappa\partial} q(0) - \kappa^2 \int_0^t e^{(t-s)\kappa\partial} \big[ |m(\kappa, q(s))|^2 + 2\Re m(\kappa,q(s)) \big]'\, ds .
\end{equation*}
Next, we observe that $q\mapsto[|m|^2 + 2\Re m]'$ is a Lipschitz function.  This follows from \eqref{m est}, \eqref{dm0}, \eqref{diffeo 1}, the fundamental theorem of calculus, and the fact that $H^{s+1}$ is an algebra:
\begin{align*}
\bnorm{ \big[|m(\kappa, q)|^2 &+ 2\Re m(\kappa,q) \big]'-\big[|m(\kappa, \wt{q})|^2 + 2\Re m(\kappa,\wt{q})\big]'}_{H^s}\\
&\lesssim \bnorm{ \big[|m(\kappa, q)|^2 + 2\Re m(\kappa,q) \big]-\big[|m(\kappa, \wt{q})|^2 + 2\Re m(\kappa,\wt{q})\big]}_{H^{s+1}}\\
&\lesssim \norm{ m(\kappa,q) - m(\kappa,\wt{q}) }_{H^{s+1}} \bigl[ \norm{ m(\kappa,q) }_{H^{s+1}}+ \norm{m(\kappa,\wt{q}) }_{H^{s+1}}+1\bigr]\\
&\lesssim \norm{q-\wt{q}}_{H^s} \bnorm{ dm|_q}_{H^s\to H^{s+1}}\bigl[ \norm{q}_{H^s} + \norm{\wt{q}}_{H^s} +1\bigr]\\
&\lesssim \norm{q-\wt{q}}_{H^s}
\end{align*}
uniformly for $q,\tilde q\in (\BA)_{**}$ and $\kappa\geq \kappa_0$.  (For this notation, see Lemma~\ref{L:beta boot}.)  Thus, local well-posedness on this larger set follows by Picard iteration.

Next we address the propagation of additional regularity.  By Proposition~\ref{P:m EU} we know that $q\in H^\infty$ implies $m(\kappa, q)\in H^\infty$.  Indeed, the quantitative bound \eqref{quant H infty} together with a Gronwall argument shows that higher regularity norms can grow at most exponentially in time.  Most important for us is the conclusion that when $q(0)\in H^\infty$, so $q(t)\in H^\infty$ for all times of existence.  

For $H^\infty$ solutions to the $\hk$ flow,  Lemma~\ref{L:Poisson 2} shows that
\begin{equation*}
\tfrac{d}{dt} \beta(\vk) = \{ \beta(\vk),H_\kappa \}= - \kappa^2 \{ \beta(\vk),\beta(\kappa) \} + \kappa \{ \beta(\vk),P(q) \}=0.
\end{equation*}
The conservation of $\beta(\vk)$ for $H^s$-solutions then follows from the $H^s$-continuity of $q\mapsto\beta(\vk;q)$ and the local well-posedness of the flow.

As Lemma~\ref{L:beta boot} demonstrates, the conservation of $\beta$ ensures that the local-in-time argument may be iterated indefinitely, thus yielding global well-posedness in $H^s$.
	
Lastly, we verify that the $\hk$ and the Benjamin--Ono flows commute on $H^\infty$ solutions.  We have
\begin{equation*}
\{ \hk,\hbo \} = - \kappa^2 \{ \beta(\kappa), \hbo \}  + \kappa \{ P,\hbo \}.
\end{equation*}
Each bracket on the right-hand side above vanishes because the $\hbo$ flow conserves both $\beta$ (see~\eqref{beta dot BO}) and the momentum $P$.
\end{proof}

Due to their commutativity, one may define a joint flow under both the Benjamin--Ono and $\hk$ Hamiltonians, at least for $H^\infty$ initial data.  The conservation of $\beta$ under both of these flows provides bounds and equicontinuity of joint orbits:

\begin{corollary}\label{C:Qstar}
Given $A>0$ and a set of real-valued initial data  $Q\subset B_A^s\cap H^\infty$, we define
\begin{equation}\label{Qstar}
Q_* = \big\{ e^{J\nabla (t_1\hbo + t_2\hk)}(q) : q\in Q,\ t_1,t_2\in\R,\ \kappa\geq\kappa_0(A) \big\}.
\end{equation}
Then $Q_*\subset Q_{**}$ and so $Q_{*}$ is bounded; indeed,
\begin{equation}\label{Qstar 2}
\bigl( 1 + \|q_0\|_{H^s} \bigr)^{-2|s|}  \| q_0 \|_{H^s}
	\lesssim \| q \|_{H^s} \lesssim \bigl( 1 + \|q_0\|_{H^s} \bigr)^{\frac{2|s|}{1-2|s|}}  \| q_0 \|_{H^s}
\end{equation}
for every $q\in\{q_0\}_{**}$ and $q_0\in \BA$. If $Q$ is equicontinuous, then so too is $Q_*$.
\end{corollary}

\begin{proof}
As we saw in \eqref{beta dot BO} and \eqref{beta dot Hk}, both flows defining $Q_*$ conserve $\beta$.  Thus $Q_*\subset Q_{**}$ and so Lemma~\ref{L:beta boot} may be applied.  The right-hand inequality in \eqref{Qstar 2} is just a recapitulation of \eqref{beta APB}.  The left-hand inequality follows from this by reversing the roles of $q$ and $q_0$.
\end{proof}

As discussed in the introduction, we wish to show that trajectories under the $H_\kappa$ Hamiltonian closely parallel the original Benjamin--Ono flow.  How is this to be done?  An obvious approach would be to compute the difference of the two vector fields and endeavor to show this is small in some sense.  This strikes the immediate hurdle that \eqref{BO} does not define a vector field on $H^s$ because the operator $q\mapsto q^2$ is not well-defined, even as a distribution.  Before taking the difference, we must make a gauge transformation; specifically, we will use $q\mapsto n= m(\vk, q)$.  Recall that by Proposition~\ref{P:m Diff}, this is a diffeomorphism from bounded subsets of $H^s$ into $H^{s+1}$, provided $\vk$ is sufficiently large.

The special property \eqref{Hk flow Pk} of our Lax pair representation \eqref{Lax form} makes it easy to deduce the dynamics of the new unknown $n=(\mc L + \vk)^{-1}q_+$ under the $H_\kappa$ flow:
\begin{align}\label{n dot Hk v1}
\tfrac{d}{dt} n &= [\mc P_\kappa,R(\vk)]q_+ + R(\vk) \mc P_\kappa q_+ = \mc P_\kappa R(\vk) q_+ = \mc P_\kappa n.
\end{align}
Indeed, this is the argument we used to deduce \eqref{m dot BO}, which says that
\begin{align}\label{n dot BO v1}
\tfrac{d}{dt} n = \mc P n = - in'' - 2C_+([q-q_+] n)' - 2q_+n' 
\end{align}
under the \eqref{BO} flow.

While these formulas are succinct and do make sense for $q\in H^s$, they obscure the numerous subtle cancellations that we must exploit in order to show convergence of the $\hk$ flows to the Benjamin--Ono flow as $\kappa\to\infty$.  Indeed, in the form presented, it is far from clear that the  $\kappa\to\infty$ limit of $\mc P_\kappa n$ even exists!  Our next step is to rewrite the evolution of $n$ under both the Benjamin--Ono and the $H_\kappa$ Hamiltonians in a new way that is amenable to demonstrating this essential convergence.

%In order to unify the formulas across the two geometries and so streamline the subsequent development, we present our new formulas under the hypothesis
%\begin{equation}\label{mean0hyp}
%\int_\T q\,dx = 0  \quad \text{in the periodic case.}
%\end{equation}
%This is purely for expository convenience; all the key cancellations needed to proceed without it will be presented in the proof.  

\begin{lemma}\label{L:best n dot} If $q(t)$ is an $H^\infty(\R)$ solution of \eqref{BO} on the line, then
\begin{align}\label{m dot BO 3}
\tfrac{d}{dt} n= \bigl\{ \mc{L}n - C_+(q_- n) \bigr\}' - iq_+\mc{L}n + q'_+n - iq_+C_+(q n),
\end{align}
while for solutions of the $H_\kappa$ flow on the line we have
\begin{align}\label{m dot Hk 2}
\tfrac{d}{dt}n ={}&\bigl\{ \kappa R(\kappa) \mc{L} n - \kappa^2 C_+\big[\ol{m} R(\kappa)n \big] \bigr\}'
	- i \kappa q_+ R(\kappa) \mc{L}n  + \kappa m'n\notag\\
& - i\kappa m C_+( [q-q_-+\kappa\ol{m}] n ) - i \kappa C_+(q_+\ol{m}) \cdot \mc{L} R(\kappa) n \\
&  + \kappa C_+\big( |m|^2 \big)' \cdot n - i\kappa  [1-C_-](q_+\ol{m})\cdot mn . \notag
\end{align}
On the circle, these formulas are modified as follows: 
\begin{align*}
\tfrac{d}{dt} n&= \textup{RHS\eqref{m dot  BO 3}} + [\tint ] n' \qtq{under \eqref{BO},} \\
\tfrac{d}{dt} n&= \textup{RHS\eqref{m dot  Hk 2}} + [\tint ] n' \qtq{under the $H_\kappa$ flow.}
\end{align*}
\end{lemma}

\begin{proof}
We will provide the details in the periodic setting to explain the appearance of the extra term.  The key distinction comes from \eqref{Cid}.

From \eqref{m dot BO}, which we recalled in \eqref{n dot BO v1}, under the Benjamin--Ono flow we have
\begin{align*}
\tfrac{d}{dt} n
&= - in'' - 2C_+([q -q_+]n)' - 2q_+n' \\
&= \bigl\{ - in' - C_+(qn) - C_+([q-q_+]n) \bigr\}' + (q_+n)' - 2q_+n' \\
&=\bigl\{ \mc{L}n - C_+([q-q_+]n) \bigr\}' + q'_+n - iq_+\mc{L}n - iq_+C_+(qn)\\
&=\textup{RHS\eqref{m dot  BO 3}} + [\tint ] n'.
\end{align*}

We now consider the $H_\kappa$ flow. Our starting point is \eqref{n dot Hk v1} with $\mc P_\kappa$ as defined in \eqref{PkT}.  Let us manipulate some of these terms before putting them together:
\begin{align*}
i\kappa^3 R(\kappa,q)n + \kappa\partial n &=  i\kappa^2n + \kappa\partial n - i\kappa\mc L n   + i \kappa \mc L R(\kappa) \mc L n\\
&= i\kappa^2n  + i\kappa C_+ (qn) + i \kappa\mc L R(\kappa) \mc L n,\\
i\kappa^2 \btint R(\kappa)n &= i\kappa C_+\bigl(\btint n\bigr)
	- i\kappa \btint R(\kappa) \mc L n,\\
- i\kappa^2 (m+1)C_+\bigl([\ol{m}+1]n\bigr) &= -i\kappa^2 n - i\kappa^2  C_+([m+\ol m]n) - i\kappa^2 m C_+ (\ol m n).
\end{align*}
In this way, we find
\begin{equation}\label{n dot Hk v2}
\begin{aligned}
\tfrac{d}{dt} n ={}& i \kappa \mc L R(\kappa) \mc L n - i\kappa C_+\bigl(\bigl[\kappa m +\kappa \ol m - q - \tint\bigr] n\bigr)
	- i\kappa^2 m C_+ (\ol m n) \\
& - i\kappa \btint R(\kappa) \mc L n + i\kappa^2\beta(\kappa) R(\kappa)n + \btint n'.
\end{aligned}
\end{equation}

Our next simplification involves the second term on the RHS\eqref{n dot Hk v2}; by \eqref{m eq},
\begin{align*}
\bigl[\kappa m +\kappa \ol m - q - \tint \bigr] = (\kappa m - q_+) +  (\kappa \ol m - q_-) = - \mc L m - \ol{\mc L m}.
\end{align*}
Regarding the first and fourth terms on the RHS\eqref{n dot Hk v2}, we have
\begin{align*}
i \kappa \mc L R(\kappa) \mc L n - i\kappa \btint R(\kappa) \mc L n
	= \bigl( \kappa R(\kappa) \mc L n \bigr)' - i\kappa C_+\bigl([q_+ + q_-] R(\kappa) \mc L n\bigr).
\end{align*}
Incorporating this information reveals
\begin{equation}\label{n dot Hk v3}
\begin{aligned}
\tfrac{d}{dt} n
={}& \big( \kappa R(\kappa)\mc{L} n \big)' - i\kappa q_+ R(\kappa)\mc{L}n - i\kappa^2 mC_+(\ol{m}n) + i\kappa (\mc{L}m) n  \\
& + i\kappa C_+\big[ (\ol{\mc{L}m}) n \big] - i\kappa C_+\big[ q_-\mc{L}R(\kappa)n \big] + i\kappa^2\beta(\kappa) R(\kappa)n +\btint n'.
\end{aligned}
\end{equation}

Consideration of the fifth and sixth summands on RHS\eqref{n dot Hk v3} leads us to observe
\begin{align*}
i\kappa C_+\Bigl[ ( \ol{\mc{L}m}) n  - q_- \mc{L}R(\kappa)n \Bigr]
&= i\kappa C_+\Bigl[ \ol{\mc{L}m}\cdot (\mc L +\kappa) R(\kappa) n  - \ol{(\mc{L}+\kappa)m} \cdot \mc{L}R(\kappa)n \Bigr] \\
&= i\kappa^2 C_+\Bigl[ \ol{\mc{L}m} \cdot R(\kappa) n  - \ol{m} \cdot \mc{L}R(\kappa)n \Bigr],
\end{align*}
to which we apply Lemma~\ref{L:id0}.  This yields
\begin{align*}
i\kappa C_+\Bigl[ ( \ol{\mc{L}m}) n - q_- \mc{L}R(\kappa)n \Bigr]
&= - \kappa^2 C_+\big[ \ol{m}\,R(\kappa)n \big]' + i\kappa^2 [1-C_-](q_+\ol{m})\cdot R(\kappa)n.
\end{align*}
Before using this to rewrite $\tfrac{d}{dt} n$, let us pause to observe that \eqref{beta} shows that the last term here may be profitably combined with the second to last term in \eqref{n dot Hk v3}:
\begin{align*}
 i\kappa^2[1-C_-](q_+\ol{m})\cdot R(\kappa)n + i\kappa^2\beta(\kappa) R(\kappa)n &= i\kappa^2 C_+( q_+\ol{m})\cdot R(\kappa)n \\
&= i\kappa C_+( q_+\ol{m})\cdot [ n - R(\kappa)\mc L n ].
\end{align*}
Incorporating all these deductions into \eqref{n dot Hk v3}, we find
\begin{align}\label{n dot Hk v4}
\tfrac{d}{dt} n
={}& \big( \kappa R(\kappa)\mc{L} n \big)' - i\kappa q_+ R(\kappa)\mc{L}n - i\kappa^2 mC_+(\ol{m}n) + i\kappa (\mc{L}m) n \notag \\
& - \kappa^2 C_+\big[ \ol{m}\,R(\kappa)n \big]'  - i\kappa C_+( q_+\ol{m}) \cdot R(\kappa)\mc L n  + i\kappa C_+( q_+\ol{m})\cdot  n  + \btint n'.
\end{align}

Two terms require further attention: neither $i\kappa (\mc{L}m) n$ nor $i\kappa C_+( q_+\ol{m})\cdot  n$ admit a $\kappa\to\infty$ limit.  However, the combination does!  Using the definition of $\mc{L}$ together with Lemma~\ref{L:id0}, we may write
\begin{align*}
i\kappa \bigl[ (\mc{L}m)& +C_+( q_+\ol{m} )\bigr] n\\
&= \kappa m'n - i\kappa [q - q_-] mn - i\kappa nC_+(q_- m) + i\kappa nC_+( q_+\ol{m} )\\
&= \kappa m'n - i\kappa [q - q_-] mn - i\kappa nC_+ \big\{ \ol{q_+} m - \ol{m}q_+ \big\}  \\
&= \kappa m'n - i\kappa [q - q_-] mn - i\kappa nC_+ \big\{ \ol{(\mc{L}+\kappa)m} \cdot m - \ol{m}\cdot (\mc{L}+\kappa)m \big\}  \\
&= \kappa m'n - i\kappa [q - q_-] mn - i\kappa nC_+ \big\{ \ol{\mc{L}m} \cdot m - \ol{m}\cdot\mc{L}m \big\}  \\
&= \kappa m'n - i\kappa m C_+\bigl( [q - q_-] n \bigr) + \kappa nC_+\big( |m|^2 \big)'  - i\kappa mn  [1-C_-](q_+\ol{m}).
\end{align*}
Inserting this into \eqref{n dot Hk v4} completes our treatment of the $H_\kappa$ flow.
\end{proof}

\begin{theorem}\label{t:conv}
Let $\{q_j^0\}_{j\geq 1}\subset H^\infty$ be a sequence of real-valued initial data that converges in $H^s$.  Then for all $T>0$, the corresponding $H^\infty$ solutions $q_j(t)$ to \eqref{BO} converge in $C([-T,T];H^s)$.
\end{theorem}

\begin{proof}
Let $Q = \{ q_j^0 : j\geq 1 \}$ and let $Q_*$ be defined as in~\eqref{Qstar}.  By Corollary~\ref{C:Qstar}, $Q_*$ is bounded and equicontinuous in $H^s$.	

As the $\hk$ and $\hbo$ flows commute (cf. Theorem~\ref{t:Hk flow}), we may write	
\begin{equation*}
q_j(t)  = e^{tJ\nabla \hbo} (q_j^0) = e^{tJ\nabla(\hbo-\hk)} \circ e^{tJ\nabla\hk} (q_j^0) 
\end{equation*}
and so
\begin{equation}
\begin{aligned}
\sup_{|t|\leq T} \norm{ q_j(t) - q_\ell(t) }_{H^s}
\leq{}& \sup_{|t|\leq T} \bnorm{ e^{tJ\nabla\hk} (q_j^0) - e^{tJ\nabla\hk} (q_\ell^0) }_{H^s} \\
&+ 2 \sup_{q\in Q_*}\,\sup_{|t|\leq T} \bnorm{ e^{tJ\nabla(\hbo-\hk)} (q) - q }_{H^s} .
\end{aligned}
\label{conv 1}
\end{equation}
By the well-posedness of the $\hk$ flows, the first term on RHS\eqref{conv 1} converges to zero as $j,\ell\to\infty$ for each fixed $\kappa\geq \kappa_0$.  Therefore, it suffices to show that
\begin{equation}\label{conv 2}
\lim_{\kappa\to\infty}\, \sup_{q\in Q_*}\,\sup_{|t|\leq T} \bnorm{ e^{tJ\nabla(\hbo-\hk)} (q) - q }_{H^s} = 0 .
\end{equation}

We adopt the following notation: given initial data $q\in Q_*$, we write
\begin{equation*}
q(t) = e^{tJ\nabla(\hbo-\hk)} (q)
\end{equation*}
for the corresponding solution to the difference flow and $n(t) := n(x;\vk,q(t))$.  By the diffeomorphism property demonstrated in Proposition~\ref{P:m Diff}, \eqref{conv 2} will follow from
\begin{equation}\label{conv 3}
\lim_{\kappa\to\infty}\,\sup_{q\in Q_*}\,\sup_{|t|\leq T} \norm{ n(t) - n(0) }_{H^{s+1}} = 0. 
\end{equation}

Note that as $Q_*$ is bounded and equicontinuous in $H^s$, the diffeomorphism property together with the translation identity~\eqref{m id} yield that the set
\begin{equation*}
\big\{ n(x;\vk,q(t)) : q\in Q_*,\ t\in\R \big\}
\end{equation*}
is bounded and equicontinuous in $H^{s+1}$.  As equicontinuity in a high regularity space together with convergence in a low regularity space imply convergence in the high regularity space, we see that to prove \eqref{conv 3} it suffices to show   
\begin{equation}\label{conv 4}
\lim_{\kappa\to\infty}\,\sup_{q\in Q_*}\,\sup_{|t|\leq T} \norm{ n(t) - n(0) }_{H^{-2}} = 0. 
\end{equation}

By the fundamental theorem of calculus, \eqref{conv 4} is a consequence of 
\begin{equation}\label{conv 5}
\lim_{\kappa\to\infty}\,\sup_{q\in Q_*}\,\sup_{|t|\leq T} \norm{ \tfrac{dn}{dt}  }_{H^{-2}} = 0,
\end{equation}
where the time derivative of $n$ is dictated by the difference flow.  The equation for this evolution may be deduced immediately from Lemma~\ref{L:best n dot}. In taking this difference, the distinction between the two geometries disappears.

Combining  \eqref{m dot BO 3}, \eqref{m dot Hk 2}, and the identity $\mc L - \kappa R(\kappa)\mc L = \mc L  R(\kappa) \mc L$, we find  
\begin{align}
\tfrac{d}{dt} n
={}& \big\{ \mc{L} R(\kappa)\mc{L} n  - C_+\big[ \big( q_- - \kappa^2\ol{m} R(\kappa) \big) n \big] \big\}' - i q_+\mc{L} R(\kappa)\mc{L}n \notag\\
&+ \big(q'_+ - \kappa m'\big) n - i(q_+-\kappa m) C_+(qn) - i\kappa m C_+([q_-  - \kappa \ol m]n) \label{n dot diff}\\
&+ i \kappa C_+(q_+\ol{m}) \cdot \mc{L} R(\kappa) n - \kappa C_+\big( |m|^2 \big)' \cdot n + i \kappa  [1-C_-](q_+\ol{m}) \cdot mn.\!\!\!\! \notag
\end{align}
We will verify \eqref{conv 5} by showing that each of these terms converges to zero in $H^{-2}$ as $\kappa\to\infty$, uniformly for $q(t)\in (Q_*)_* = Q_*$.  Before delving in the details of this, let us recall some basic bounds that we will use repeatedly:
\begin{align}\label{basicB}
\|q\|_{H^s} + \| n\|_{H^{s+1}} + \| m\|_{H^{s+1}_\kappa} + \| \kappa m\|_{H^{s}} \lesssim 1
\end{align}
uniformly for $q\in Q_*$ and $\kappa\geq \kappa_0$.  The first two of these were noted above; the latter two follow from \eqref{m est}.

For the first term in \eqref{n dot diff}, we use~\eqref{L est} and \eqref{equicty 6} to see that
\begin{align*}
\bnorm{ \big\{ \mc{L} R(\kappa,q)\mc{L} n \big\}' }_{H^{-2}}
\lesssim \norm{ \mc{L} R(\kappa,q)\mc{L} n }_{H^{s}} \lesssim \norm{R(\kappa,q)\mc{L} n }_{H^{s+1}}\to 0 \qtq{as} \kappa\to \infty,
\end{align*}
uniformly for $q\in Q_*$.

Using \eqref{m eq},  \eqref{L est}, and \eqref{equicty 6}, we obtain
\begin{equation}\label{equicty 7}
\norm{ \kappa m - q_+ }_{H^s}
= \norm{\mc{L}m}_{H^s}
\lesssim \norm{m}_{H^{s+1}}
\to 0 \quad\text{as }\kappa\to\infty
\end{equation}
uniformly for $q\in Q_*$.  Employing $\kappa R(\kappa) = 1 - \mc L  R(\kappa)$ and \eqref{E:Hs mult}, we deduce
\begin{align*}
\bnorm{  C_+\big[ \big( q_- - \kappa^2\ol{m}& R(\kappa,q) \big) n \big]' }_{H^{-2}}\\
&\lesssim \norm{ \big( q_- - \kappa\ol{m} \big) n }_{H^{s}} + \kappa \norm{ \ol{m} \mc{L}R(\kappa,q) n }_{H^{s}} \\
&\lesssim \norm{ q_+ - \kappa m }_{H^{s}} \norm{n}_{H^{s+1}} + \norm{\kappa m}_{H^{s}} \norm{ \mc{L}R(\kappa,q) n }_{H^{s+1}}.
\end{align*}
By \eqref{equicty 6}, \eqref{basicB}, and \eqref{equicty 7}, this converges to zero as $\kappa\to \infty$ uniformly for $q\in Q_*$.

Next, we use the estimates~\eqref{low reg alg} and~\eqref{L est} to bound
\begin{equation*}
\norm{ q_+\mc{L} R(\kappa,q)\mc{L}n }_{H^{-2}}
\lesssim \norm{q_+}_{H^s} \norm{\mc{L} R(\kappa,q)\mc{L}n }_{H^{s}}
\lesssim \norm{q_+}_{H^s} \norm{R(\kappa,q)\mc{L}n }_{H^{s+1}}.
\end{equation*}
By \eqref{equicty 6}, this converges to zero as $\kappa\to \infty$ uniformly for $q\in Q_*$.
	
Using the triangle inequality, \eqref{E:Hs mult}, and~\eqref{low reg alg}, we may bound
\begin{align*}
\bnorm{ \big( \kappa m' - q'_+\big) n }_{H^{-2}}
&\leq \bnorm{ \big\{ ( \kappa m - q_+) n \big\}' }_{H^{-2}} + \bnorm{ ( \kappa m - q_+) n' }_{H^{-2}} \\
&\lesssim \norm{  ( \kappa m - q_+) n}_{H^{s}} + \bnorm{ ( \kappa m - q_+) n' }_{H^{2s-1}} \\
&\lesssim \norm{ \kappa m - q_+ }_{H^{s}} \norm{n}_{H^{s+1}} + \norm{ \kappa m - q_+ }_{H^s} \snorm{ n' }_{H^{s}} \\
&\lesssim \norm{ \kappa m - q_+ }_{H^{s}} \norm{n}_{H^{s+1}}.
\end{align*}
By \eqref{basicB} and \eqref{equicty 7}, this converges to zero as $\kappa\to \infty$ uniformly for $q\in Q_*$.
	
Using \eqref{low reg alg} and~\eqref{E:Hs mult} again, we may bound
\begin{align*}
\norm{ ( q_+ - \kappa  m ) C_+(q n) }_{H^{-2}}
&\lesssim \norm{ q_+ - \kappa  m }_{H^{s}} \norm{ q n }_{H^{s}} \\
&\lesssim \norm{ q_+ - \kappa  m }_{H^{s}} \norm{q}_{H^{s}} \norm{ n }_{H^{s+1}} .
\end{align*}
This converges to zero as $\kappa\to \infty$ uniformly for $q\in Q_*$ in view of \eqref{basicB}, \eqref{equicty 7}.
	
By \eqref{low reg alg}, and~\eqref{E:Hs mult}, we have
\begin{align*}
\bnorm{ \kappa m C_+\big[ (q_- - \kappa\ol{m}) n \big] }_{H^{-2}} 
&\lesssim \norm{ \kappa m }_{H^{s}} \bnorm{ (q_- - \kappa\ol{m}) n }_{H^{s}} \\
&\lesssim \norm{ \kappa m }_{H^{s}} \norm{ q_+ - \kappa m }_{H^{s}} \norm{n}_{H^{s+1}}.
\end{align*}
This converges to zero as $\kappa\to \infty$ uniformly for $q\in Q_*$ because of \eqref{basicB}, \eqref{equicty 7}.
	
Using \eqref{low reg alg} and~$\kappa R(\kappa) = 1 - R(\kappa) \mc L$, followed by \eqref{E:Hs mult} and~\eqref{L est}, we have
\begin{align*}
\norm{ \kappa C_+(q_+\ol{m}) \cdot \mc{L}  R(\kappa,q) n }_{H^{-2}}
&\lesssim \norm{ q_+\ol{m} }_{H^{s}} \big[ \norm{ \mc{L} n }_{H^{s}} + \norm{ \mc{L} R(\kappa,q) \mc{L} n }_{H^{s}} \big] \\
&\lesssim \norm{ q_+ }_{H^{s}} \norm{ m}_{H^{s+1}} \big[\norm{ n }_{H^{s+1}} + \norm{ R(\kappa,q) \mc{L} n }_{H^{s+1}}\big],
\end{align*}
which converges to zero as $\kappa\to \infty$ uniformly for $q\in Q_*$ in view of \eqref{equicty 6}, \eqref{basicB}.
	
By the triangle inequality and the estimates~\eqref{E:Hs mult} and~\eqref{low reg alg}, we may bound
\begin{align*}
\bnorm{ \kappa C_+\big( |m|^2 \big)' \cdot n }_{H^{-2}}
&\leq \kappa \bnorm{ C_+\big( |m|^2 \big) \cdot n }_{H^{-1}} + \kappa \bnorm{ C_+\big( |m|^2 \big) \cdot n' }_{H^{-2}} \\
&\lesssim \kappa \norm{ C_+\big( |m|^2 \big) }_{H^{s}} \bigl[ \bnorm{ n }_{H^{s+1}} + \bnorm{ n' }_{H^{s}} \bigr]\\
&\lesssim \kappa \norm{ m }_{H^s}\norm{ m }_{H^{s+1}}\norm{ n }_{H^{s+1}},
\end{align*}
which converges to zero as $\kappa\to \infty$ uniformly for $q\in Q_*$ as follows from \eqref{basicB} and \eqref{equicty 7}.
	
Finally, using the estimates~\eqref{low reg alg} and~\eqref{E:Hs mult}, we have
\begin{align*}
\norm{ \kappa [1-C_-](q_+\ol{m}) \cdot mn  }_{H^{-2}}
&\lesssim \kappa \norm{ q_+\ol{m} }_{H^{s}} \norm{ mn }_{H^{s}} \\
&\lesssim \kappa \norm{ q_+ }_{H^{s}} \norm{ m}_{H^{s+1}} \norm{ m }_{H^{s}} \norm{ n }_{H^{s+1}},
\end{align*}
which converges to zero as $\kappa\to \infty$ uniformly for $q\in Q_*$ in view of \eqref{basicB}, \eqref{equicty 7}.
	
Collecting all our estimates, we deduce \eqref{conv 5}, which completes the proof of the theorem.
\end{proof}

\begin{proof}[Proof of Theorem \ref{t:main}]
By the prior work discussed in the introduction, it suffices to consider $-\frac{1}{2} < s < 0$.  We want to show that the solution map $\Phi$ for~\eqref{BO} extends uniquely from $H^\infty$ to a jointly continuous map $\Phi : \R\times H^s \to H^s$.
	
Given initial data $q_0\in H^s$, we define $\Phi(t,q_0)$ as follows:  Let $\{ q_j^0 \}_{j\geq 1}$ be a sequence of $H^\infty$ functions that converges to $q_0$ in $H^s$.  Applying Theorem~\ref{t:conv} to the sequence $\{ q_j^0 \}_{j\geq 1}$, we see that the corresponding $H^\infty$ solutions $q_j(t)$ to~\eqref{BO} converge in $H^s$ and the limit is independent of the sequence $\{ q_j^0 \}_{j\geq 1}$.  Consequently,
\begin{equation*}
\Phi(t,q_0) := \lim_{j\to\infty} q_j(t)
\end{equation*}
is well-defined.
	
We must show that $\Phi$ is jointly continuous.  Fix $T>0$ and let $\{ q_j^0 \}_{j\geq 1}$ be a sequence of initial data in $H^s$ that converges to $q_0$ in $H^s$.  By the definition of $\Phi$, we may choose another sequence $\wt{q}_j(t)$ of $H^\infty$ solutions to~\eqref{BO} such that
\begin{equation}\label{2:36}
\sup_{|t|\leq T}  \norm{ \Phi(t,q_j^0) - \wt{q}_j(t) }_{H^s} \to 0 \quad\text{as }j\to\infty .
\end{equation}
In particular, $\wt{q}_j(0) \to q_0$ in $H^s$, and so Theorem~\ref{t:conv} yields
\begin{equation}\label{2:37}
\sup_{|t|\leq T}  \norm{ \wt{q}_j(t) - \Phi(t,q_0) }_{H^s} \to 0 \quad\text{as }j\to\infty .
\end{equation}

Given $\{t_j\}\subset[-T,T]$ that converges to some $t\in [-T,T]$, we may bound
\begin{align*}
&\norm{ \Phi(t_j,q_j^0) - \Phi(t,q_0) }_{H^s}\\
&\qquad\leq \norm{ \Phi(t_j,q_j^0) - \wt{q}_j(t_j) }_{H^s} +  \norm{ \wt{q}_j(t_j)-\wt{q}_j(t)  }_{H^s} + \norm{\wt{q}_j(t)  - \Phi(t,q_0) }_{H^s}\\
&\qquad\leq \sup_{|t|\leq T}  \norm{ \Phi(t,q_j^0) - \wt{q}_j(t) }_{H^s} + \norm{ \wt{q}_j(t_j)-\wt{q}_j(t)  }_{H^s} +\sup_{|t|\leq T}  \norm{ \wt{q}_j(t) - \Phi(t,q_0) }_{H^s}.
\end{align*}
The right-hand side above converges to zero as $j\to\infty$ by \eqref{2:36}, \eqref{2:37}, and Theorem~\ref{t:conv}.  This demonstrates that $\Phi$ is jointly continuous.
\end{proof}

%%%%%%%%%%%%%%%%%%%%%%%%%%%%%%%%%%%%%%%%%%%%%%%%%%%%%%%%%%%%%%%%%%%%%%%%%%
\section{The tau function and virial identities for the full hierarchy}\label{S:L&P} %%%%%%%%%%%%%%%%%%%%%%%%%%%%%%%%%%%%%
%%%%%%%%%%%%%%%%%%%%%%%%%%%%%%%%%%%%%%%%%%%%%%%%%%%%%%%%%%%%%%%%%%%%%%%%%%

This section presents two new families of identities.  The first is Theorem~\ref{T:Gerard}, which generalizes G\'erard's explicit formula \cite{Gerard2022}; the second is Theorem~\ref{T:Virial}, which presents virial-type identities fulfilling the promises made at the end of Section~\ref{S:gauge}.

Throughout this section we will work on the line and consider the flow generated by employing $\beta(\kappa;q)$ as Hamiltonian.  This leads to the dynamics
\begin{equation}\label{beta flow}
\tfrac{d}{dt} q = \big(m + \ol{m} + |m|^2\big)',
\end{equation}
whose well-posedness in $H^s$ follows from the arguments presented in Theorem~\ref{t:Hk flow}.  Indeed, the $H_\kappa$ flow differs from the $\beta(\kappa)$ flow only by a time rescaling and a spatial translation.  Because of this relationship, Proposition~\ref{P:HK} also provides us with a Lax pair representation of this flow, namely,
\begin{equation}\label{PbkR}
\mcPbk:= -i\kappa(\mc L +\kappa)^{-1} + i (m+1)C_+(\ol{m}+1) .
\end{equation}

We will be studying the evolution \eqref{beta flow} with initial data $q^0\in L^2$.  The equation \eqref{beta flow} is also well-posed in this finer topology, as can be shown by mimicking the proof of Theorem~\ref{t:Hk flow}.  To avoid such repetition, we offer the following alternate argument.  By \eqref{poisson 2}, we know that the flow \eqref{beta flow} preserves the $L^2$ norm.  In this way, continuity of the data-to-solution map follows from mere weak continuity, which may be derived from $H^s$ well-posedness.

The central theme of this section is how the special properties of the Lax representation \eqref{PbkR} lead quickly to the sought-after formulas.   In addition to the special properties 
\begin{equation}\label{PbkR spec}
\tfrac{d}{dt}  q_+(t) = \mcPbk q_+(t) \qtq{and}  \tfrac{d}{dt}  n(x;q(t)) = \mcPbk n(x;q(t))
\end{equation}
that played an important role in the previous section, we also need two more.

One of these additional properties is that $\mcPbk 1=0$. Strictly speaking, this is only true in the circle setting, where it follows from the arguments used to prove \eqref{beta4}.  On the line, the constant function $1$ does not belong to the natural domain of $\mcPbk$.  We will prove a proper analogue in Lemma~\ref{L:P10}.

The second additional property is the value of the commutator between $\mcPbk$ and the operator $X$ corresponding to multiplication by $x$ presented in Lemma~\ref{L:X}; this is the subject of Lemma~\ref{L:comm 2}.  

As motivation for such preliminaries, let us now present our generalization of G\'erard's explicit formula from \cite{Gerard2022}: 

\begin{theorem}\label{T:Gerard}
Let $A>0$ and $\kappa_0(A)$ satisfy Convention~\ref{C:Akmn}.  Then for any $q^0\in\BA\cap L^2(\R)$ and any $\kappa\geq\kappa_0$, the solution $q(t)$ to \eqref{beta flow} with initial data $q^0$ satisfies
\begin{equation}\label{gerard beta R}
q_+(t,z) = \tfrac{1}{2\pi i} I_+ \Big( \big( X - t\kappa R(\kappa;q^0)^2 - z \big)^{-1} q^0_+\Big)
\end{equation}
for all $\Im z > 0$.
\end{theorem}

Although \eqref{gerard beta R} only contains the positive frequency part of $q(t)$ and only off the real axis, this is sufficient to recover the entire waveform; indeed,
$$
q(t,x) = \lim_{y\downarrow 0} \Bigl[ q_+(t,x+iy) + \ol{ q_+(t,x+iy) } \;\! \Bigr]
$$
in $L^2(\R)$ sense.

Our next lemma gives the promised line analogue of the relation $\mc P^\beta_\kappa 1 =0$ valid on the circle.

\begin{lemma}\label{L:P10}
Let $A>0$ and $\kappa_0(A)$ satisfy Convention~\ref{C:Akmn}.  Then 
\begin{equation}\label{E:P10}
\chi_y(x)=\tfrac{iy}{x+iy} \qtq{satisfies} \lim_{y\to\infty} \mcPbk \chi_y = 0
\end{equation}
in $L^2_+$-sense uniformly for $q$ in $L^2$-compact subsets of $\BA \cap L^2(\R)$.
\end{lemma}

\begin{proof}
Using the resolvent identity and elementary manipulations, we find that
\begin{align*}
[\kappa R(\kappa) - 1]\chi_y &= R(\kappa) C_+ q \kappa R_0(\kappa) \chi_y  + [\kappa R_0(\kappa) - 1]\chi_y  \\
	&= R (\kappa)q_+ - R(\kappa) C_+(q-q\chi_y) +  [R(\kappa) C_+ q + 1][\kappa R_0(\kappa) - 1]\chi_y.
\end{align*}
As $R(\kappa)q_+=m$ and $\kappa R_0(\kappa)=1-R_0(\kappa)\mc L_0$, we deduce that 
\begin{align*}
\bigl\| [\kappa R(\kappa) - (m +1)]\chi_y \bigr\|_{L^2} \lesssim 
	\bigl\| m(1-\chi_y) \bigr\|_{L^2} + \bigl\| q(1-\chi_y) \bigr\|_{L^2} + \bigl\| \mc L_0\chi_y \bigr\|_{L^2} ,
\end{align*}
which converges to zero as $y\to\infty$, uniformly on compact subsets of $\BA\cap L^2(\R)$.

To complete the proof of \eqref{E:P10}, it remains to show that $(m +1) C_+ (\ol m\chi_y)\to 0$ in $L^2$ as $y\to\infty$.
Noting that $C_+(\ol m)=0$, we find
\begin{align*}
\| (m +1) C_+ (\ol m\chi_y) \|_{L^2}  \lesssim \bigl[1+\|m\|_{H^{s+1}} \bigr]\| \ol m(1-\chi_y) \|_{L^2} \to 0 \qtq{as} y\to\infty,
\end{align*}
uniformly on compact subsets of $\BA\cap L^2(\R)$.
\end{proof}

Next we record another algebraic virtue of the operators $\mcPbk$, regarding their commutator properties with the operator $X$.

\begin{lemma}\label{L:comm 2}
Let $A>0$ and $\kappa_0(A)$ satisfy Convention~\ref{C:Akmn} and suppose $q\in \BA$ satisfies $q \in H^\infty(\R)$ and $\langle x\rangle q\in L^2(\R)$.  Then
\begin{equation}\label{comm 2}
[X,\mcPbk] = - \kappa R(\kappa,q)^2 
\end{equation}
as operators on $D(X)$.
\end{lemma}

\begin{proof}
We adopt the shorthand $R = R(\kappa;q)$.  We have
\begin{align}\label{6:21}
-i\kappa [X, R] = i\kappa R [X,\mc L] R = - \kappa R^2 - i\kappa R [ X, C_+ q] R .
\end{align} 

Using \eqref{comm 1} for $m\in H^\infty(\R)$ and $f\in D(X)$, we obtain
\begin{align}
[X,i(m+1)C_+(\ol{m}+1)]f 
&= i[X,(m+1)] C_+ (\ol{m}+1) f + i(m+1) [X, C_+(\ol{m}+1)]f \notag\\
&= i [X, m]C_+(1+\ol{m}) f \notag \\
&= - \tfrac{1}{2\pi} m \cdot I_+\bigl( f + C_+(\ol{m} f) \bigr) \notag\\
& = - \tfrac{1}{2\pi} R q_+ \cdot I_+\bigl( f + C_+(\ol{m} f) \bigr). \label{6:20}
\end{align}

Using \eqref{I+ defn} and noting that $\ol m f\in L^1$, we find
\begin{equation*}
I_+\bigl( C_+(\ol{m}f) \bigr) = \int\ol m f\, dx= \langle R q_+ , f\rangle = \langle q_+ , R f\rangle = I_+\big( C_+(qRf) \big) .
\end{equation*}
Note that the hypothesis $\langle x\rangle q\in L^2$ ensures that $ C_+(qRf)\in D(X)$ whenever $f\in D(X)$.   As derivatives vanish at zero frequency, we also have
\begin{equation*}
I_+\bigl( f \bigr) = I_+\bigl( (\mc L + \kappa) R f \bigr) = \kappa I_+\bigl( R f \bigr) - I_+\bigl( C_+(q R f) \bigr) .
\end{equation*}

Employing the last two identities in \eqref{6:20} and invoking \eqref{comm 1}, we obtain
\begin{align*}
[X,i(m+1)C_+(\ol{m}+1)]f &= i\kappa R [ X, C_+ q] R f. 
\end{align*}
The identity~\eqref{comm 2} now follows by combining this with \eqref{6:21}.
\end{proof}

Our last result before the proof of Theorem~\ref{T:Gerard} ensures the propagation of the weighted decay condition $\langle x\rangle q\in L^2(\R)$ under the flow \eqref{beta flow}.

\begin{lemma}\label{L:prop x}
Let $A>0$ and $\kappa_0(A)$ satisfy Convention~\ref{C:Akmn} and suppose $q^0\in\BA$ satisfies $\langle x\rangle q^0(x)\in L^2(\R)$ and $q^0\in H^\infty(\R)$.  Let $q(t)$ denote the evolution of $q^0$ under \eqref{beta flow} with $\kappa\geq \kappa_0$.  Then
\begin{align}\label{weighted}
\bigl\|\langle x\rangle q(t,x) \bigr\|_{L^2} + \bigl\| q(t,x) \bigr\|_{H^\sigma} +  \bigl\|\langle x\rangle [n(x;\vk, q(t))+\ol{n}(x;\vk, q(t))] \bigr\|_{L^2} < \infty
\end{align}
for all $t\in\R$, all $\vk\geq \kappa_0$, and all $\sigma\in \N$.
\end{lemma}

\begin{proof}
The smoothness of solutions to \eqref{beta flow} follows from \eqref{quant H infty} and a simple Gronwall argument.  Our main focus here is on spatial decay.

Combining \eqref{m eq} and its complex conjugate shows
\begin{align}\label{(:22}
\bigl( |\partial| + \kappa \bigr) (m+\ol m) = q + C_+(qm)+C_-(q\ol m).
\end{align}

We first study the last two terms in \eqref{(:22}.  As $H^{s+1}\hookrightarrow L^\infty$, so
\begin{align*}
\| \langle x\rangle qm \|_{L^2} + \| \langle x\rangle q\ol m \|_{L^2} \lesssim \| \langle x\rangle q \|_{L^2} \| m \|_{H^{s+1}}.
\end{align*}
This shows that $\widehat {qm} \in H^1(\R)$ and likewise for the Fourier transform of $q\ol m$.  To deduce that $\langle x\rangle[C_+(qm)+C_-(q\ol m)]$ is square integrable, we need to confirm only that the Fourier transform has no discontinuity at the origin.  This is guaranteed by the middle equality in \eqref{beta}.

The arguments presented in the previous paragraph yield the quantitative bound
\begin{align*}
\bigl\| \langle x\rangle [ q + C_+(qm)+C_-(q\ol m) ] \bigr\|_{L^2} \lesssim \bigl[1 + \| m \|_{H^{s+1}} \bigr] \| \langle x\rangle q \|_{L^2} .
\end{align*}
Noting that the commutator $[x,|\partial|]$ is $L^2$ bounded, this can be combined with \eqref{(:22} to yield
\begin{align}\label{(:38}
\bigl\| \langle x\rangle [ m+ \ol m \;\!] \bigr\|_{H^1} \lesssim \bigl[1 + \| m \|_{H^{s+1}} \bigr] \| \langle x\rangle q \|_{L^2} .
\end{align}
This does \emph{not} say that $\langle x\rangle m \in L^2$ because Fourier truncation will typically introduce a discontinuity at the frequency origin.  Taking a derivative remedies this and we may conclude that
\begin{equation} \label{(:53}\begin{aligned}
\bigl\| \langle x\rangle \bigl[ m + \ol m + |m|^2 \bigr]' \bigr\|_{L^2} &\lesssim \bigl[1 + \| m \|_{L^\infty} \bigr] \| \langle x\rangle m' \|_{L^2} \\\
	&\lesssim \bigl[1 + \| m \|_{H^{s+1}} \bigr]^2 \| \langle x\rangle q \|_{L^2} .
\end{aligned}
\end{equation}

Combining \eqref{(:53} with a simple Gronwall argument shows that $\langle x\rangle q(t,x) \in L^2$ for all time.  Combining this with \eqref{(:38} provides the claimed bounds for the function $n=m(x;\vk,q(t))$.
\end{proof}

\begin{proof}[Proof of Theorem~\ref{T:Gerard}]
We start by observing that both sides of \eqref{gerard beta R} depend continuously on $q^0$ in $L^2(\R)$.  In the case of the left-hand side, this follows from the well-posedness of the flow on $L^2(\R)$.  Regarding the right-hand side, we note that $X-t\kappa R^2$ is also maximally accretive (with the same domain as $X$) and so $X-t\kappa R^2-z$ is boundedly invertible on $L^2$ for $z\in \C$ with $\Im z>0$.  In this way, continuity follows from the resolvent identity.

By virtue of this continuity, it suffices to verify \eqref{gerard beta R} for the special case of initial data $q^0\in H^\infty$ satisfying $\langle x\rangle q^0\in L^2$.  Lemma~\ref{L:prop x} guarantees that these properties remain true for $q(t)$ and so allow us to apply Lemma~\ref{L:comm 2} at all times.

As $t\mapsto \mcPbk(t)$ is a continuous curve of bounded anti-selfadjoint operators, so
\begin{equation*}
\tfrac{d}{dt} U(t) = \mcPbk(t)  U(t) \qtq{with} U(0) = \Id 
\end{equation*}
has a unique solution, which is unitary at every time.  Moreover, by virtue of the Lax pair representation and \eqref{PbkR spec},  we know that
\begin{equation}\label{U magic}
U(t)^* \:\!\! R(\kappa;q(t)) U(t) = R(\kappa;q^0)  \qtq{and} q_+(t) = U(t) q^0_+  \qtq{for all} t\in \R.
\end{equation}

Fixing $z$ with $\Im z >0$, we consider two one-parameter families of bounded operators:
\begin{equation}\label{Y12}
Y_1(t) := \big( X - t\kappa R(\kappa;q^0)^2 - z \big)^{-1} \qtq{and} Y_2(t):=U(t)^* (X-z)^{-1} U(t) .
\end{equation}
Both are solutions to
\begin{equation}\label{Y eqn}
\tfrac{d}{dt} Y(t)= \kappa  Y(t) R(\kappa;q^0)^2 Y(t)  \qtq{with} Y(0)=(X-z)^{-1} .
\end{equation}
In the case of $Y_1$, this follows immediately from the resolvent identity.  For $Y_2(t)$, it follows from Lemma~\ref{L:comm 2} and \eqref{U magic}:
\begin{align*}
\tfrac{d}{dt} Y_2(t)=U(t)^* [(X-z)^{-1},\mcPbk] U(t) &=  - U(t)^* (X-z)^{-1} [X,\mcPbk] (X-z)^{-1} U(t) \\
	&=\kappa  Y_2(t) U(t)^* R(\kappa;q(t))^2 U(t) Y_2(t) \\
	&=  \kappa  Y_2(t) R(\kappa;q^0)^2 Y_2(t) .
\end{align*}

A simple Gronwall argument (in operator norm) shows that \eqref{Y eqn} has at most one solution and consequently,
\begin{equation}
\big( X - t\kappa R(\kappa;q^0)^2 - z \big)^{-1} q^0_+ = U(t)^* (X-z)^{-1} U(t) q^0_+ 
\end{equation}
for all times.  Recalling \eqref{U magic} and the Cauchy integral formula \eqref{CIT}, this yields
\begin{align*}%\label{CIT'}
q_+(t,z)  =  \lim_{y\to\infty} \tfrac{1}{2\pi i} \bigl\langle U(t)^* \chi_y, \big( X - t\kappa R(\kappa;q^0)^2 - z \big)^{-1} q^0_+ \bigr\rangle.
\end{align*}
To complete the proof of \eqref{gerard beta R}, it remains only to observe that $U(t)^*\chi_y-\chi_y\to 0$ in $L^2$ as $y\to \infty$, uniformly for $t$ in compact sets, which follows easily from Lemma~\ref{L:P10}.
\end{proof}

The proof of Theorem~\ref{T:Gerard} shows that the mapping between the Hamiltonian and the time-dependent term in the explicit formula is actually linear.  Suppose, for example, we adopt
\begin{equation}\label{phi as H}
 \sum c_j \beta(\kappa_j) = \langle q_+, \phi(\mc L)q_+\rangle \qtq{where} \phi(E) = \sum c_j (E+\kappa_j)^{-1}
\end{equation}
as the Hamiltonian.  This admits a Lax pair representation with $\mc P = \sum c_j \mc P^\beta_{\kappa_j}$.  Furthermore, taking the commutator with $X$ is also a linear operation.  In this way, we find the associated explicit formula
\begin{equation}\label{psi as GE}
q_+(t,z) = \tfrac{1}{2\pi i} I_+ \Big( \big( X - t \psi(\mc L_{q_0}) - z \big)^{-1} q^0_+\Big) \qtq{with} \psi(E) = \phi(E) +  E\phi'(E).\end{equation}

One may also allow $\phi\equiv 1$, which leads to the Hamiltonian $P$ generating translations and to $\psi(\mc L_{q_0})=\Id$.  In this setting, the formula \eqref{psi as GE} is a direct consequence of \eqref{CIT}.  Indeed, $t$ is merely modifying the real part of $z$.  This parallels our discussion in the introduction of the role of $t_0$ in the definition of the $\tau$-function.

Underlining such a $\tau$-function interpretation is the fact that linear combinations of the functions $1$ and $E\mapsto\tfrac{1}{\kappa+E}$ are dense in the class of continuous functions on intervals of the form $[-E_0,\infty]$. 

The linearity property described above also allows us to consider performing a $\kappa\to\infty$ expansion of \eqref{gerard beta R}.  Recall from \eqref{beta as genfun} that this is precisely how $\beta(\kappa)$ encodes the traditional Hamiltonians.  Indeed, the \eqref{BO} flow corresponds to choosing $\phi(E)=E$ and so to $\psi(E)=2E$.  In this way, we recover the explicit formula
\begin{equation}
q_+(t,z) = \tfrac{1}{2\pi i} I_+ \big\{ \big( X - 2t\mc{L}_{q_0} - z \big)^{-1} q_+^0 \big\} 
\label{gerard BO R}
\end{equation}
presented in \cite{Gerard2022}; see also \cite{Sun2021} for the special case where $q$ is an exact multisoliton.

\medskip

We turn now to our last topic.  In \eqref{E:CofB defn} we introduce our extension of the notion of the center of momentum to all conserved quantities of the \eqref{BO} hierarchy, expressed through the generating function $\beta$.  The property that makes these special is that they move at a constant speed dictated by other Hamiltonians in the hierarchy.  As discussed in subsection~\ref{ss:higher}, this also generalizes the Fokas--Fuchssteiner recursion for the construction of conserved quantities.

\begin{theorem}[Virial identities]\label{T:Virial}
Suppose $\langle x\rangle q(x) \in L^2(\R)$.  Then
\begin{equation}\label{E:CofB defn}
\CofB(\vk) := \tfrac12 \int x q [n + \ol n] \,dx 
\end{equation}
satisfies
\begin{equation}\label{B CofB}
\bigl\{ \CofB(\vk)  , \beta(\kappa) \bigr\} = -\kappa \langle q_+, R(\kappa)R(\vk)R(\kappa) q_+\rangle
	= -\kappa \tfrac{\partial}{\partial\kappa} \ \tfrac{\beta(\kappa)-\beta(\vk)}{\kappa-\vk}.
\end{equation}
\end{theorem}

\begin{proof}
Given a pair of real-valued functions $f,g\in L^2(\R)$ with $\langle x\rangle f(x) \in L^2(\R)$,
\begin{equation}
\langle g_+, X f_+\rangle + \langle X f_+, g_+ \rangle = \int_{-\infty}^\infty \ol{\widehat g(\xi)} \cdot i \widehat f'(\xi) \, d\xi =  \int x f(x) g(x) \,dx.
\end{equation}
In this way, we see that the definition of $\CofB$ may be rewritten as
\begin{equation}
\CofB(\vk) =  \tfrac12 \langle  n,  X q_+ \rangle + \tfrac12 \langle X q_+, n \rangle .
\end{equation}

Exploiting \eqref{PbkR spec} and the antisymmetry of $\mcPbk$, we deduce that
\begin{align*}
\bigl\{ \CofB(\vk)  , \beta(\kappa) \bigr\} &= \tfrac12 \langle n, [X,\mcPbk] q_+ \rangle + \tfrac12 \langle [X,\mcPbk] q_+, n \rangle.
\end{align*}
The first identity in \eqref{B CofB} now follows from \eqref{comm 2} and the selfadjointness of $R(\vk)$.  The second identity is a consequence of \eqref{beta2}.
\end{proof}

By expanding the resolvent,  we find that
\begin{equation*}%\label{6Rx}
n = \vk^{-1} q_+ - \vk^{-2} \mc L q_+  + \vk^{-3} \mc L^2 q_+ \pm \cdots
\end{equation*}
and so also that
\begin{align}
\CofB(\vk) = \vk^{-1} \CofP  - \vk^{-2} \CofE + O(\vk^{-3}\bigr).
\end{align}
In this way, both \eqref{P virial} and \eqref{E virial} can be recovered as elementary corollaries of \eqref{B CofB} and the definition \eqref{beta} of $\beta$.

One cannot give an exhaustive account of all possible virial-type identities associated with \eqref{BO} or its hierarchy.  Our goal in this section has been to exhibit how our modified Lax representation begets dramatic algebraic simplifications.  Let us offer just one more example.  Consider
$$
\VofP(q) := \int \tfrac12 x^2 q^2 \,dx = \langle X q_+ , X q_+ \rangle,
$$
which may be viewed as expressing the variance of the momentum distribution.  By the results of this section, we find
\begin{align*}
\Bigl\{ {\int} \tfrac12 x^2 q^2 \,dx ,\ \beta(\kappa) \Bigr\} = 2\kappa \frac{d}{d\kappa} \CofB(\kappa)
\end{align*}
and consequently, this variance has a very simple time dependence under \eqref{beta flow}:
\begin{align*}
\VofP\bigl(q(t)\bigr)= - t^2\bigl( \kappa \tfrac{d^2\beta}{d\kappa^2} + \kappa^2\tfrac{d^3\beta}{d\kappa^3} \bigr)(\kappa;q(0))
	+ 2t\kappa \tfrac{d}{d\kappa} \CofB(\kappa;q(0))  + \VofP\bigl(q(0)\bigr).
\end{align*}
This represents the generalization to the full \eqref{BO} hierarchy of an important identity from \cite{Ifrim2019}.

\bibliographystyle{habbrv}
\bibliography{refs-2}

\begin{thebibliography}{10}

\bibitem{Abdelouhab1989}
L.~Abdelouhab, J.~L. Bona, M.~Felland, and J.-C. Saut.
\newblock Nonlocal models for nonlinear, dispersive waves.
\newblock {\em Phys. D}, 40(3):360--392, 1989.

\bibitem{MR690743}
M.~J. Ablowitz, A.~S. Fokas, and R.~L. Anderson.
\newblock The direct linearizing transform and the {B}enjamin-{O}no equation.
\newblock {\em Phys. Lett. A}, 93(8):375--378, 1983.

\bibitem{AnguloPava2010}
J.~Angulo~Pava and S.~Hakkaev.
\newblock Ill-posedness for periodic nonlinear dispersive equations.
\newblock {\em Electron. J. Differential Equations}, pages No. 119, 19, 2010.

\bibitem{Benjamin1967}
T.~B. Benjamin.
\newblock Internal waves of permanent form in fluids of great depth.
\newblock {\em Journal of Fluid Mechanics}, 29(3):559--592, 1967.

\bibitem{MR1837253}
H.~A. Biagioni and F.~Linares.
\newblock Ill-posedness for the derivative {S}chr\"{o}dinger and generalized
  {B}enjamin-{O}no equations.
\newblock {\em Trans. Amer. Math. Soc.}, 353(9):3649--3659, 2001.

\bibitem{Bock1979}
T.~L. Bock and M.~D. Kruskal.
\newblock A two-parameter {M}iura transformation of the {B}enjamin-{O}no
  equation.
\newblock {\em Phys. Lett. A}, 74(3--4):173--176, 1979.

\bibitem{MR4304314}
B.~Bringmann, R.~Killip, and M.~Visan.
\newblock Global well-posedness for the fifth-order {K}d{V} equation in
  {$H^{-1}(\mathbb R)$}.
\newblock {\em Ann. PDE}, 7(2):Paper No. 21, 46, 2021.

\bibitem{Burq2006}
N.~Burq and F.~Planchon.
\newblock The {B}enjamin-{O}no equation in energy space.
\newblock In {\em Phase space analysis of partial differential equations},
  volume~69 of {\em Progr. Nonlinear Differential Equations Appl.}, pages
  55--62. Birkh\"{a}user Boston, Boston, MA, 2006.

\bibitem{MR1073870}
R.~R. Coifman and M.~V. Wickerhauser.
\newblock The scattering transform for the {B}enjamin-{O}no equation.
\newblock {\em Inverse Problems}, 6(5):825--861, 1990.

\bibitem{Davis1967}
R.~E. Davis and A.~Acrivos.
\newblock Solitary internal waves in deep water.
\newblock {\em Journal of Fluid Mechanics}, 29(3):593--607, 1967.

\bibitem{MR3346690}
Y.~Deng.
\newblock Invariance of the {G}ibbs measure for the {B}enjamin-{O}no equation.
\newblock {\em J. Eur. Math. Soc. (JEMS)}, 17(5):1107--1198, 2015.

\bibitem{MR2307748}
S.~A. Denisov and A.~Kiselev.
\newblock Spectral properties of {S}chr\"{o}dinger operators with decaying
  potentials.
\newblock In {\em Spectral theory and mathematical physics: a {F}estschrift in
  honor of {B}arry {S}imon's 60th birthday}, volume~76 of {\em Proc. Sympos.
  Pure Math.}, pages 565--589. Amer. Math. Soc., Providence, RI, 2007.

\bibitem{Fokas1983}
A.~S. Fokas and M.~J. Ablowitz.
\newblock The inverse scattering transform for the {B}enjamin-{O}no
  equation---a pivot to multidimensional problems.
\newblock {\em Stud. Appl. Math.}, 68(1):1--10, 1983.

\bibitem{Fokas1981}
A.~S. Fokas and B.~Fuchssteiner.
\newblock The hierarchy of the {B}enjamin-{O}no equation.
\newblock {\em Phys. Lett. A}, 86(6-7):341--345, 1981.

\bibitem{Gerard2022}
P.~G\'erard.
\newblock An explicit formula for the {B}enjamin--{O}no equation.
\newblock Preprint \texttt{arXiv:2212.03139}, 2022.

\bibitem{MR4275336}
P.~G\'{e}rard and T.~Kappeler.
\newblock On the integrability of the {B}enjamin-{O}no equation on the torus.
\newblock {\em Comm. Pure Appl. Math.}, 74(8):1685--1747, 2021.

\bibitem{MR4155287}
P.~G\'{e}rard, T.~Kappeler, and P.~Topalov.
\newblock On the spectrum of the {L}ax operator of the {B}enjamin-{O}no
  equation on the torus.
\newblock {\em J. Funct. Anal.}, 279(12):108762, 75, 2020.

\bibitem{Gerard2020}
P.~G\'erard, T.~Kappeler, and P.~Topalov.
\newblock Sharp well-posedness results of the {B}enjamin-{O}no equation in
  ${H}^{s}(\mathbb{T},\mathbb{R})$ and qualitative properties of its solution.
\newblock Preprint \texttt{arXiv:2004.04857}, 2020.

\bibitem{Ginibre1989a}
J.~Ginibre and G.~Velo.
\newblock Commutator expansions and smoothing properties of generalized
  {B}enjamin-{O}no equations.
\newblock {\em Ann. Inst. H. Poincar\'{e} Phys. Th\'{e}or.}, 51(2):221--229,
  1989.

\bibitem{Ginibre1989}
J.~Ginibre and G.~Velo.
\newblock Propri\'{e}t\'{e}s de lissage et existence de solutions pour
  l'\'{e}quation de {B}enjamin-{O}no g\'{e}n\'{e}ralis\'{e}e.
\newblock {\em C. R. Acad. Sci. Paris S\'{e}r. I Math.}, 308(11):309--314,
  1989.

\bibitem{Ginibre1991}
J.~Ginibre and G.~Velo.
\newblock Smoothing properties and existence of solutions for the generalized
  {B}enjamin-{O}no equation.
\newblock {\em J. Differential Equations}, 93(1):150--212, 1991.

\bibitem{harropgriffiths2022global}
B.~Harrop-Griffiths, R.~Killip, M.~Ntekoume, and M.~Visan.
\newblock Global well-posedness for the derivative nonlinear {S}chr\"odinger
  equation in ${L}^2(\mathbb{R})$.
\newblock Preprint \texttt{arXiv:2204.12548}, 2022.

\bibitem{harropgriffiths2020sharp}
B.~Harrop-Griffiths, R.~Killip, and M.~Visan.
\newblock {S}harp well-posedness for the cubic {NLS} and {mKdV} in
  ${H^s(\mathbb R)}$.
\newblock Preprint \texttt{arXiv:2212.03139}, 2020.

\bibitem{Ifrim2019}
M.~Ifrim and D.~Tataru.
\newblock Well-posedness and dispersive decay of small data solutions for the
  {B}enjamin-{O}no equation.
\newblock {\em Ann. Sci. \'{E}c. Norm. Sup\'{e}r. (4)}, 52(2):297--335, 2019.

\bibitem{Ionescu2007}
A.~D. Ionescu and C.~E. Kenig.
\newblock Global well-posedness of the {B}enjamin-{O}no equation in
  low-regularity spaces.
\newblock {\em J. Amer. Math. Soc.}, 20(3):753--798, 2007.

\bibitem{Iorio1986}
R.~J. I\'{o}rio, Jr.
\newblock On the {C}auchy problem for the {B}enjamin-{O}no equation.
\newblock {\em Comm. Partial Differential Equations}, 11(10):1031--1081, 1986.

\bibitem{MR2267286}
T.~Kappeler and P.~Topalov.
\newblock Global wellposedness of {K}d{V} in {$H^{-1}(\mathbb T,\mathbb R)$}.
\newblock {\em Duke Math. J.}, 135(2):327--360, 2006.

\bibitem{Kaup1998a}
D.~J. Kaup, T.~I. Lakoba, and Y.~Matsuno.
\newblock Complete integrability of the {B}enjamin-{O}no equation by means of
  action-angle variables.
\newblock {\em Phys. Lett. A}, 238(2--3):123--133, 1998.

\bibitem{Kaup1998}
D.~J. Kaup and Y.~Matsuno.
\newblock The inverse scattering transform for the {B}enjamin-{O}no equation.
\newblock {\em Stud. Appl. Math.}, 101(1):73--98, 1998.

\bibitem{Kenig2003}
C.~E. Kenig and K.~D. Koenig.
\newblock On the local well-posedness of the {B}enjamin-{O}no and modified
  {B}enjamin-{O}no equations.
\newblock {\em Math. Res. Lett.}, 10(5-6):879--895, 2003.

\bibitem{MR2310217}
R.~Killip.
\newblock Spectral theory via sum rules.
\newblock In {\em Spectral theory and mathematical physics: a {F}estschrift in
  honor of {B}arry {S}imon's 60th birthday}, volume~76 of {\em Proc. Sympos.
  Pure Math.}, pages 907--930. Amer. Math. Soc., Providence, RI, 2007.

\bibitem{MR4145790}
R.~Killip, J.~Murphy, and M.~Visan.
\newblock Invariance of white noise for {K}d{V} on the line.
\newblock {\em Invent. Math.}, 222(1):203--282, 2020.

\bibitem{killip2021wellposedness}
R.~Killip, M.~Ntekoume, and M.~Visan.
\newblock On the well-posedness problem for the derivative nonlinear
  {S}chr\"odinger equation.
\newblock Preprint \texttt{arXiv:2101.12274}, 2021.

\bibitem{Killip2019}
R.~Killip and M.~Vi\c{s}an.
\newblock Kd{V} is well-posed in {$H^{-1}$}.
\newblock {\em Ann. of Math. (2)}, 190(1):249--305, 2019.

\bibitem{MR3820439}
R.~Killip, M.~Vi\c{s}an, and X.~Zhang.
\newblock Low regularity conservation laws for integrable {PDE}.
\newblock {\em Geom. Funct. Anal.}, 28(4):1062--1090, 2018.

\bibitem{MR4400881}
C.~Klein and J.-C. Saut.
\newblock {\em Nonlinear dispersive equations---inverse scattering and {PDE}
  methods}, volume 209 of {\em Applied Mathematical Sciences}.
\newblock Springer, Cham, 2021.

\bibitem{Koch2003}
H.~Koch and N.~Tzvetkov.
\newblock On the local well-posedness of the {B}enjamin-{O}no equation in
  {$H^s({\mathbb R})$}.
\newblock {\em Int. Math. Res. Not.}, (26):1449--1464, 2003.

\bibitem{Koch2005}
H.~Koch and N.~Tzvetkov.
\newblock Nonlinear wave interactions for the {B}enjamin-{O}no equation.
\newblock {\em Int. Math. Res. Not.}, (30):1833--1847, 2005.

\bibitem{MR4356987}
T.~Laurens.
\newblock Kd{V} on an incoming tide.
\newblock {\em Nonlinearity}, 35(1):343--387, 2022.

\bibitem{MR4541923}
T.~Laurens.
\newblock Global well-posedness for {$H^{-1}(\mathbb R)$} perturbations of
  {K}d{V} with exotic spatial asymptotics.
\newblock {\em Comm. Math. Phys.}, 397(3):1387--1439, 2023.

\bibitem{MR784923}
Y.~Matsuno.
\newblock Note on the {B}\"{a}cklund transformation of the {B}enjamin-{O}no
  equation.
\newblock {\em J. Phys. Soc. Japan}, 54(1):45--50, 1985.

\bibitem{Miller2012}
P.~D. Miller and Z.~Xu.
\newblock The {B}enjamin-{O}no hierarchy with asymptotically reflectionless
  initial data in the zero-dispersion limit.
\newblock {\em Commun. Math. Sci.}, 10(1):117--130, 2012.

\bibitem{MR252826}
R.~M. Miura, C.~S. Gardner, and M.~D. Kruskal.
\newblock Korteweg-de {V}ries equation and generalizations. {II}. {E}xistence
  of conservation laws and constants of motion.
\newblock {\em J. Mathematical Phys.}, 9:1204--1209, 1968.

\bibitem{Molinet2007}
L.~Molinet.
\newblock Global well-posedness in the energy space for the {B}enjamin-{O}no
  equation on the circle.
\newblock {\em Math. Ann.}, 337(2):353--383, 2007.

\bibitem{Molinet2008}
L.~Molinet.
\newblock Global well-posedness in {$L^2$} for the periodic {B}enjamin-{O}no
  equation.
\newblock {\em Amer. J. Math.}, 130(3):635--683, 2008.

\bibitem{Molinet2012}
L.~Molinet and D.~Pilod.
\newblock The {C}auchy problem for the {B}enjamin-{O}no equation in {$L^2$}
  revisited.
\newblock {\em Anal. PDE}, 5(2):365--395, 2012.

\bibitem{Molinet2009}
L.~Molinet and F.~Ribaud.
\newblock Well-posedness in {$H^1$} for generalized {B}enjamin-{O}no equations
  on the circle.
\newblock {\em Discrete Contin. Dyn. Syst.}, 23(4):1295--1311, 2009.

\bibitem{Molinet2001}
L.~Molinet, J.~C. Saut, and N.~Tzvetkov.
\newblock Ill-posedness issues for the {B}enjamin-{O}no and related equations.
\newblock {\em SIAM J. Math. Anal.}, 33(4):982--988, 2001.

\bibitem{MR4148823}
A.~Moll.
\newblock Finite gap conditions and small dispersion asymptotics for the
  classical periodic {B}enjamin-{O}no equation.
\newblock {\em Quart. Appl. Math.}, 78(4):671--702, 2020.

\bibitem{Nakamura1979}
A.~Nakamura.
\newblock B\"{a}cklund transform and conservation laws of the {B}enjamin-{O}no
  equation.
\newblock {\em J. Phys. Soc. Japan}, 47(4):1335--1340, 1979.

\bibitem{MR4520307}
M.~Ntekoume.
\newblock Symplectic nonsqueezing for the {K}d{V} flow on the line.
\newblock {\em Pure Appl. Anal.}, 4(3):401--448, 2022.

\bibitem{Ono1975}
H.~Ono.
\newblock Algebraic solitary waves in stratified fluids.
\newblock {\em J. Phys. Soc. Japan}, 39(4):1082--1091, 1975.

\bibitem{Ponce1991}
G.~Ponce.
\newblock On the global well-posedness of the {B}enjamin-{O}no equation.
\newblock {\em Differential Integral Equations}, 4(3):527--542, 1991.

\bibitem{MR0493420}
M.~Reed and B.~Simon.
\newblock {\em Methods of modern mathematical physics. {II}. {F}ourier
  analysis, self-adjointness}.
\newblock Academic Press [Harcourt Brace Jovanovich, Publishers], New
  York-London, 1975.

\bibitem{MR0493421}
M.~Reed and B.~Simon.
\newblock {\em Methods of modern mathematical physics. {IV}. {A}nalysis of
  operators}.
\newblock Academic Press [Harcourt Brace Jovanovich, Publishers], New
  York-London, 1978.

\bibitem{Saut1979}
J.-C. Saut.
\newblock Sur quelques g\'{e}n\'{e}ralisations de l'\'{e}quation de
  {K}orteweg-de {V}ries.
\newblock {\em J. Math. Pures Appl. (9)}, 58(1):21--61, 1979.

\bibitem{MR0215084}
R.~S. Strichartz.
\newblock Multipliers on fractional {S}obolev spaces.
\newblock {\em J. Math. Mech.}, 16:1031--1060, 1967.

\bibitem{Sun2021}
R.~Sun.
\newblock Complete integrability of the {B}enjamin-{O}no equation on the
  multi-soliton manifolds.
\newblock {\em Comm. Math. Phys.}, 383(2):1051--1092, 2021.

\bibitem{Talbut2021}
B.~Talbut.
\newblock {\em Benjamin-{O}no at {L}ow {R}egularity: {A}n {I}ntegrability
  {A}pproach}.
\newblock PhD thesis, 2021.
\newblock Thesis (Ph.D.)--University of California, Los Angeles.

\bibitem{Talbut2021a}
B.~Talbut.
\newblock Low regularity conservation laws for the {B}enjamin-{O}no equation.
\newblock {\em Math. Res. Lett.}, 28(3):889--905, 2021.

\bibitem{Tao2004}
T.~Tao.
\newblock Global well-posedness of the {B}enjamin-{O}no equation in {$H^1({\bf
  R})$}.
\newblock {\em J. Hyperbolic Differ. Equ.}, 1(1):27--49, 2004.

\bibitem{MR2233925}
T.~Tao.
\newblock {\em Nonlinear dispersive equations}, volume 106 of {\em CBMS
  Regional Conference Series in Mathematics}.
\newblock Published for the Conference Board of the Mathematical Sciences,
  Washington, DC; by the American Mathematical Society, Providence, RI, 2006.
\newblock Local and global analysis.

\bibitem{MR700302}
M.~Wadati and K.~Sogo.
\newblock Gauge transformations in soliton theory.
\newblock {\em J. Phys. Soc. Japan}, 52(2):394--398, 1983.

\bibitem{MR3484397}
Y.~Wu.
\newblock Simplicity and finiteness of discrete spectrum of the
  {B}enjamin-{O}no scattering operator.
\newblock {\em SIAM J. Math. Anal.}, 48(2):1348--1367, 2016.

\bibitem{Wu2017}
Y.~Wu.
\newblock Jost solutions and the direct scattering problem of the
  {B}enjamin-{O}no equation.
\newblock {\em SIAM J. Math. Anal.}, 49(6):5158--5206, 2017.

\end{thebibliography}

\end{document}